\documentclass[11pt]{amsart}
\newcommand{\ov}[1]{\overline{#1}}

\usepackage[utf8]{inputenc}

\usepackage{amssymb, graphicx}

\usepackage{amsfonts}
\usepackage{amsmath}
\usepackage{latexsym}
\usepackage{amscd}
\usepackage{verbatim}

\usepackage{exercise}

\allowdisplaybreaks[1]

\newcommand{\Ad}{\mathrm{ad}\,}
\newcommand\RR{\mathbb{R}}

\newcommand\CC{\mathbb{C}}
\newcommand\ZZ{\mathbb{Z}}
\newcommand\QQ{\mathbb{Q}}

\newtheorem{theorem}{Theorem}[section]

\newtheorem{lemma}[theorem]{Lemma}
\newtheorem{corollary}[theorem]{Corollary}
 
\theoremstyle{remark}
\newtheorem{remark}[theorem]{Remark}
\newtheorem{definition}[theorem]{Definition}  

\newtheorem{XxmpX}[theorem]{Example} 
\newenvironment{example}    
{%
   \pushQED{\qed}\begin{XxmpX}}
  {\popQED\end{XxmpX}}

\numberwithin{equation}{section}

\begin{document}
\title[Functional identities and their applications]{Functional identities and their applications}

\author{Matej Bre\v sar}

\address{Faculty of Mathematics and Physics,  University of Ljubljana, Slovenia}
\address{Faculty of Natural Sciences and Mathematics, University 
of Maribor, Slovenia}
\address{ Institute of Mathematics, Physics, and Mechanics, Ljubljana, Slovenia} \email{matej.bresar@fmf.uni-lj.si}
\begin{abstract}
The paper surveys the
 theory of functional identities and its applications. 
No prior knowledge of the theory  is required to follow the paper.
	\end{abstract}

\keywords{Functional identity, standard solution, $d$-free set, commuting function, biderivation, quasi-polynomial,  polynomial identity, generalized polynomial identity, generalized functional identity, Cayley-Hamilton polynomial, prime ring,  simple ring,  rings of quotients,
symmetric fractional degree, ring with involution, Lie ideal,  Lie homomorphism, Lie derivation,  group grading, Poisson algebra, Lie superhomomorphism,  Lie-admissible algebra, Jordan homomorphism, Jordan $*$-derivation, $f$-homomorphism, near-derivation, linear preserver problems,  zero product determined algebra.}
\thanks{2020 {\em Math. SubJ.\ Class.} Primary: 16R60. Secondary: 15A86, 16N60, 16R20, 16R50, 16S50, 16S85, 16W10, 16W20, 16W25 16W50,  17B40, 17B60, 17C50.}

\thanks{Supported by ARRS Grant P1-0288.}

\maketitle 
\tableofcontents

\section{Introduction}

A functional identity (FI) can be roughly defined as an identical relation on a subset of a ring which involves arbitrary  functions  that are considered as unknowns.
Although the concept of an FI  extends the concept  of a polynomial identity (PI), the FI theory can be viewed as a complement, rather than a generalization  of the PI theory.
It was initiated in the author's PhD Thesis in 1990 and has been   constantly developed since then.   The main motivation for this development has been its applicability to a variety of 
problems in different areas of mathematics. The term “functional identities”  appeared as a new item (16R60) in the 2010 Mathematics Subject Classification.

The fundamentals of the FI theory  are presented in the 2007 book \cite{FIbook} by the  author, M. A. Chebotar, and W. S. Martindale 3rd. Our aim in this  paper is to give a survey of the most important results obtained both before and after the publication of the book. Since enough time has passed, it is now  clear which of the old results should not  be missed  in the exposition. The selected old results along with the more recent results that open new areas of investigation
will  hopefully give a complete picture of the theory.

With a few illustrative exceptions, detailed proofs will be omitted. We will, however,  sketch the main ideas of several proofs.

The paper is organized as follows. In Section \ref{s2}, we will present various concrete examples of FIs that will help the reader to grasp the general theory. The latter will be considered in Section 3. Therein, we will first present the theory of $d$-free sets, introduced by K. I. Beidar and M. A. Chebotar, and then proceed to the more recent topic of nonstandard solutions of FIs in low dimensional algebras.  Section 4 is devoted to applications. We will start with the oldest, i.e., solutions to Herstein's Lie map conjectures, and end with the newest, i.e., involving FIs in the study of zero product determined algebras.

\section{Examples of FIs} \label{s2}

By a ring we will mean an associative ring, not necessarily with unity and not necessarily 
commutative. The FI theory is actually rather empty for commutative rings, so we are in fact interested in rings that are not commutative.
We will sometimes assume that our ring has a unity, in which case we will refer to it as a unital ring. 

Throughout this section we assume that $R$ is a nonzero subring of a unital ring $Q$. The reason for the latter notation is that, in practice,
$Q$ is often a  ring of quotients of $R$. Until further notice, however, $Q$ can be any unital ring. General FIs consider the situation where $R$ is just a subset, not necessarily a subring of $Q$. The assumption that  $R$ is a subring is added  to make examples in this section  simpler.  

\begin{example} Let $F:R\to Q$ be an arbitrary function. An extremely simple example of an FI is
\begin{equation} \label{e1}
xF(y)=0
\end{equation}
for all $x,y\in R$. 
What is the form of $F$ if \eqref{e1} holds? Under mild assumptions, in particular if  $R$ is a unital subring (by this we mean that it contains the unity of $Q$),  $F$ must be zero, as we see by simply taking $x=1$. We therefore call $F=0$ the {\em standard solution} of the FI \eqref{e1}. Every element $q$ in the range of $F$  satisfies $Rq=\{0\}$. In fact, the existence of a nonstandard (i.e., nonzero) solution is  equivalent to the condition that $Q$ contains 
a nonzero element $q$ with this property.  Thus, either our FI has only the standard solution or  $R$ has a certain special property.  This is a typical conclusion when studying an FI, which is the reason why we started with such a trivial example.
\end{example} 

\begin{example}\label{ex2}
Let $F_1,F_2:R\to Q$ be  functions. Consider the FI
\begin{equation} \label{e2}xF_1(y) + yF_2(x)=0
\end{equation}
for all $x,y\in R$. This is obviously a generalization of  \eqref{e1}. It is still a very simple FI, but, unlike \eqref{e1}, not exactly trivial. For reasons that will become clear soon, we define  the standard solution of this FI to be $$F_1=F_2=0.$$ Observe that  nonstandard solutions exist whenever $R$ is commutative; indeed, just take, for example, $$F_1(y)=y\quad\mbox{and}\quad
F_2(x)=-x.$$ Assume now that $R$ is not commutative and let us show that, under reasonable additional assumptions, the only solution is the standard one. 
Take any $x,y,z,w\in R$ and note that \eqref{e2} implies that
\begin{align*}x(yzF_1(w))&= - xwF_2(yz)\\ &= y(zF_1(xw)) \\ &= - yx(w F_2(z))\\ &=yxzF_1(w).
\end{align*}
As usual, we write $[x,y]$ for the {\em commutator} (also called the {\em Lie product}) of the ring elements $x$ and $y$, that is,
$$[x,y]=xy-yx.$$
Observe that the above identity can be written as
$$[x,y]z F_1(w) =0$$
for all $x,y,z,w\in R$. This means that  every element $q$ in the range of $F_1$ satisfies
$Kq=\{0\}$ where $K$ is the commutator ideal of $R$ (i.e., the ideal generated by all commutators in $R$). Thus, if, for example, $Q$ is  a noncommutative domain or $R$ is a noncommutative simple unital ring,
then $F_1=0$ (and similarly, $F_2=0$).
\end{example}

Let $F$ be a field. By $F\langle X_1, X_2,\dots\rangle$ we denote the free algebra on indeterminates $X_i$, i.e., the algebra of noncommutative polynomials in $X_i$. Let $R$ be an algebra over $F$. A noncommutative polynomial
$$f(X_1,\dots,X_n)\in F\langle X_1, X_2,\dots\rangle$$ is said to be a {\em polynomial identity} of $R$ if $$f(a_1,\dots,a_n)=0$$ for all $a_i\in R$. If some nonzero noncommutative polynomial is a polynomial identity of 
$R$, then we say that $R$ is a PI-algebra.  These algebras will play an important role in this paper, and facts 
from the PI theory
will be frequently used.
We refer the reader to the books \cite{AGPR, GZ, Rowen} for a thorough treatment, and to \cite[Sections 6 and 7]{INCA} for an elementary introduction.

\begin{example}\label{ex3} Assume that 
 $F_1,F_2:R\to Q$ satisfy \begin{equation} \label{e3}[xF_1(y) + yF_2(x),z]=0
\end{equation}
for all $x,y,z\in R$. That is, $$xF_1(y)+ yF_2(x) \in C_Q(R)$$
for all $x,y\in R$, where
$C_Q(R)$
 stands for the {\em centralizer} of $R$ in $Q$, i.e.,
 $$C_Q(R)=\{c\in Q\,|\, cz=zc\,\,\,\mbox{for every $z\in R$}\}. $$
  In the basic case where $R=Q$ this means that $xF_1(y)+ yF_2(x)$ belongs to the center of $Q$. Centers of standard noncommutative rings are usually  relatively small subsets.  It is therefore natural to expect that  \eqref{e3} can be handled in a similar fashion as  \eqref{e2}, which is obviously its special case. However, this is true only to some extent. We have noticed above that 
 if $R$ is a noncommutative simple unital ring, then 
 \eqref{e2} has only standard  solutions. Let us show that this does not hold for 
 \eqref{e3}. 
 
We, of course, define that the standard solution of \eqref{e3} is  $F_1=F_2=0$. 
Consider the $2\times 2$ matrix ring $R=M_2(Z)$ where
 $Z$ is a field.  The Cayley-Hamilton Theorem states  that, for every $x\in R$,
 $x^2 - {\rm tr}(x)x $ is the scalar matrix $-\det(x)$, and hence
\begin{equation}
    \label{linthis}
[xF(x),z]=0\end{equation}
 for all $x,z\in R$, where $F:R\to R(=Q)$ is defined by
 $$F(x)=x-{\rm tr}(x)$$
 (here, we have identified the scalar matrix $\lambda 1$, where $1$ is the identity matrix, by the scalar $\lambda$).
 Linearizing \eqref{linthis} (i.e., considering the identity obtained by replacing $x$ by $x+y$ in \eqref{linthis}), we  arrive at
 $$[xF(y) + yF(x),z]=0$$
 for all $x,y,z\in R$. Thus, we have arrived at \eqref{e3} with $$F_1=F_2=F\ne 0.$$ This means that
 \eqref{e3} has a nonstandard solution in the noncommutative simple unital ring $R=M_2(Z)$.
 
 Our goal now is to show  that $M_2(Z)$, along with some related rings,  is in fact the only noncommutative simple unital ring in which \eqref{e3} has nonstandard solutions. Thus, assume from now on
that $R$ is a simple unital ring which is not a field (equivalently, it is not commutative).

 We start by replacing $y$ by $zy$ in \eqref{e3}, which gives
 $$[xF_1(zy),z] + [zyF_2(x),z]=0.$$
 Since
 $$[zyF_2(x),z] = z[yF_2(x),z]
 = - z[xF_1(y),z],
 $$
 it follows that
 $$[xF_1(zy),z] - z[xF_1(y),z]=0.
 $$
 We can rewrite this as
\begin{equation}
    \label{lddd}
xF_1(zy)z -z x\big(F_1(zy)+F_1(y)z \big)   + z^2xF_1(y)=0.
 \end{equation}
 Assume that there exists 
 a $z\in R$ such that the elements $1,z,z^2$ are linearly independent over the center $Z$ of $R$ (which is a field since $R$ is simple). Then  we can use the Artin-Whaples Theorem
 (see \cite[Corollary 5.24]{INCA}) to obtain elements $a_i,b_i\in R$ such that 
 $$\sum_i a_ib_i = 0,\quad \sum_i a_izb_i = 0,\quad \sum_i a_iz^2b_i =1.$$
 Substituting $b_i$ for $x$ 
in \eqref{lddd}, and then multiplying by $a_i$ from the left,  we obtain by summing up
$$\sum_i a_ib_iF_1(zy)z -\sum_ia_iz b_i\big(F_1(y)z+ F_1(zy)\big)   + \sum_i a_iz^2b_iF_1(y)=0,
 $$
 which, by the choice of $a_i,b_i$, reduces to $$F_1(y)=0$$
 for every $y\in R$. Similarly we see that $F_2(x)=0$ for every $x\in R$. Thus, 
 the FI \eqref{e3} has only standard solutions in the case
under consideration. 

We may therefore assume that $1,z,z^2$ are linearly dependent over $Z$ for every $z\in R$. Then $R$ is a PI-algebra since, for example,
$$\big[[X_1^2,X_2],[X_1,X_2]\big]$$ is a polynomial identity of $R$. It is a standard fact from the PI theory that a simple
unital PI-algebra is finite-dimensional over its center $Z$. However, more can be said. Since 
our algebra satisfies a polynomial identity of degree $5$ and is not commutative, the only possibility is that $[R:Z]$, the dimension of $R$ over $Z$, is $4$. The 
 classical Wedderburn's theorem
 then tells us that either $R\cong M_2(Z)$ or $R$
 is a $4$-dimensional division algebra. In the latter case, we can find nonstandard solutions of \eqref{e3} in much the same way as in the $2\times 2$ matrix case (see \cite[Corollary C.3]{FIbook}). 
 
 We can  summarize our findings as follows. If $R$ is a simple unital ring, then the FI \eqref{e3} has nonstandard solutions if and only if $R$ is either of dimension $1$ or of dimension $4$ over its center,
so, more specifically, $R$ satisfies one of following three conditions: (a) $R$ is commutative, (b) $R\cong M_2(Z)$ where $Z$ is a field, (c) $R$ is a $4$-dimensional division algebra over its center.
 \end{example}

A simple unital ring $R$ with center $Z$ is either infinite-dimensional over  $Z$ or $[R:Z]=n^2$ for some positive integer $n$. 
We saw that the FI  \eqref{e3} (resp.\ \eqref{e2}) gives rise to a characterization of simple rings of dimension 
at most $4$ (resp.\ $1$) over their centers among all simple unital rings (not necessarily finite-dimensional over their centers). How to characterize simple unital rings of higher finite dimensions over their centers?

\begin{example} \label{ex4}To answer the question just posed, it is natural to consider the $n\times n$ matrix algebra
$R=M_n(Z)$ with $Z$ a field. Motivated by the preceding example, we invoke the Cayley-Hamilton Theorem which
in particular states that there is a function $\phi:R\to R$ such that
\begin{equation}
\label{chid}    
x\phi(x) = \det(x)
\end{equation}
for every 
$x\in R$ (as above, $\det(x)$ stands for $\det(x)1$). 
We  call \eqref{chid} the {\em Cayley-Hamilton identity}. 
Of course, $$\phi(x) = {\rm adj}(x),$$ the adjugate matrix of $x$. 
Its form is well known, i.e., it
 can be expressed through the powers of $x$ and the trace. For our purposes, however, it is enough to observe that $\phi$ is the {\em trace of an $(n-1)$-additive function}. By this we mean that  it  can be presented as 
$$\phi(x)= \Phi(x,\dots,x)$$
where $\Phi:R^{n-1}\to R$ is a function that is additive in each variable. Since $x\phi(x)$ is always a scalar matrix, we have
$$ [x\Phi(x,\dots,x),z]=0$$
for all $x,z\in R$. Linearizing this identity, we see that
$F:R^{n-1}\to R$, defined by
$$F(x_1,\dots, x_{n-1})=\sum_{\sigma \in S_{n-1}} \Phi(x_{\sigma(1)},\dots, x_{\sigma(n-1)}),$$
satisfies
$$ [x_1F(x_2,\dots,x_n) + x_2F(x_1,x_3,\dots,x_n) + \dots + x_n F(x_1,\dots,x_{n-1}),z] =0 $$
for all $x_1,\dots,x_n,z\in R$.
What is important for us here is that  the ring $R=M_n(Z)$
satisfies this FI for some  nonzero function $F$.  

Consider now a more general FI
\begin{equation}\label{e4} [x_1F_1(x_2,\dots,x_n) + x_2F_2(x_1,x_3,\dots,x_n) + \dots + x_n F_n(x_1,\dots,x_{n-1}),z] =0
\end{equation}
for all $x_1,\dots,x_n,z\in R$,
where $F_1,\dots,F_n:R^{n-1}\to Q$ are any functions. Let us also write \eqref{e4} in an equivalent, but  more readable form: 
$$x_1F_1(x_2,\dots,x_n) + x_2F_2(x_1,x_3,\dots,x_n) + \dots + x_n F_n(x_1,\dots,x_{n-1})\in C_Q(R)$$
for all $x_1,\dots,x_n\in R$. We define  the standard solution of \eqref{e4} to be
$$F_1=\dots=F_n=0.$$ 
We have shown  that nonstandard solutions exist if $R=Q=M_n(Z)$. Now, it can be proved that if $R=Q$ is any simple unital ring with center $Z$, then nonstandard solutions exist if and only if $[R:Z]\le n^2$. The proof for a general $n$ is of course technically more involved than the one
for $n=2$ that was given 
in the preceding example. The main ideas, however, are the same. For now we omit details, but in the next section we will consider \eqref{e4} in a more general framework.\end{example}

Standard solutions of all of the above FIs were just zero functions. We continue with an example of a very simple  FI for which this is not the case.

\begin{example}\label{ex5}
Let $E,F:R\to Q$ satisfy the FI
\begin{equation}\label{e5} E(x)y=xF(y)
\end{equation}
for all $x,y\in R$. An obvious
possibility for \eqref{e5} to hold is that there exists a 
$q\in Q$ such that
$$E(x)=xq\quad\mbox{and}\quad F(y)=qy $$
for all $x,y\in R$. If $E$ and $F$ are of such a form then we say that  they present a standard solution of \eqref{e5}. If $R$ is unital, then nonstandard solutions cannot exist as we see by first setting $y=1$ and then $x=1$ in \eqref{e5}. The nonunital case can be different. For example, 
if $R=2\ZZ$, $Q=\ZZ$ and 
$E, F$
 are given by
 $$E(x)=\frac{1}{2} x \quad\mbox{and}\quad
 F(y)=\frac{1}{2} y,$$
 then they form a nonstandard solution of \eqref{e5}. However, if we take $Q=\QQ$ instead of $Q=\ZZ$, then all solutions are standard. Indeed, we define $q=\frac{F(2)}{2}$.
\end{example}

It may now be a little bit clearer what are standard solutions of FIs. There is no general definition that would cover all possible cases, but one may think of them as the ``obvious" solutions which make sense in any ring, regardless of its special features. 

Example \ref{ex5} suggests that rings of quotients may be useful in  proving that all solutions are standard. Indeed,
they occur throughout the FI theory. We  will  now introduce  one of them, and will deal with some  others later. As we will see, the definition is closely related to the FI \eqref{e5}.


Recall first  that $R$ is a {\em prime ring} if the product of any of its two nonzero ideals is nonzero. Equivalently, for any
$a,b\in R$, $aRb=\{0\}$ implies  $a=0$ or $b=0$. Prime rings form a fairly large class of rings, which in particular includes simple rings, primitive rings, and domains. The FI theory is not limited to prime rings, but they present the most natural setting for its study.  It should be mentioned that in Examples \ref{ex3} and \ref{ex4} we have restricted ourselves to simple unital rings only for simplicity. Dealing with considerably more general prime rings  would not make the discussion much longer, but we would have to refer to somewhat less elementary results. In particular,   involving 
some  rings of quotients when studying FIs in prime rings is  unavoidable. The one from the next definition  actually is not the most appropriate one for developing the general theory, but is somewhat simpler than others and suitable for our current purposes.

We assume that $R$ is a prime ring.

\begin{definition}\label{def1}
 A ring $Q=Q_s(R)$ is called a {\em symmetric Martindale ring of quotients} of $R$  if it satisfies the  following conditions:
\begin{enumerate}
\item[(a)] $R$ is a subring of $Q$. 
\item[(b)] For every  $q\in Q$ there exists a nonzero ideal $I$ of $R$ such that $Iq + qI \subseteq R$.
\item[(c)]  For every nonzero $q\in Q$ and every nonzero ideal $J$ of $R$,  $Jq\ne \{0\}$ and $qJ\ne \{0\}$.
\item[(d)] If $I, J$ are nonzero ideals of $R$ and functions $E:I\to R$, $F:J\to R$ satisfy $E(x)y=xF(y)$ (i.e.,
\eqref{e5}) for all $x\in I$ and all $y\in J$,
 then there exists  a $q\in Q$ such that $E(x)=xq$ for all $x\in I$ and $F(y)=qy$ for all $y\in J$.
\end{enumerate}
\end{definition}

\begin{remark}
\label{rinthel}
The functions $E$ and $F$ from condition (d) are usually assumed to be left and right $R$-module homomorphisms, respectively.
  We have avoided this unnecessary assumption in order   to point out  the connection with FIs which involve arbitrary functions. 
  The proof of the unnecessity
 is easy.  Indeed, taking $r\in R$, $x\in I$ and
 replacing $x$ with $rx$ in
  $E(x)y=xF(y)$ it follows that
$$E(rx)y= rxF(y) = rE(x)y,$$
and so $$\bigl(E(rx) - rE(x)\bigr)y=0,$$ which, by (c),
implies that $E(rx) = rE(x)$. Similarly we see that $E$ is additive, so it is a left $R$-module homomorphism. The proof that $F$ is a right $R$-module homomorphism is analogous.

\end{remark} 
 
 For properties and examples of $Q_s(R)$  we refer the reader to the books \cite{BMMb, Lam, Pass}. The basic property is that 
 $Q_s(R)$ exists for every prime ring $R$ and is unique up to isomorphism. 
 We  therefore usually speak about {\em the} symmetric Martindale ring of quotients.
 
 The center of $Q_s(R)$ is called the {\em extended centroid} of $R$ and is traditionally denoted by $C$. 
 The extended centroid is a field that contains the center of $R$, but is not always its field of fractions. Among its many nice properties, we point out the following one: If $a,b$ are elements in $R$ that satisfy \begin{equation}
   atb=bta \label{atb}  
 \end{equation}
 for every $t\in R$, then they are linearly dependent over $C$.
 
 Before proceeding to the next example, we recall that an additive map $d$ from $R$
 to  $R$ (or, more generally, from $R$ to an $R$-bimodule) is called a {\em derivation}
 if 
 $$d(xy) =d(x)y+xd(y)$$ for all $x,y\in R$. For example, for any element $a$,
 $d(x)=[a,x]$ is a derivation. Such a derivation is called {\em inner}.
 
 \begin{example}
 \label{ex6}  A function $F:R\to R $ is said to be {\em commuting} if
 \begin{equation}
 \label{comm} F(x)x=xF(x)
 \end{equation}
 for every $x\in R$. 
 These functions are of great importance in the FI theory.

 Without imposing some additional condition, not much can be said about the function $F$ satisfying \eqref{comm}. A long time ago,    functions such as automorphisms or   derivations  were shown to be commuting only in some trivial situations \cite{Div, Pos}.
 These functions, however, are too special from the point of view of the FI theory. 
 
 Let us only assume that 
 $F$ is additive. Linearizing
 \eqref{comm}
 we obtain
  \begin{equation}
 \label{comm2} F(x)y + F(y)x=xF(y) + yF(x)
 \end{equation}
 for all $x,y\in R$. This is actually the only time that the  assumption that $F$ is additive is needed, that is, from now on $F$ can be just any function from $R$ to $R$ that satisfies \eqref{comm2}. We can therefore view \eqref{comm2} as an FI.
 
We approach \eqref{comm2} 
by defining 
the function $D:R\times R\to R$
by
$$D(x,y)=[F(x),y].$$
Observe that $D$ is an (inner) derivation in the second variable. On the other hand, 
since, by \eqref{comm2}, 
 $$D(x,y) = [x,F(y)],$$
 $D$ is a derivation also in the first variable. Until further notice, we consider
 any function $D:R\times R\to R$ that is a derivation in each variable. We call such a function  a {\em biderivation}.
 
 In order to describe the form of a biderivation $D$, we take any $x,y,z,w\in R$ and
  consider $D(xz,yw)$. Since $D$ is a derivation in the first variable, we have
  $$D(xz,yw)=D(x,yw)z+ xD(z,yw).$$
  Using that $D$ is a derivation in the second variable, it follows that
  $$D(xz,yw)=D(x,y)wz+ yD(x,w)z + xD(z,y)w+xyD(z,w).$$
  On the other hand, if we reverse the order and first  use the condition that $D$ is a derivation in the second variable, we obtain
  \begin{align*}
  D(xz,yw)&= D(xz,y)w+ 
      yD(xz,w)\\ &= D(x,y)zw+xD(z,y)w + yD(x,w)z + yxD(z,w).
  \end{align*}
 Comparing both expressions we obtain
 $$
 D(x,y)[z,w] = [x,y]D(z,w)
 $$
for all $x,y,z,w\in R$. 
 Replacing $z$ by $tz$ and using $[tz,w]= [t,w]z + t[z,w]$ and 
 $D(tz,w)= D(t,w)z+tD(z,w)$ it follows that 
\begin{equation}\label{bid}
 D(x,y)t[z,w] = [x,y]tD(z,w)
 \end{equation}
 for all $x,y,z,w,t\in R$. 
 
 We are now in a position to use the aforementioned result that concerns the identity \eqref{atb} and  the extended centroid. Assume, therefore, that
 $R$ is a prime ring, and, moreover, that it is not commutative. We remark that the commutative case 
 is indeed different; for example, if
 $R=\RR[X]$, then $$D(f,g)=f'g'$$ is an example of a biderivation that is not of the form that we are about to derive under the assumption that $R$ is noncommutative. The study of biderivations in commutative rings has different goals, see \cite{LM}.
 
 Take any $z,w\in R$ such that $[z,w]\ne 0$. Using \eqref{bid} for 
 $x=z$ and $y=w$ it follows that there exists an element $\lambda$ from the extended centroid $C$ of $R$ such that
 $D(z,w)=\lambda[z,w]$. Hence, \eqref{bid} can be written as
 $$\big(D(x,y) - \lambda[x,y]\big)t[z,w]=0$$
 for all $x,y\in R$. That is, 
  $$\big(D(x,y) - \lambda[x,y]\big)J=0,$$ where $J$ is the ideal of $R$ generated by $[z,w]$. Using condition (c) from Definition \ref{def1} it follows that
  \begin{equation} \label{d} D(x,y) = \lambda[x,y]\end{equation}
  for all $x,y\in R$. Thus, every biderivation of $R$ is simply the Lie product  of the variables followed by multiplication by an element from the extended centroid. 
  
  We can now return to the FI \eqref{comm2}. Since $(x,y)\mapsto[F(x),y]$ is a biderivation, it follows that there exists a $\lambda\in C$ such that
  $$[F(x),y]=\lambda[x,y]$$
  for all $x,y\in R$, that is,
  \begin{equation}\label{fla}
  [F(x)-\lambda x,y]=0.\end{equation}
  Now, if $q$ is an element in
  $Q_s(R)$ that commutes with any element $y$ in $R$, then $q$ lies in $C$, that is, it commutes with any element $p$ in $Q_s(R)$. This is because by (b) (from Definition \ref{def1}) we can choose a nonzero ideal $I$ of $R$ such that $pI\subseteq R$,
  hence $$0=[q,pu] = [q,p]u+p[q,u]=[q,p]u$$ for every $
  u\in I$, and therefore
  $[q,p]=0$ by (c). Consequently,
  \eqref{fla} implies that 
  $F(x)-\lambda x \in C$. Thus,
  $F$ is of the form
  \begin{equation}\label{fla2}
  F(x)=\lambda x  + \mu(x) \end{equation}
  for some $\lambda\in C$ and $\mu:R\to C$. Finally, if $F$ is additive, then obviously $\mu$ is additive too. 
  
  Going back to the very beginning, we can now say  that if $R$ is a  prime ring, then
  every commuting additive function
  $F:R\to R$ is  of the form
  \eqref{fla2} with $\lambda$ an element in $C$ and $\mu$ an additive function from $R$ to $C$. Observe that 
  this is trivially true if $R$ is commutative, so the noncommutativity assumption is no longer necessary.
  
  It is interesting to add here that even if $R$ is unital and its center $Z$ is a field, $\lambda$ and $\mu(x)$
  do not necessarily lie  in  $Z$. Indeed, consider the following example. Let $V$ be an infinite-dimensional complex vector space, let $R_0$
be the algebra of all  $\CC$-linear operators from $V$ to $V$ that have finite rank, and let $R$ be the algebra 
of all operators of the form
$a_0 + t 1$, where $a_0\in R_0$, $t\in \RR$, and $1$ is the identity operator on $V$.
Observe that $R$ is a prime ring
with center $\RR\, (=\RR 1)$. It is not difficult to see that the extended centroid  of $R$ is $\CC$. Define $F:R\to R$ by
$$F(a_0 + t1)= ia_0. $$ Clearly, $F$ is an additive commuting function. Note that
$F$ is of the form \eqref{fla2} with $\lambda =i$ and $  \mu(a_0+ t1) = -it1.$ However, 
$F$ cannot be presented  in the form \eqref{fla2} with $\lambda\in\RR$.
\end{example}
 
 \begin{example}\label{ex8}
 In this final example we again consider a commuting function $F$, but instead of additivity we assume that $F$ is the trace of a biadditive function, i.e., there exists a biadditive function $B: R\times 
  R\to R$ such that
  $F(x)=B(x,x)$ for every $x\in R$. Our assumption thus reads as 
  \begin{equation}\label{b22}
      [B(x,x),x] =0 
  \end{equation} 
  for all $x\in R$.  What is the form of $B(x,x)$? Assume again that $R$ is a prime ring. In light of the previous example, it seems natural to expect that 
\begin{equation}\label{b22a}B(x,x)= \lambda x^2 + \mu(x) x + \nu(x,x)\end{equation}
  for all $x\in R$, where $\lambda$ is an element
  in $C$, $\mu$ is an additive function from $R$ to $C$, and $\nu$
  is a biadditive function from $R\times R$ to $C$. Under a mild assumption that  char$(R)\ne 2$, this turns out to be true. While conjecturing the result was easy in view of the result from the preceding example, the proof is much more involved and will not be given at this point. Let us just mention that the first easy step is linearizing \eqref{b22} to obtain
\begin{equation}\label{b22b}[B(x,y)+B(y,x),z] + [B(z,x) + B(x,z),y] + [B(y,z) + B(z,y),x] =0 \end{equation}
  for all $x,y,z\in R$. This is an FI that fits into a general theory which will be discussed in the next section.
  
  One can proceed and consider a commuting function $F$ that is the trace of an $n$-additive function, i.e., is  of the form $$F(x)=B(x,\dots,x)$$ where $B$ is a $n$-additive function. As one may conjecture, under suitable mild assumptions  it can be shown that $F(x)$ is the sum of expressions of the form
  $$\lambda_i(x,\dots,x)x^i$$
  where $\lambda_i$ is an $(n-i)$-additive function from $R^{n-i}$ to $C$.  
\end{example}

The FI theory actually  originated from the study of commuting functions. The above presented results on commuting additive functions and commuting traces of biadditive functions are contained already in the author's  PhD Thesis from 1990 and were published a few years later in \cite{B3} and \cite{B5}. The approach via biderivations, however, was noticed in the subsequent paper \cite{BMM} (and later, but independently, in \cite{FL}). The original result on commuting traces of biadditive functions required an additional  technical  assumption that was  later removed. For a slightly outdated, but still quite thorough survey on commuting functions see 
\cite{surveycomm}. 
See also \cite{BreSem, Shitov} for results on commuting functions that do not necessarily arise from multiadditive functions, and \cite{Choiet, Fr1, Fr2, LiuYang} for results on traces of multiadditive functions that are commuting on sets that are not closed under addition.  
We also mention that biderivations 
 are nowadays  studied mostly in Lie algebras and often independently of the FI theory (see \cite{BZ, E, KK, LGZ, Tang, WY, YL} and references therein).

\section{The general theory}

In this section, we will survey the most important topics  of the general FI theory. We  divide it into several subsections.

\subsection{The definition of a $d$-free set}

The general theory is based on the concept of a $d$-free set, so we start with its definition. First we introduce some notation.

Let $Q$ be a unital ring with center $C$, and let $R$
be a nonempty subset of $Q$.
The situation from the preceding examples where $R$ is a subring is particularly important, but the general theory considers arbitrary subsets. 

  For any positive integer $k$, we write $R^k$ for the Cartesian product of $k$ copies of $R$, and for convenience we also write $R^0 = \{0\}$.
 Let $m$ be a positive integer, let
 $x_1,\dots, x_m \in R$, and, if $m >1$, let $1\le i < j \le m$. 
We write
\begin{eqnarray*}
\ov{x}_m &=& (x_1,\ldots,x_m)\in{R}^m,\\
\ov{x}_{m}^{i}&=& (x_1,\ldots,x_{i-1},x_{i+1},\ldots,x_m)\in{R}^{m-1},\\
\ov{x}_{m}^{ij}= \ov{x}_{m}^{ji} &=& (x_1,\ldots,x_{i-1},x_{i+1},\ldots,
x_{j-1},x_{j+1}\ldots,x_m)\in{R}^{m-2}.
\end{eqnarray*}

Let $I,J\subseteq \{1,2,\ldots,m\}$. For each $i\in I$ and $j\in J$,
let  
$$E_i:R^{m-1}\to Q\quad\mbox{and}\quad F_j:R^{m-1}\to Q$$
be arbitrary functions.  If $m=1$,  $E_i$ and 
$F_j$ may be regarded as elements in $Q$.

There are two fundamental FIs on which the general theory is based.
The first one is
\begin{equation}
\sum_{i\in I}E_i(\ov{x}_m^i)x_i+ \sum_{j\in J}x_jF_j(\ov{x}_m^j)
 = 
0\label{3S1}\end{equation}
for all $\ov{x}_m\in R^m$,
and the second one is
\begin{equation}
\sum_{i\in I}E_i(\ov{x}_m^i)x_i+ \sum_{j\in J}x_jF_j(\ov{x}_m^j)
 \in  C\label{3S2}
\end{equation}
for all $\ov{x}_m\in R^m$.
Obviously,
 (\ref{3S1})  implies (\ref{3S2}). One should thus
 consider  (\ref{3S1}) and  (\ref{3S2}) as independent FIs, that is, the functions satisfying \eqref{3S1} are not the same as the functions  satisfying \eqref{3S2}. 
   Examples \ref{ex2} and \ref{ex3} may give a clue of why both apparently almost identical FIs deserve to be treated separately.
   
   A natural possibility when  (\ref{3S1}), and hence also (\ref{3S2}), is satisfied  is 
   that there exist functions 
\begin{eqnarray*}
&  & p_{ij}:R^{m-2}\to Q,\;\;
i\in I,\; j\in J,\; i\not=j,\nonumber\\
&  & \lambda_k:R^{m-1}\to C,\;\; k\in I\cup J,
\end{eqnarray*}
such that 
\begin{eqnarray}\label{3S3}
E_i(\ov{x}_m^i) & = & \sum_{j\in J,\atop
j\not=i}x_jp_{ij}(\ov{x}_m^{ij})
+\lambda_i(\ov{x}_m^i),\quad i\in I,\nonumber\\
F_j(\ov{x}_m^j) & = & -\sum_{i\in I,\atop
i\not=j}p_{ij}(\ov{x}_m^{ij})x_i
-\lambda_j(\ov{x}_m^j),\quad j\in J,\\
&   & \lambda_k=0\quad\mbox{if}\quad k\not\in I\cap J\nonumber.
\end{eqnarray}
 Indeed, one immediately checks that  \eqref{3S3}  implies  (\ref{3S1}) (and therefore it also implies   (\ref{3S2})).
We call \eqref{3S3} a {\em standard solution} of the FIs (\ref{3S1}) and (\ref{3S2}).

\begin{remark}\label{rema31}
We  follow the convention that the sum over
$\emptyset$ is $0$. If  $m=1$, it should thus be understood   that $p_{ij} =0$ and 
$\lambda_k$ is an element in $C$.
\end{remark}


\begin{remark}
Observe that if $E_i$ and $F_j$ are multilinear noncommutative polynomials, \eqref{3S1} becomes a multilinear polynomial identity. As the study of general polynomial identities can be often reduced to the multilinear ones,  functional identities can thus be, at least formally,  viewed as generalizations of polynomial identities.
\end{remark}

Let us give two examples that consider special cases of \eqref{3S1} and
\eqref{3S2}, and their standard solutions \eqref{3S3}.  

\begin{example}\label{exx4}
Consider the case where $m=3$,
$I=\{1,2\}$, and $J=\{1,3\}$.  Then \eqref{3S1} reads as 
\begin{equation}\label{eq34} E_1(x_2,x_3)x_1 + E_2(x_1,x_3)x_2 + x_1 F_1(x_2,x_3) + x_3F_3(x_1,x_2) =0\end{equation}
for all
$x_1,x_2,x_3\in R$. Standard solutions of \eqref{eq34} are of the form
\begin{align*}
E_1(x_2,x_3) &= x_3p_{13}(x_2) + \lambda_1(x_2,x_3),\\
E_2(x_1,x_3) &= x_1p_{21}(x_3) + x_3p_{23}(x_1) ,\\
F_1(x_2,x_3) &= -p_{21}(x_3)x_2 - \lambda_1(x_2,x_3) ,\\
F_3(x_1,x_2) &=-p_{13}(x_2)x_1- p_{23}(x_1)x_2,
\end{align*}
where $p_{13},p_{21},p_{23}$ are arbitrary functions from $R$ to $Q$ and $\lambda_1$ is an arbitrary function from $R^2$ to the center $C$.
\end{example}

\begin{example}\label{exx0}If $J=\emptyset$,  (\ref{3S1}) and 
(\ref{3S2})  become
$$
 \sum_{i\in I}E_i(\ov{x}_m^i)x_i =0
\quad\mbox{and}\quad \sum_{i\in I}E_i(\ov{x}_m^i)x_i \in C,
$$
and, according to the convention mentioned in Remark \ref{rema31}, the standard solution (\ref{3S3}) of each of these two FIs is simply
 $E_i(\ov{x}_m^i)=0.$  Similarly, the standard solution of
$$\sum_{j\in J}x_jF_j(\ov{x}_m^j) =0 \quad\mbox{and}\quad \sum_{j\in J}x_jF_j(\ov{x}_m^j) \in C 
$$
 is $F_j(\ov{x}_m^i)=0$ (compare Example \ref{ex4}).
\end{example}

We can now give our fundamental definition, originally introduced in \cite{BC1}.

\begin{definition}\label{defd}Let $d$ be a positive integer. The set $R$ is  a {\em
$d$-free} subset of $Q$
if  
the following two conditions hold for all $m\ge 1$ and all
 $I,J\subseteq \{1,2,\ldots,m\}$:
\begin{enumerate}
\item[(a)]If
$\max\{|I|,|J|\}\le d$, then (\ref{3S1}) implies (\ref{3S3}).
\item[(b)] If
$\max\{|I|,|J|\}\le d-1$, then (\ref{3S2}) implies (\ref{3S3}).
\end{enumerate}
\end{definition}



Thus,
$R$ is a $d$-free subset of $Q$ if the fundamental FI \eqref{3S1} (resp.\ \eqref{3S2}) has only standard solutions provided that it involves  at most $d$ (resp.\ $d-1$) functions $E_i$, as well as at most $d$ (resp.\ $d-1$) functions $F_j$. In the basic case where $I=J=\{1,\dots,m\}$ we can speak about the number of variables instead of the number of functions.

It is usually clear which ring $Q$ we have in mind. In this case we simply say that $R$ is a {\em $d$-free set}. 
The case where  $R=Q$ is of special interest.

\begin{definition}
 A unital ring $R$ is said to be a 
 {\em $d$-free ring}  if $R$ is a $d$-free subset of itself.
\end{definition}

\begin{remark}\label{remdd}It is immediate from the definition that 
``$d$-free" implies ``$d'$-free" for every $d'<d$.
\end{remark}

\begin{remark}
The definition implies that the standard solutions \eqref{3S3} on $d$-free subsets are unique, provided of course that the condition on $\max\{|I|,|J|\}$ is fulfilled. Indeed,
if 
\begin{equation*}
E_i(\ov{x}_m^i) = \sum_{j\in J,\atop
j\not=i}x_jp_{ij}(\ov{x}_m^{ij})
+\lambda_i(\ov{x}_m^i)
\end{equation*}
and also
\begin{equation*}
E_i(\ov{x}_m^i) = \sum_{j\in J,\atop
j\not=i}x_jp_{ij}'(\ov{x}_m^{ij})
+\lambda_i'(\ov{x}_m^i)
\end{equation*}
for some functions
\begin{eqnarray*}
&  & p_{ij},p_{ij}':R^{m-2}\to Q,\;\;
i\in I,\; j\in J,\; i\not=j,\nonumber\\
&  & \lambda_k,\lambda_k':R^{m-1}\to C,\;\; k\in I\cup J
\end{eqnarray*}
such that $\lambda_k=\lambda_k'=0$ if
$k\notin I\cap J$,  then
$$
\sum_{j\in J,\atop
j\not=i}x_j(p_{ij}-p_{ij}')(\ov{x}_m^{ij})
=(\lambda_i'-\lambda_i)(\ov{x}_m^i)\in C.
$$
This is an FI of the type \eqref{3S2} with $I=\emptyset$
and so its standard solution consists of the zero functions.  This easily implies that, under an appropriate assumption on $\max\{|I|,|J|\}$,
$p_{ij}=p_{ij}'$ and
$\lambda_i=\lambda_i'$.
One just has to be a bit careful and consider the case where $i\notin J$ separately. We leave the details to the reader. 
\end{remark}

\begin{remark}
To prove that a set $R$ is $1$-free one
only has to consider condition (a) with
both $I$ and $J$ containing at most one element. For example, if $R$ contains the unity $1$, then it is automatically $1$-free. The problem of proving that $R$ is $2$-free is much more interesting. When speaking about $d$-free sets, we usually have in mind that $d\ge 2$. However, there is no reason to exclude the $d=1$ case when establishing the general theory.
\end{remark}

\begin{remark}\label{remcom}
A commutative ring  is never $2$-free. This is evident from Example \ref{ex2}  (along with Example \ref{exx0}). 
\end{remark}

\begin{remark}
Both conditions, (a) and (b), are truly needed because of applications. They seem very similar, so one may wonder whether they are really independent. This question  actually is not easy to answer. However, it has turned out that they are,  that is, examples showing that neither (a) implies (b) nor (b) implies (a)  were constructed, see \cite{Bremdfree}.
\end{remark}

\subsection{The symmetric fractional degree} 
 So far, we only know that
 commutative rings are never $d$-free for $d\ge 2$ (Remarks \ref{remdd} and \ref{remcom}) and
 that the ring of $n\times n$ matrices over a field is not
 $d$-free for $d > n$ (Examples \ref{ex4} and \ref{exx0}). 
It is not clear at this point  whether $d$-free sets with $d\ge 2$ exist at all. In this subsection, we will not yet give concrete examples, but present  a technical condition under which  a subring is  $d$-free.

We keep the notation from the preceding subsection, but additionally assume that $R$ is a subring. Thus, $Q$ will stand for a unital ring with center  $C$ and $R$ will be its (not necessarily unital) subring.
The definition of the aforementioned condition reads as follows.

\begin{definition}\label{def2}
The {\em symmetric fractional degree} of an element $t\in R$, denoted
$$\mbox{sf-$\deg(t)$,}$$
is said to be greater than a nonnegative integer $n$
(sf-$\deg(t) > n$) if 
there exist $a_{k},b_{k} \in R$ such that
$$\sum_{k}a_{k}t^ib_{k}=0,\,\,\,i=0,1,\dots,n-1,$$
and
$$ a = \sum_{k}a_{k}t^nb_{k}
$$
satisfies the following conditions:
\begin{enumerate}
\item[(SF1)] If $q\in Q$ is such that  $aRq = \{0\}$ or $qRa=\{0\}$, then $q=0$.

\item[(SF2)] If 
$U,V:R\to Q$ are  functions satisfying
$$ U(x)ya= axV(y)$$ for all $x,y\in R$,
then there exists  a $q\in Q$ such that 
$$U(x) = axq,\,\,\,V(y)=qya$$
for all $x,y\in R$.\end{enumerate}
We  define  $\mbox{sf-$\deg(t)= n$}$ if
sf-$\deg(t) > n-1$ but sf-$\deg(t) \not> n$.
If
 sf-$\deg(t) > n$ for every positive integer $n$, then we write  sf-$\deg(t) =\infty$.
\end{definition}

\begin{remark}
It should be pointed out that 
sf-$\deg(t) > n$ 
implies sf-$\deg(t) > n-1$ (otherwise the definition would not make sense). Indeed, if $a_k,b_k$ are elements from the definition that correspond to the sf-$\deg(t) > n$ case, then $a_kt,b_k$ are suitable for  the sf-$\deg(t) > n-1$ case.
\end{remark}

\begin{remark}
The symmetric fractional degree of $t$ depends on $R$ and $Q$, so it would be more appropriate to write something like sf-$\deg(t;R,Q)$ rather than only sf-$\deg(t)$. However, it will always be clear from the context which $R$ and $Q$ we have in mind.
\end{remark}

The concept of the symmetric fractional degree is a variation of that of a {\em fractional degree} which was used as the basic tool in the book \cite{FIbook}. It was introduced (in a slightly more general framework) in the more recent paper \cite{B16} in order to cover the situation where
$Q$ is one of the symmetric rings of quotients. The methods based on the fractional degree yield only results on left  rings of quotients. 
Since FIs are left-right symmetric, involving 
symmetric rings of quotients seems more natural. Moreover,
using this approach we will obtain a definitive result for symmetrically closed prime rings (Theorem \ref{tbei2}). 

The definition of the symmetric fractional degree may seem
rather complicated. To get some feeling, consider the case where $R=Q$ is a division ring. Then every nonzero element $a$ satisfies (SF1)
and (SF2) (the latter easily follows by taking $y=a^{-1}$ in $U(x)ya= axV(y)$), and 
the existence of $a_k,b_k$
having the desired properties follows from the Artin-Whaples Theorem (which we used already in Example \ref{ex3}), provided that $1,t,\dots,t^n$ are linearly independent over the center. This was just to indicate that Definition
\ref{def2} is not so unapproachable as it may seem at first glance. We will consider more general situations in the next subsection.

We now focus on showing that 
if $R$ is such that $C_Q(R)=C$ (i.e., the elements from the center $C$ of $Q$ are the only elements in $Q$ that commute with all elements in $R$), then
the existence of
an element $t\in R$ with sf-$\deg(t)\ge d$ implies that $R$ is a $d$-free subset of $Q$. We will give a complete proof since this result is of central importance in the FI theory.

We have to introduce some 
technical definitions and additional notation.

Given $t\in R$ and   $H:R^p \to Q$, we  write 
\begin{align*}H(x_it)\,\,\,&\mbox{for}\,\,\,H(x_1,\dots,x_{i-1},x_i t, x_{i+1},\dots,x_p),\\
H(x_it,x_jt)\,\,\,&\mbox{for}\,\,\,H(x_1,\dots,x_{i-1},x_i t, x_{i+1},\dots,x_{j-1},x_jt,x_{j+1},\dots,x_p), \,\,\,\mbox{etc.,}\end{align*}
and similarly,
\begin{align*}H(tx_i)\,\,\,&\mbox{for}\,\,\,H(x_1,\dots,x_{i-1},tx_i , x_{i+1},\dots,x_p),\\
H(tx_i,tx_j)\,\,\,&\mbox{for}\,\,\,H(x_1,\dots,x_{i-1},tx_i , x_{i+1},\dots,x_{j-1},tx_j,x_{j+1},\dots,x_p), \,\,\,\mbox{etc.}
\end{align*}
Let $1\le r\le p$.  We define $\mathcal{L}_{r,t}(H):R^p \to Q$ by
\begin{align*}
\mathcal{L}_{r,t}(H)(\ov{x}_p) 
=& t^{p-1} H(\ov{x}_p) - \sum_{1\le i\le p,\atop i\ne r} t^{p-2} H(tx_i) \\ & +   \sum_{1\le i< j \le  p,\atop i,j\ne r} t^{p-3}  H(tx_i,tx_j) \\
&
- \sum_{1\le i< j < k \le  p,\atop i,j,k\ne r} t^{p-4} H(tx_i,tx_j,tx_k) \\  &+ \dots + (-1)^{p-1} H(tx_1,\dots,tx_{r -1},tx_{r +1},\dots,tx_p).
\end{align*}

Similarly we define $\mathcal{R}_{s,t}(H):R^p \to Q$, where $1\le s\le p$,
 by
\begin{align*}
\mathcal{R}_{s,t}(H)(\ov{x}_p) 
=& H(\ov{x}_p)t^{p-1} - \sum_{1\le i\le p,\atop i\ne s} H(x_it) t^{p-2}\\ & +   \sum_{1\le i< j \le  p,\atop i,j\ne s} H(x_it,x_jt) t^{p-3} \\
&
- \sum_{1\le i< j < k \le  p,\atop i,j,k\ne s} H(x_it,x_jt,x_kt) t^{p-4}\\  &+ \dots + (-1)^{p-1} H(x_1t,\dots,x_{s -1}t,x_{s +1}t,\dots,x_pt).
\end{align*}
For example, if $s =2$ and $p=4$, then
\begin{align*}
\mathcal{R}_{2,t}(H)&(x_1,x_2,x_3,x_4) \\
=& H(x_1,x_2,x_3,x_4)t^3 \\
 -& H(x_1t,x_2,x_3,x_4)t^2 - H(x_1,x_2,x_3t,x_4)t^2 - H(x_1,x_2,x_3,x_4t)t^2\\
  + &H(x_1t,x_2,x_3t,x_4)t  + H(x_1t,x_2,x_3,x_4t)t +   H(x_1,x_2,x_3t,x_4t)t\\
-& H(x_1t,x_2,x_3t,x_4t). 
 \end{align*}
 
Note that
$$\mathcal{L}_{r,t}(H_1 + H_2) =  \mathcal{L}_{r,t}( H_1) + \mathcal{L}_{r,t}( H_2)$$
and$$\mathcal{R}_{s,t}(H_1 + H_2) =  \mathcal{R}_{s,t}( H_1) + \mathcal{R}_{s,t}( H_2). $$

We will call $H:R^p\to Q$   a {\em right $j$-function} if there exists a function $F:R^{p-1}\to Q$ such that
$$
H(\ov{x}_p)= x_jF(\ov{x}_p^j)
$$
for all $\ov{x}_p\in R^p$. Any function that is a sum of right $j$-functions will be called a {\em right function}. Similarly,
$K:R^p\to Q$  is  a {\em left $i$-function} if there exists a function $E:R^{p-1}\to Q$ such that
$$
K(\ov{x}_p)= E(\ov{x}_p^i)x_i
$$
for all $\ov{x}_p\in R^p$, and 
a {\em left function} is a sum of 
left $i$-functions. We remark that the basic FI \eqref{3S1} considers the situation where a left function is equal to a  right function.

We continue with a few simple lemmas. The first two need no proof.

\begin{lemma}\label{lr1}  
If $H$ is a right $j$-function, then so is $\mathcal{R}_{s,t}(H)$.  Therefore,
if  $H$ is a right function, then so is   $\mathcal{R}_{s,t}(H)$.
\end{lemma}

\begin{lemma} \label{ll1}  
If $H$ is a left $i$-function, then so is $\mathcal{L}_{r,t}(H)$. Therefore,
if
  $H$ is a left function, then so is   $\mathcal{L}_{r,t}(H)$. 
\end{lemma}

\begin{lemma}\label{lr2}  If $H$ is a left function, that is,   $$H(\ov{x}_p)= \sum_{i=1}^p E_i(\ov{x}_p^i)x_i,$$ then there exist functions
$G_i:R^{p-1}\to Q$, $i=0,1,\dots,p-2$, such that
  $$\mathcal{R}_{s,t}(H)(\ov{x}_p)=  \sum_{i=0}^{p-2} G_i (\ov{x}_p^s)x_s t^{i} + E_s(\ov{x}_p^s)x_s t^{p-1} 
	$$
 for all $\ov{x}_p\in R^p$.
\end{lemma}

\begin{proof}  If $H(\ov{x}_p)=  E_s(\ov{x}_p^s)x_s$, that is,  if
 $H$ is a   left $s$-function,
 the lemma is clear.
Therefore, it suffices to show that 
$\mathcal{R}_{s,t}(H)=0$ if $H$ is
a   left $i$-function where $i\ne s$. We will prove this for $s =1$ and $i=2$. The other cases can be handled similarly.

We are thus assuming that  $H(\ov{x}_p)= E(\ov{x}_p^2)x_2$. First observe that the first two terms from the definition of $\mathcal{R}_{1,t}(H)(\ov{x}_p) $ cancel out. The next terms, $- H(x_it) t^{p-2}$ with $i\ge 3$, cancel out with the terms $H(x_2t,x_it) t^{p-3}$ from the next summation. Next, the terms
 $H(x_it,x_jt) t^{p-3}$ with $3\le i < j\le p$ cancel out with the terms $-H(x_2t,x_it,x_jt) t^{p-4}$. Proceeding in this way, we finally observe that the terms $(-1)^{p-2}H(x_3t,\dots,x_pt)t$ and
$(-1)^{p-1}H(x_2t,x_3t,\dots,x_pt)$ cancel out.
\end{proof}

The following lemma can be proved similarly.

\begin{lemma}\label{ll2}  If $H$ is a right function, that is,   $$H(\ov{x}_p)= \sum_{j=1}^p x_jF_j(\ov{x}_p^j),$$ then there exist functions
$K_j:R^{p-1}\to Q$, $j=0,1,\dots,p-2$, such that
  $$\mathcal{L}_{r,t}(H)(\ov{x}_p)=  \sum_{j=0}^{p-2}t^j x_r  K_j (\ov{x}_p^r) + t^{p-1}x_r F_r(\ov{x}_p^r) 
	$$
 for all $\ov{x}_p\in R^p$.
\end{lemma}

We are now ready to prove   the theorem that we announced.

\begin{theorem}\label{mt} Let $Q$
be a unital ring with center $C$ and let $R$ 
 be a subring of  $Q$ such that $C_Q(R)=C$. Let $d\ge 1$. If $R$ contains an element  $t$ with {\rm sf}-$\deg(t) \ge d$, then  
 $R$ is a $d$-free subset of $Q$.
\end{theorem}

\begin{proof} 
Our goal is to show that conditions (a) and (b) of Definition \ref{defd} are fulfilled.  We will thus be interested in FIs \eqref{3S1} and \eqref{3S2}. We remark that we will consider the operator  $\mathcal{R}_{s,t}$  (resp.\ $\mathcal{L}_{r,t}$) with respect to the variables $x_i$ with $i\in I$ (resp.\ $x_j$ with $j\in J$); thus, the role of $p$ will be played by $|I|$ (resp.\ $|J|$). Indeed the number of variables $m$ may be greater than $p$, but the additional variables will be considered fixed when dealing with $\mathcal{R}_{s,t}$ and $\mathcal{L}_{r,t}$.

 Let us first prove (a). Assume, therefore, that the functions $E_i,F_j$ satisfy  (\ref{3S1})  with
$\max\{|I|,|J|\}\le d$. We have to show that they are of the form  \eqref{3S3}. 

Assume that  $I\ne \emptyset$. Choose $s\in I$ and apply the operator 
  $\mathcal{R}_{s,t}$ 
 to the FI \eqref{3S1}.
 Using  Lemmas \ref{lr1} and \ref{lr2} we see that there exist functions $G_i,H_j:R^{m-1}\to Q$ such that
\begin{equation}\label{eqr}
\sum_{i=0}^{|I|-2} G_i (\ov{x}_m^s)x_s t^{i} + E_s(\ov{x}_m^s)x_s t^{|I|-1} + \sum_{j\in J }  x_j H_j(\ov{x}_m^j)=0	
\end{equation}
for  all $\ov{x}_m\in R^m$.

By assumption,  there exist $t\in R$ and $a_k,b_k\in R$ such that $$\sum_{k}a_{k}t^ib_{k}=0,\,\,\, i=0,1,\dots,  d-2,$$ and $$a=\sum_{k}a_{k}t^{d-1}b_{k}$$
satisfies conditions (SF1)
and (SF2) from Definition \ref{def2}.
Replace $x_s$ by $x_s a_k$ in \eqref{eqr} and multiply the identity so obtained   from the right by $t^{d-|I|}b_k$ (here we use that $|I|\le d$). Summing up over  $k$ we obtain
\begin{equation}\label{eqk}E_s(\ov{x}_m^s)x_s a+ \sum_{j\in j} x_j L_j(\ov{x}_m^j)   =0
\end{equation}
for some functions $L_j:R^{m-1}\to Q$. 

If $J=\emptyset$, then  \eqref{eqk} shows that $E_s=0$ since $a$ satisfies  (SF1).  This means that \eqref{3S3} holds in this case.  We may therefore assume that $J\neq \emptyset$, and, analogously, that
$I\neq \emptyset$. Furthermore, this also show that
 it is now enough    to prove  that the functions $E_i$ are of the  form \eqref{3S3}. Indeed, assuming this is true, we can write \eqref{3S1} as
$$\sum_{j\in J} x_j\Bigl(F_j(\ov{x}_m^j) + \sum_{i\in I,\atop
i\not=j}p_{ij}(\ov{x}_m^{ij})x_i
+\lambda_j(\ov{x}_m^j)\Bigr) =0,$$
which is an FI  of the type \eqref{3S1} with $I=\emptyset$, implying that 
$$F_j(\ov{x}_m^j) + \sum_{i\in I,\atop
i\not=j}p_{ij}(\ov{x}_m^{ij})x_i
+\lambda_j(\ov{x}_m^j) =0;$$
this means that the functions $F_j$ are of the desired form.

Take $r\in J$ and 
apply  $\mathcal{L}_{r,t}$ to \eqref{eqk}. Using  Lemmas \ref{ll1} and \ref{ll2} it follows that there are functions  $\widehat{E},K_j:R^{m-1}\to Q$ such that 
$$
\widehat{E}(\ov{x}_m^s)x_s a + \sum_{j=0}^{|J|-2}t^j x_{r}  K_j (\ov{x}_m^{r}) + t^{|J|-1}x_{r} L_{r}(\ov{x}_m^{r})  =0.
$$
 Replacing $x_{r}$ by $b_kx_{r}$, multiplying from the left by $a_kt^{d-|J|}$, and summing up over $k$ we arrive at
\begin{equation}\label{edvec}
\widetilde{E}(\ov{x}_m^s)x_s a =- ax_{r} L_{r}(\ov{x}_m^{r}) 
\end{equation}
for some function $\widetilde{E}:R^{m-1}\to Q$. If $r\ne s$, then
we can use (SF2) (where all variables except $x_s$ and $x_{r}$ are considered fixed). Hence, there exists a $p_{s r}(\ov{x}_m^{s r})\in Q$ such that
$$-L_{r}(\ov{x}_m^{r}) = p_{s r}(\ov{x}_m^{s r})x_s a.$$
If $s\in J$ and  $r =s$, then  we substitute $y_s x_s$ for $x_s$  in   \eqref{edvec} to obtain
$$\bigl(\widetilde{E}(\ov{x}_m^s)y_s\bigr)x_s a = ay_s\bigl(-x_s L_s(\ov{x}_m^s)\bigr).$$
Observe that again we may  use (SF2). Thus, in particular there exists a 
$\lambda_s(\ov{x}_m^s)\in Q$ such that $$-x_s L_s(\ov{x}_m^s) = \lambda_s(\ov{x}_m^s)x_s a.$$
Hence, $$y_s\lambda_s(\ov{x}_m^s)x_s a = -y_s x_s L_s(\ov{x}_m^s)
= \lambda_s(\ov{x}_m^s)(y_s x_s)a.$$ 
This means that  $$[y_s,\lambda_s(\ov{x}_m^s)]x_sa =0.$$ Using (SF1) we see that $\lambda_s(\ov{x}_m^s)$ belongs to $C_Q(R)$, which is equal to $C$ by our assumption.  Setting $\lambda_s(\ov{x}_m^s) =0$ if $s\notin J$, we now see that  \eqref{eqk} can be written as 
$$\Bigl(E_s(\ov{x}_m^s) -  \sum_{j\in J,\atop
j\not=s}x_jp_{s j}(\ov{x}_m^{s j})
-\lambda_s(\ov{x}_m^s)\Bigr) x_sa=0.$$
Applying (SF1) once again it follows  that $E_s$ is of the form    \eqref{3S3}. This completes the proof of (a).

Let us prove (b). Assume, therefore, that
\begin{equation}\label{ec}\mu(\ov{x}_m) = \sum_{i\in I}E_i(\ov{x}_m^i)x_i+ \sum_{j\in J}x_jF_j(\ov{x}_m^j)
 \in  C\end{equation}
and $\max\{|I|,|J|\}\le d-1$. Since (a) holds, it is enough to prove that $\mu(\ov{x}_m)=0$.

Take $s\in I$ and  apply  $\mathcal{R}_{s,t}$ 
 to \eqref{ec}. Using Lemmas \ref{lr1} and \ref{lr2} we see that there are functions $G_i,H_j:R^{m-1}\to Q$ and $\mu_i:R^m\to C$ satisfying
\begin{equation}\label{ec2}\mu(\ov{x}_m) t^{|I|-1} +\sum_{i=0}^{|I|-2} \mu_{i} (\ov{x}_m) t^{i} =
\sum_{i=0}^{|I|-1} G_i (\ov{x}_m^s)x_s t^{i} +  \sum_{j\in J }  x_j H_j(\ov{x}_m^j).\end{equation}
Let $t$ be as above and replace $x_s$ by $x_s t$.
This gives
$$\sum_{i=0}^{|I|-1} \mu_{i}' (\ov{x}_m) t^{i} =
\Bigl(\sum_{i=0}^{|I|-1} G_i (\ov{x}_m^s)x_s t^{i} \Bigr) t+  \sum_{j\in J }  x_j H_j'(\ov{x}_m^j)$$ for some
functions $H_j':R^{m-1}\to Q$ and $\mu_i':R^m\to C$.
Using \eqref{ec2} we can  
rewrite the first summation on the right-hand side, and hence arrive at
\begin{equation}\label{ec3}\mu(\ov{x}_m) t^{|I|}+\sum_{i=0}^{|I|-1} \mu_i''(\ov{x}_m)t^i =  \sum_{j\in J }  x_j L_j(\ov{x}_m^j)\end{equation}
with $L_j:R^{m-1}\to Q$ and $\mu_i'':R^m\to C$.

Take $r\in J$ and
apply  $\mathcal{L}_{r,t}$ to \eqref{ec3}.  By Lemma \ref{ll2}, the right-hand side then becomes
\begin{equation}\label{egh}\sum_{j=0}^{|J|-2}t^j x_{r}  K_j (\ov{x}_m^{r}) + t^{|J|-1}x_{r} L_{r}(\ov{x}_m^{r}),
\end{equation}
while the left-hand side consists of terms from $Ct^i$. Therefore,
the expression \eqref{egh} commutes with $t$, which gives
$$ t^{|J|}x_{r} L_{r}(\ov{x}_m^{r}) + \sum_{j=0}^{|J|-1}t^j x_{r}  M_j (\ov{x}_m^{r})  =0$$
for some functions $M_j:R^{m-1}\to Q$.  
Let $a_k,b_k$ be the elements from the definition of {\rm sf}-$\deg(t)$.
Replacing $x_{r}$ by $b_kx_{r}$, multiplying from the left by $a_kt^{d-|J|-1}$ (here we use that $d-|J|-1\ge 0$), and summing up over  $k$ we obtain $$ax_rL_{r}(\ov{x}_m^{r})=0,$$
which by (SF1) yields $L_{r}(\ov{x}_m^{r})=0$. Therefore, the right-hand side of \eqref{ec3} is zero, and hence so is the left-hand side.
Multiplying 
$$\mu(\ov{x}_m) t^{|I|}+\sum_{i=0}^{|I|-1} \mu_i''(\ov{x}_m)t^i=0$$
 from the left by  $a_kt^{d-|I|-1}$, from the right by $b_k$, and summing up over $k$  we obtain $\mu(\ov{x}_m) a=0$. Using (SF1) once again we finally arrive at $\mu(\ov{x}_m)=0$. 
\end{proof}

\subsection{$d$-free prime rings} Our goal now is to show that Theorem \ref{mt} is applicable to prime rings. Starting with a prime ring $R$, one of course has to find a suitable ring $Q$ so that the symmetric fractionable degree of elements in $R$ can be computed.  For this purpose, we will define the maximal symmetric ring of quotients, introduced and studied in \cite{Lann}. 
 
We first recall that
 a left ideal $L$ of  a ring $R$ is {\em  dense} if for any
 $x_1,x_2\in R$ with $x_1\ne 0$ there exists an $r\in R$ such that $rx_1\ne 0$ and $rx_2\in L$.
A dense right ideal $T$ is defined analogously, i.e., 
for any
 $x_1,x_2\in R$ with $x_1\ne 0$ there exists an $s\in R$ such that $x_1s\ne 0$ and $x_2s\in T$.


\begin{definition}
Let $R$ be a ring.
A ring $Q=Q_{ms}(R)$ is called a {\em maximal symmetric  ring of quotients of $R$} if it satisfies the following conditions:
\begin{enumerate}
\item[(a')] $R$ is a subring of $Q$. 
\item[(b')] For every $q\in Q$ there exist a dense left ideal $L$ of $R$ and a dense right ideal $T$ of $R$  such that $Lq \subseteq R$ and $qT\subseteq R$.
\item[(c')]  For every nonzero $q\in Q$, every
 dense left  ideal $L$ of $R$, and every dense right ideal $T$ of $R$,  $Lq\ne \{0\}$ and $qT\ne \{0\}$.
\item[(d')] If $L$ is a dense left ideal of $R$, $T$  is dense right ideal of $R$, and  
 $E:L\to R$, $F:T\to R$ are functions satisfying 
$E(x)y= xF(y)$ for all $x\in L$ and all $y\in T$,
 then there exists a $q\in Q$ such that $E(x)=xq$ for all $x\in L$ and $F(y)=qy$ for all $y\in T$.
\end{enumerate}
\end{definition}

Like in Definition \ref{def1}, we have avoided the  assumption that $E$ and $F$ in (d) are 
left and right
 $R$-module homomorphisms, respectively. The proof that this assumption is unnecessary is essentially the same as that given in Remark \ref{rinthel}.
 
 We are interested only in the case where $R$ is a prime ring, so let us assume this.
  Then the maximal symmetric  ring of quotients $Q_{ms}(R)$ exists, is unique up to isomorphism, and contains the symmetric Martindale ring of quotients $Q_s(R)$ as a subring. On the other hand,  $Q_{ms}(R)$ is contained in the 
  {\em maximal left  ring of quotients} $Q_{ml}(R)$; for the definition and properties of the latter see \cite{BMMb} or \cite{Lam}.
  All the rings
  $Q_s(R)$, $Q_{ms}(R)$, and 
  $Q_{ml}(R)$ have the same center, i.e., the extended centroid $C$.  Moreover, if $Q$ is any of these three rings,
  then \begin{equation}
      \label{cqrjec}C_Q(R)=C
  \end{equation}(see \cite[Remark 2.3.1]{BMMb}). Another important property of $Q_{ms}(R)$  is that, unlike $Q_{s}(\,\cdot\,)$, $Q_{ms}(\,\cdot\,)$ is a closure operation \cite{Lann}, that is,
\begin{equation}\label{qclos}Q_{ms}(Q_{ms}(R))= Q_{ms}(R).\end{equation}
  
  The following lemma is crucial for computing the symmetric fractional degree
of elements in a prime ring $R$ with respect to $Q=Q_{ms}(R)$.

\begin{lemma}\label{lt}
Let $R$ be a prime ring.  If $a,b$ are nonzero elements in $R$ and $U,V:R\to Q_{ms}(R)$ are  functions satisfying 
\begin{equation}\label{fixy2}U(x)ya= bxV(y)\end{equation} for all $x,y\in R$,
 then there exists a $q\in Q_{ms}(R)$ such that $$U(x)=bxq,\,\,\,V(y)=qya$$ for all $x,y\in R$.
\end{lemma}

\begin{proof}
   By substituting  $x+x'$ for $x$
in  \eqref{fixy2} we infer that $$\bigl(U(x+x') - U(x)-U(x')\bigr)ya=0$$
for all $x,x',y\in R$.
That is, 
$$\bigl(U(x+x') - U(x)-U(x')\bigr)I=\{0\}$$
where $I$ is the ideal of $R$ generated by $a$. Note that every nonzero ideal of a prime ring is dense as a right (or left) ideal.
Therefore, (c') implies that  $$U(x+x') = U(x)+U(x')$$ for all $x,x'\in R$.  

Next, for all $x,y,z\in R$
 we have
  $$bxU(y)za= bx(byV(z))= b(xby)V(z) =
 U(xby)za.$$ Thus,  $$\bigl(U(xby) - bxU(y)\bigr)za=0,$$ which, as above, implies that
\begin{equation}\label{exbyy}
U(xby)=bxU(y)\end{equation} for all $x,y\in R$.

Write $Q$ for $Q_{ms}(R)$.
We claim that $L=QbR$  is a dense left ideal of $Q$. To prove this, let $y_1,y_2\in Q$ with $y_1\ne 0$. According to (b'), $L'y_2\subseteq R$  holds for some  dense left ideal $L'$ of $R$. By (c'), there is a $u'\in L'$ such that 
$u'y_1\ne 0$.
 Since $L$ contains $RbR$, which is a  nonzero ideal of $R$ and hence a dense left ideal of $R$, (c') also implies that $u(u'y_1)\ne 0$
 for some $u\in L$. The element $q=uu'$ therefore satisfies $qy_1\ne 0$ and $qy_2\in L$. This proves our claim.

Let 
$E:L\to Q$ be given by
$$E\Bigl(\sum_i q_iby_i\Bigr) = \sum_i q_iU(y_i).$$
We must prove  that $E$ is well-defined. Assume, therefore, that $\sum_i q_iby_i =0$. In view of (b'), for every $i$ we can choose a dense left ideal $L_i$ of $R$ such that $L_iq_i\subseteq R$. Observe that
$\widehat{L}=\bigcap_i L_i$ is also a dense left ideal of $R$, and that
$\widehat{L}q_i\subseteq R$ for every $i$. Take $y\in \widehat{L}$. Since
$yq_i\in R$ and $U$ is an additive function, we have
\begin{equation}\label{dof} \sum_iU(yq_iby_i) = U\Bigl( y\cdot \sum_i q_iby_i\Bigr) =
0.\end{equation}
By \eqref{exbyy}, 
$U (yq_iby_i) = byq_iU(y_i).$
Therefore it follows from \eqref{dof}  that
$$I\widehat{L} \Bigl(\sum_i q_iU(y_i)\Bigr) = \{0\},$$
where $I$ is the ideal of $R$ generated by $b$.
Since both $I$ and $\widehat{L}$ are dense left ideals,
it follows from (c') that 
 $\sum_i q_iU(y_i)=0$. This proves that $E$ is well-defined.

Making obvious modifications in the above arguments one shows that $T=RaQ$ is a dense right ideal of $Q$ and 
that $V$ is an additive function that satisfies 
\begin{equation*}\label{eyaz}
V(xay)=V(x)ya\end{equation*}
for all $x,y\in R$, from which we deduce that
$F:T\to Q$ given by
$$ F\Bigl(\sum_j z_jas_j\Bigr) = \sum_j V(z_j)s_j$$
is a well-defined function.

Using \eqref{fixy2}, we see that the functions $E$ and $F$ are connected as follows:
\begin{align*}
E\Bigl(\sum_i q_iby_i\Bigr)\Bigl(\sum_j z_jas_j\Bigr) &= \sum_{i,j}q_iU(y_i)z_jas_j\\ &= \sum_{i,j} q_i by_iV(z_j)s_j =  \Bigl(\sum_i q_iby_i\Bigr) F\Bigl(\sum_j z_jas_j\Bigr).
\end{align*}
We may therefore use (d') to obtain an element  $q\in Q_{ms}(Q)$ such that $E(u)=uq$ for all $u\in L$ and $F(v)=qv$
for all $v\in T$. Accordingly, $$U(x)= E(bx)= bxq$$ for all  $x\in R$, and 
$$V(y)=F(ya)= qya$$ for all $y\in R$. 
Finally, 
   \eqref{qclos} tells us that $q$ actually lies in $Q$. This completes the proof.
\end{proof}

We are now able to relate the strong fractional degree to the following more familiar notion.
For an element $t$ in a prime ring $R$ and a positive integer $n$, we write $$\deg(t)=n$$ if $t$ is algebraic of degree $n$ over the extended centroid $C$. If $t$ is not algebraic over $C$, we write $\deg(t)=\infty$. 

\begin{lemma}\label{lemqr}
Let $R$ be a prime ring and let 
$Q=Q_{ms}(R)$. Then {\rm sf}-$\deg(t) =\deg(t)$ for every $t\in R$.
\end{lemma}

\begin{proof}
Assume that $\deg(t)> n$. Then
  $1,t,\dots,t^n$ are linearly independent over $C$,
  and hence there exist 
   $a_k,b_k\in R$ such that $\sum_{k}a_{k}t^ib_{k}=0$, $i=0,1\dots,  n-1$, and $a=\sum_{k}a_{k}t^nb_{k} \ne 0$ \cite[Theorem 2.3.3]{BMMb}. Lemma \ref{lt} (for $b=a$)  shows that $a$ satisfies condition (SF2). The primeness of $R$ implies that $a$ also satisfies condition (SF1). Therefore, sf-$\deg(t) > n$.
   
Conversely, sf-$\deg(t) > n$  implies $\deg(t) > n$. This is because
$ \sum_{k}a_{k}t^ib_{k}=0$,  $i=0,1,\dots,n-1$, along with $ \sum_{k}a_{k}t^nb_{k}\ne 0$ can hold only if $t^n$ does not lie in the linear span of $1,t,\dots,t^{n-1}$.
\end{proof}

 For any prime ring $R$, we set $$\deg(R)=\sup\{\deg(t)\,|\,t\in R\}.$$ It is well known that 
$\deg(R) \le n < \infty$ if and only if $R$ satisfies $s_{2n}$,
the standard polynomial identity of degree $2n$.  Equivalently, $R$ can be 
embedded to the ring of $n\times n$ matrices over a field (or, more precisely, the ring of central quotients of $R$ is a central simple algebra of dimension at most $n^2$).

We are now in a position to prove the fundamental theorem.

\begin{theorem}\label{tbei}
Let $R$ be a prime ring and let $d\ge 1$. Then $R$  is a $d$-free subset of $Q=Q_{ms}(R)$ if and only if
  $\deg(R)\ge d$.
\end{theorem}

\begin{proof}
Suppose $\deg(R) \ge d$. Take $t\in R$ with $\deg(t)\ge d$. Then {\rm sf}-$\deg(t)\ge d$ by Lemma \ref{lemqr}, and hence Theorem \ref{mt} shows that $R$ is a $d$-free subset of $Q$.

If $\deg(R) < d$, then $R$ can be embedded to $M_{d-1}(Z)$ with $Z$ a field, and  we can essentially repeat the argument from Example \ref{ex4} to show that $R$ is not a $d$-free subset of $Q$; see \cite[Corollary 4.21 and Theorem C.2]{FIbook} for details. \end{proof}

Theorem \ref{tbei} is taken from  the present author's paper \cite{B16}, but should  nevertheless be  attributed to  Beidar   \cite{Bei}
(who generalized a version of the $d=2$
case obtained earlier by the author \cite{Bbideg2}). Indeed
Beidar's theorem involves $Q_{ml}(A)$ 
rather than
$Q_{ms}(A)$, but this is a  technical issue. 

One cannot substitute $Q_s(R)$ for $Q_{ms}(R)$ in Theorem \ref{tbei}, see \cite[Corollary 4.3]{B16} for a counterexample. The 
main obstacle is that
$Q_s(\,\cdot\,)$ is not a closure operation. It is therefore natural to restrict ourselves to
prime rings $R$ such that $Q_s(R)=R$.  They are  called  {\em symmetrically closed} prime rings. 

Making some rather obvious modifications (actually simplifications) in the proof of Lemma \ref{lt} one proves the following.

\begin{lemma}\label{lt2}
Let $R$ be a symmetrically closed prime ring.  If $a,b$ are nonzero elements in $R$ and $U,V:R\to R$ are  functions satisfying 
\begin{equation*}\label{fixy22}U(x)ya= bxV(y)\end{equation*} for all $x,y\in R$,
 then there exists a $q\in R$ such that $$U(x)=bxq,\,\,\,V(y)=qya$$ for all $x,y\in R$.
\end{lemma}

The following version of Lemma \ref{lemqr} now readily follows.

\begin{lemma}\label{lemqr2}
Let $R=Q$ be a symmetrically closed prime ring. Then {\rm sf}-$\deg(t) =\deg(t)$ for every $t\in R$.
\end{lemma}

Repeating the proof of Theorem \ref{tbei}, but referring to Lemma \ref{lemqr2} instead of Lemma
\ref{lemqr}, we obtain the following theorem.

\begin{theorem}\label{tbei2}
Let $R$ be a symmetrically closed prime ring and let $d\ge 1$. Then $R$  is a $d$-free ring  if and only if
  $\deg(R)\ge d$.
\end{theorem}

We remark that \cite[Corollary 4.2]{B16} shows that we cannot replace 
``symmetrically closed" by 
``centrally closed" in Theorem \ref{tbei2}.

The obvious advantage of Theorem \ref{tbei2} compared to Theorem \ref{tbei} is that it does not involve larger rings than $R$. On the other hand,  it considers a considerably smaller class of rings. Anyway, this class does include  important examples. First of all, simple unital rings are obviously symmetrically closed, so the following corollary holds.

\begin{corollary}\label{ctbei2}
Let $R$ be a simple unital ring and let $d\ge 1$. Then $R$  is a $d$-free ring  if and only if
 the dimension of $R$ over its center is at least $d^2$.
\end{corollary}
Thus, if $R$ is infinite-dimensional over its center, then it is $d$-free for every $d\ge 1$.

 The simplest but most important case of Corollary \ref{ctbei2}  is the following.

\begin{corollary}\label{ctbei2a}
Let $F$ be a field. The ring
$M_d(F)$ is $d$-free, but is not $(d+1)$-free.
\end{corollary}

In fact, more can be said about simple unital rings $R$. Namely, FIs involving functions that have their ranges in any unital $R$-bimodule can be treated in much the same way; see \cite[Corollary 2.21]{FIbook}. 

Another important example of a symmetrically closed prime ring is a noncommutative free algebra \cite{K, Pass0}.

\begin{corollary}\label{mt3k} Let $F$ be a field. 
The free algebra $F\langle X_1, X_2,\dots\rangle$  on at least two indeterminates  is a $d$-free ring for every $d\ge 1$.
\end{corollary}

The proofs given in this and the preceding subsection are based on  a direct method for establishing the $d$-freeness.  
The more standard method,  used everywhere in the book \cite{FIbook}, is based  on the more general notion of $(t;d)$-freeness. It is somewhat more complicated (especially notationally), so we have decided not to consider it in detail in this expository paper.
We will give only some basic information in the next few paragraphs.


We begin by introducing the necessary notation.
 As at the beginning of this section, let $Q$ be a unital ring with center $C$, let $R$ be a nonempty subset of $Q$, let
$m$ be a positive integer,
and let $I$ and $J$ be subsets of $\{1,\dots,m\}$.
Further, let $a,b$ be nonnegative integers, and for all
$i\in I$, $0\leq u\leq a$,
and $j\in
J$, $0\leq v\leq b$, let
$$E_{iu}:R^{m-1}\to Q\quad\mbox{and}\quad F_{jv}:R^{m-1}\to Q
$$
be arbitrary functions.

Fix an element $t\in Q$ and
consider the following identities:
\begin{equation}
\sum_{i\in I}\sum_{u=0}^aE_{iu}(\ov{x}_m^i)x_it^u+ \sum_{j\in
J}\sum_{v=0}^bt^vx_jF_{jv}(\ov{x}_m^j)  = 
0\label{3S21}\end{equation}
for all $\ov{x}_m\in R$, and
\begin{equation}
\sum_{i\in I}\sum_{u=0}^aE_{iu}(\ov{x}_m^i)x_it^u+ \sum_{j\in
J}\sum_{v=0}^bt^vx_jF_{jv}(\ov{x}_m^j)  \in  C\label{3S22}
\end{equation}
for all $\ov{x}_m\in R$.
Observe that \eqref{3S21}, and hence also \eqref{3S22}, holds if there exist functions
\begin{eqnarray*}
&  & p_{iujv}:R^{m-2}\to Q,\;\;
i\in I,\; j\in J,\; i\not=j,\;0\leq u\leq a,\; 0\leq v\leq b,\nonumber\\
&  & \lambda_{kuv}:R^{m-1}\to C,\;\; k\in I\cup J,\, 0\leq
u\leq a,\; 0\leq v\leq b,
\end{eqnarray*}
such that
\begin{eqnarray}\label{3S23}
E_{iu}(\ov{x}_m^i) & = & \sum_{j\in J,\atop
j\not=i}\sum_{v=0}^bt^vx_jp_{iujv}(\ov{x}_m^{ij})
+\sum_{v=0}^b\lambda_{iuv}(\ov{x}_m^i)t^v,\nonumber\\
F_{jv}(\ov{x}_m^j) & = & -\sum_{i\in I,\atop
i\not=j}\sum_{u=0}^ap_{iujv}(\ov{x}_m^{ij})x_it^u
-\sum_{u=0}^a\lambda_{juv}(\ov{x}_m^j)t^u,\\
&  & \lambda_{kuv}=0\quad\mbox{if}\quad k\not\in I\cap
J\nonumber
\end{eqnarray}
for all $\ov{x}_m\in R$, $i\in I$, $j\in J$, $0\leq u\leq
a$, $0\leq v\leq b$. We call
(\ref{3S23})  a {\it standard solution} of (\ref{3S21}) as well as of 
\eqref{3S22}.

\begin{definition}\label{3SD2} Let $d$ be a positive integer and let $t\in Q$. The set
 $R$ is said to be a  {\em $(t;d)$-free subset} of
$Q$
if the following  two
conditions hold for all $m\ge 1$, all $I,J\subseteq
\{1,2,\ldots,m\}$,  and all  $a,b\ge 0$:
\begin{enumerate}
\item[(a)] If $\max\{|I|+a,|J|+b\}\le d$, then (\ref{3S21}) implies
(\ref{3S23}).
\item[(b)] If
$\max\{|I|+a,|J|+b\}\le d-1$, then (\ref{3S22}) implies
(\ref{3S23}).
\end{enumerate}
\end{definition}

Note that in the case where
$a=b=0$, \eqref{3S21} reduces to \eqref{3S1}, \eqref{3S22} reduces to \eqref{3S2}, and \eqref{3S23} reduces to \eqref{3S3}.
Therefore, the following is true:
$$\mbox{ $R$ is a $(t;d)$-free subset of $Q$ for some $t$} \implies \mbox{$R$ is a  $d$-free subset of $Q$.}    $$
In fact, a common way to prove that a set is $d$-free is to show that it is $(t;d)$-free for some $t$. In concrete situations, there is no big difference between ``$d$-free"
and ``$(t;d)$-free for some $t$".  However, Theorems \ref{tinv}
and \ref{tza} below show that the  more general $(t;d)$-free sets can be useful for finding new examples of $d$-free sets. 

It is well known that a prime ring $R$ is a $(t;d)$-free subset of $Q_{ml}(R)$ 
whenever $t\in R$ satisfies
$\deg(t)\ge d$ \cite[Theorem 5.11]{FIbook}. In this statement, we can replace
$Q_{ml}(R)$ with $Q_{ms}(R)$. 

\begin{theorem}\label{tdbei}
Let $R$ be a prime ring and let $d\ge 1$.
If $t\in R$ is such that
$\deg(t)\ge d$, then 
$R$ is a $(t;d)$-free subset of $Q_{ms}(R)$. 
\end{theorem}

The proof differs from that of \cite[Theorem 5.11]{FIbook} in a few details. The necessary changes can be easily made. There is one point, however, which may not be so obvious at first glance. It concerns the following situation: 
$U_0,U_1,\dots,U_n$
and $V$ are functions from $R$ to $Q=Q_{ms}(R)$
such that
\begin{equation}\label{enn}
U_0(x)y + U_1(x)yt +\dots + U_n(x)yt^n=bxV(y) \end{equation}
for all $x,y\in R$, where $b$ is a nonzero element in $R$  and $n< d$. The goal is to show that there exist $q_0,q_1,\dots,q_n\in Q$ such that
\begin{equation}\label{karsmo}
    V(y) = q_0y + q_1yt+\dots + q_nyt^n
\end{equation}for all $y\in R$. To prove this, we use
\cite[Theorem 2.3.3]{BMMb} to find
 $a_k,b_k\in R$ such that
 $\sum_{k}a_{k}t^ib_{k}=0$, $i=0,1\dots,  n-1$, and $a=\sum_{k}a_{k}t^nb_{k} \ne 0$. Write $ya_k$ for $y$, multiply the relation so obtained from the right by $b_k$, and sum up over $k$ to obtain 
 $$U_n(x) y a= bx\overline{V}(y)$$
 for all $x,y\in R$ and some new function
 $\overline{V}$. Lemma \ref{lt} tells us that there exists a $q_n\in Q$ such that, in particular,
 $U_n(x) = bxq_n$
 for all $x\in R$. Similarly we see that
  $U_i(x) = bxq_i$
 for all $x,y\in R$ and some $q_i\in Q$. Hence,
 \eqref{enn} becomes
 $$bx\bigl(V(y) -q_0y - q_1yt - \dots -q_nyt^n\bigr)=0$$
 for all $x,y\in R$. Using (c') we thus arrive at \eqref{karsmo}. 
 
With this, the interested reader 
should be able to give a complete proof  of Theorem \ref{tdbei} by following the arguments from \cite{FIbook}.

Let us finally mention that the study of $d$-freeness is not limited to prime rings. In particular,
the $d$-freeness of more general semiprime rings \cite[Section 5.3]{FIbook}, and in particular of von Neumann algebras \cite{Alam}, is also well understood. 
Further, the following result, an extension of Corollary \ref{ctbei2a}, deserves to be mentioned.

\begin{theorem}\label{matfree}
 If $S$ is any unital ring, then the matrix ring $R=M_d(S)$ is $d$-free.  
\end{theorem}

 We remark that the ring $R=M_d(S)$ is prime (resp.\ semiprime) if and only if $S$ is prime (resp.\ semiprime), so this theorem is independent of the aforementioned results. 
 Moreover, as in the case of simple unital rings, we can consider
FIs involving functions whose ranges are not necessarily in $R$ but in any unital $R$-bimodule, see \cite[Corollary 2.22]{FIbook}.

One may wonder what can be said about 
 triangular rings.
The following simple example shows that we are facing serious limitations concerning their $d$-freeness.

\begin{example}
Let $S$ be any unital ring (possibly a field) and let $R=T_n(S)$, $n\ge 2$, be the ring of all upper triangular matrices over $S$. Observe that the matrix unit $e_{1n}$ satisfies 
$$ e_{1n}[x_1,x_2] =0$$
for all $x_1,x_2\in R$.
This means 
that $$E:R\to R,\,\,
E(x)= e_{1n}x,$$ is a nonzero function satisfying
$$E(x_1)x_2 - E(x_2)x_1=0$$
for all $x_1,x_2\in R$. Therefore, $R$ is not even a $2$-free ring (regardless of $n$).
\end{example}

Nevertheless, some special FIs in the rings of upper triangular matrices, as well in more general triangular rings, were successfully studied in \cite{BE, Eret, Wfi2}.

\subsection{Constructing new  $d$-free sets from old} 
Our aim now is to present several results showing that one can construct new $d$-free sets from the old ones. Unlike in the preceding subsection, these sets are not necessarily subrings. Along with the above results, this will make it possible for us to provide various concrete examples of $d$-free sets.

We will state the results in this subsection without proofs, which are mostly self-contained but intricate. 
With the exception of Theorem \ref{ttp}, which is a more recent result from the 2016 paper \cite{Btp}, all other theorems in this subsection are presented, along with detailed proofs, in the book \cite{FIbook}.
The original sources are \cite{BBCMinvii, Onher2, BC1, BeiMart}.

We will use the by now standard notation. In particular, $Q$ will denote a unital ring with center $C$, $R$ will denote its nonempty subset, and $d$ will denote a positive integer.

The condition that $R$  is $d$-free indicates that $R$ is large in some sense. Our first result is therefore very natural.

\begin{theorem}\label{t1}
Let $R$ be a $d$-free subset of $Q$. If $T$ is any subset of $Q$ such that $R\subseteq T$, then $T$ is a $d$-free subset of $Q$ too.
\end{theorem}

The proof is not that simple as one might expect in view of the simplicity of the statement.  It is based on a certain generalization of the notion of a $d$-free subset which involves functions $E_i,F_j$ defined on $R_1\times \dots\times R_m$ with $R_i$ possibly different subsets of $Q$.

The set of all upper triangular matrices of the form
$$
\left[ \begin{matrix} x & y \\ 0 & x  \end{matrix} \right],$$
where $x$ and $y$ are elements from the ring $Q$, is a ring under the standard matrix operations. We denote it by $\widetilde{Q}$. 

\begin{theorem}\label{thomder}
Let $R$ be a $d$-free subset of $Q$. If
$\delta:R\to Q$ is any function, then
the set of all matrices of the form
$$
\left[ \begin{matrix} x & \delta(x)  \\0 & x  \end{matrix} \right],$$
where $x\in R$, is a $d$-free subset of 
$\widetilde{Q}$. 
\end{theorem}

The motivation behind this theorem is the observation that, under the assumption that $R$ is a subring, the map 
$\delta:R\to Q$ is a derivation if and only if the map $\varphi:R\to \widetilde{Q}$, 
$$\varphi(x)=
\left[ \begin{matrix} x & \delta(x) \\ 0 & x  \end{matrix} \right],$$
is a homomorphism. This makes the theorem a useful tool for reducing certain problems on derivations to analogous problems on homomorphisms.

By $Q^{\rm op}$ we denote the opposite ring of the ring $Q$.

\begin{theorem}\label{t3}
Let $R$ be a $d$-free subset of $Q$. 
Then the sets $\{(x,x)\,|\ x\in R\}$
and $\{(x,-x)\,|\ x\in R\}$ are $d$-free
subsets of the ring $Q\times Q^{\rm op}$.
\end{theorem}

Let us explain the motivation behind this theorem.
Recall that a map $*$ from a ring $R$ to itself is called an {\em involution} if it satisfies
$$(x+y)^* = x^* + y^*,\quad (xy)^* =y^*x^*,\quad\mbox{and}\quad (x^*)^*=x $$
for all $x,y\in R$. 
An element $x\in R$ is said to be {\em symmetric} if $x^*=x$, and is said to be
{\em skew} (or {\em skew-symmetric}) if $x^*=-x$.
If $Q$ is any ring, then we can endow the ring $Q\times Q^{\rm op}$ with the involution defined by $(x,y)^* = (y,x)$.
The first set from Theorem \ref{t3} consists of all symmetric elements with respect to this involution, and the second set consists of all skew elements. Thus, Theorem \ref{t3} states that the sets of symmetric and skew elements (with respect to this special involution) are 
$d$-free provided that the original set is $d$-free.

The next result concerns tensor products. We now assume that our ring $Q$ is an algebra over a field.

\begin{theorem}\label{ttp}
Let  $Q$ and $A$ be unital algebras. If 
 $A$ is finite-dimensional and $R$ is a $d$-free
subset of $Q$, then  $\{x \otimes a \,|\, x\in R, a \in  A\}$ is a $d$-free subset of $Q\otimes A$.
\end{theorem}

If $R$ is a linear subspace of $Q$, then 
Theorem \ref{ttp}, together with Theorem \ref{t1}, implies that
$R\otimes A$ is a $d$-free subset of $Q\otimes A$. In particular, if  $R=Q$ is a $d$-free ring, then $R\otimes A$ is
a $d$-free ring for every finite-dimensional algebra $A$.

It was shown by an example 
that the assumption that $A$ is finite-dimensional is necessary. However, some  important FIs can be handled in $Q\otimes A$ even when $A$ is infinite-dimensional. See
\cite[Section 5]{Btp}.

 We also mention the papers \cite{Ere, Wangqa} which give some further insight into the special case where $A$ is the algebra of upper triangular matrices.

The next two theorems involve $(t;d)$-free sets.
The first one concerns rings with involution. We will assume that the subset $R$ of $Q$ is a subring endowed with involution $*$. By $S$ we denote the set of all symmetric elements in  $R$, and by $K$ the set of all skew elements in $R$.

\begin{theorem}\label{tinv}
Let $R$ be a ring with involution. If $R$ is a  $(t;2d+1)$-free subset of $Q$ for some $t\in S\cup K$, then
both $S$ and $K$ are $d$-free subsets of $Q$.
\end{theorem}

Theorem \ref{tinv} is derived from a more general result treating FIs of the form
\begin{eqnarray*}
\sum_{i\in I}\sum_{u=0}^aE_{iu}(\ov{x}_m^i)^ix_it^u
&+&\sum_{j\in J}\sum_{v=0}^bt^v x_jF_{jv}^j(\ov{x}_m^j)\\
&+&\sum_{k\in K}\sum_{w=0}^{c}G_{kw}^k(\ov{x}_m^k)x_k^*t^w +\sum_{l\in
L}\sum_{z=0}^{d}t^z x_l^*H_{lz}^l(\ov{x}_m^l) =0.
\end{eqnarray*}
However, Theorem \ref{tinv} is sufficient for most applications, so we shall not state this more general result here. The interested reader can find it in \cite[Section 3.5]{FIbook}. As one would expect, its proof is rather long.

We have already dealt with the Lie product $[x,y]=xy-yx$. We also need the {\em Jordan product}
$$x\circ y = xy+yx$$
of ring elements $x$ and $y$. Of course, an associative ring endowed with the Lie product becomes a Lie ring (meaning that $[x,x]=0$ and $[[x,y],z] + [[z,x],y] + [[y,z],x]=0$ for all $x,y,z$), and endowed with the Jordan product becomes a Jordan ring (meaning that $x\circ y=y\circ x$ and $((x\circ x)\circ y)\circ x)= (x\circ x)\circ (y\circ x)$ for all $x$ and $y$).

\begin{theorem}\label{tza}
Let $R$ be a $(t;d+1)$-free subset
of $Q$, where $t\in Q$ is not algebraic of degree $1$ or $2$ over $C$. If $R'$ is a nonempty subset of $R$ such that either $[t,R]\subseteq R'$
or $t\circ R\subseteq R'$, then $R'$ is a $d$-free subset of $Q$.
\end{theorem}

To explain the meaning of this theorem, we recall that an additive subgroup $M$ of a ring $A$  is called is called a {\em Lie subring} of $A$ if $[M,M]\subseteq M$, and that
an additive subgroup $L$ of a 
Lie subring $M$ is called 
 a {\em Lie ideal} of $M$ if $[L,M]\subseteq L$. Similarly we define {\em Jordan subrings} and their {\em Jordan ideals}. Note that Theorem \ref{tza} is applicable to Lie and Jordan ideals.

Combining the last two theorems with Theorem 
\ref{tdbei} and some standard facts on Lie ideals (see \cite[Section 5.2]{FIbook} for details) we obtain the following corollaries showing that some important subsets of prime rings are $d$-free. We remark that by a {\em noncentral} Lie ideal  we mean  a Lie ideal that is not contained in the center of the ring  considered.

\begin{corollary}
Let $R$ be a prime ring  and let $L$ be a noncentral Lie ideal of $R$. If  {\rm char}$(R)\ne 2$ and $\deg(R)\ge d+1$, then $L$ is a $d$-free subset of $Q_{ms}(R)$.
\end{corollary}

As above,  we write $S$ (resp.\ $K$) for the set of all symmetric (resp.\ skew) elements in a ring $R$ with involution.

\begin{corollary}
Let $R$ be a prime ring with involution. If  {\rm char}$(R)\ne 2$ and $\deg(R)\ge 2d+1$, then $S$ and $K$ are $d$-free subsets of $Q_{ms}(R)$.
\end{corollary}

\begin{corollary}\label{leserab}
Let $R$ be a prime ring with involution and let $L$ be a noncentral Lie ideal of $K$. If  {\rm char}$(R)\ne 2$ and $\deg(R)\ge 2d+3$, then $L$ is a $d$-free subset of $Q_{ms}(R)$.
\end{corollary}

\subsection{A characterization of $d$-free sets}
We will now consider a more general version of basic FIs \eqref{3S1} and \eqref{3S2}.
It involves a function
$$\alpha:S\to Q$$
where  $S$ is an arbitrary nonempty set, and $Q$ is, as always, a unital ring. The role of $\alpha$ will be different from the roles of other functions occurring in FIs. We consider $\alpha$ as a fixed, given function and our goal will not be to describe its form, but to describe other functions in terms of $\alpha$.
We will write $x^\alpha$ rather than  $\alpha(x)$ and $S^\alpha$ rather than $\alpha(S)$.

Let $m$, $I$, $J$, and $C$ have the usual  meaning and let 
$$E_i:S^{m-1}\to Q\quad\mbox{and}\quad F_j:S^{m-1}\to Q$$
be arbitrary functions. The meaning of $ \ov{x}_m$, 
$\ov{x}_m^i$, etc., will be the same as above, just that now each $x_i$ belongs to $S$ rather than to $R$.
The FIs we will be  interested in are
\begin{equation}
\sum_{i\in I}E_i(\ov{x}_m^i)x_i^\alpha+ \sum_{j\in J}x_j^\alpha F_j(\ov{x}_m^j)
 = 
0\label{3S1al}\end{equation}
for all $\ov{x}_m\in S^m$,
and 
\begin{equation}
\sum_{i\in I}E_i(\ov{x}_m^i)x_i^\alpha+ \sum_{j\in J}x_j^\alpha F_j(\ov{x}_m^j)
 \in  C\label{3S2al}
\end{equation}
for all $\ov{x}_m\in S^m$. If
$S=R$ and $\alpha$ is the identity function, then these are the familiar FIs \eqref{3S1} and \eqref{3S2}. 

A standard solution of 
 both  (\ref{3S1al}) and (\ref{3S2al}) is of course defined as follows: there exist
  functions 
\begin{eqnarray*}
&  & p_{ij}:S^{m-2}\to Q,\;\;
i\in I,\; j\in J,\; i\not=j,\nonumber\\
&  & \lambda_k:S^{m-1}\to C,\;\; k\in I\cup J,
\end{eqnarray*}
such that 
\begin{eqnarray}\label{3S3al}
E_i(\ov{x}_m^i) & = & \sum_{j\in J,\atop
j\not=i}x_j^\alpha p_{ij}(\ov{x}_m^{ij})
+\lambda_i(\ov{x}_m^i),\quad i\in I,\nonumber\\
F_j(\ov{x}_m^j) & = & -\sum_{i\in I,\atop
i\not=j}p_{ij}(\ov{x}_m^{ij})x_i^\alpha
-\lambda_j(\ov{x}_m^j),\quad j\in J,\\
&   & \lambda_k=0\quad\mbox{if}\quad k\not\in I\cap J\nonumber.
\end{eqnarray}

The following is a natural generalization of 
Definition \ref{defd}.

\begin{definition}\label{defdal}Let $d$ be a positive integer. The pair $(S;\alpha)$ is  said to be  {\em
$d$-free} with respect to $Q$
if  
the following two conditions hold for all $m\ge 1$ and all
 $I,J\subseteq \{1,2,\ldots,m\}$:
\begin{enumerate}
\item[(a)]If
$\max\{|I|,|J|\}\le d$, then (\ref{3S1al}) implies (\ref{3S3al}).
\item[(b)] If
$\max\{|I|,|J|\}\le d-1$, then (\ref{3S2al}) implies (\ref{3S3al}).
\end{enumerate}
\end{definition}

If $\alpha$ is an injective function from $S$ to $Q$, and hence a bijective function from $S$ onto  $R=S^\alpha$, then  by writing $y_i^{\alpha^{-1}}$ for $x_i$ we easily see that the condition that  $(S;\alpha)$ is  $d$-free with respect to $Q$ is equivalent to the condition that $R$ is a $d$-free subset of $Q$. The point of the following theorem is that this remains true if $\alpha$ is not injective.

\begin{theorem}\label{tsal}
The pair  $(S;\alpha)$ is  $d$-free with respect to $Q$ if and only if $R=S^\alpha$ is a $d$-free subset of $Q$.
\end{theorem}

The proof is given in \cite{BC1} and \cite[Section 4.2]{FIbook}. In fact, some  more general versions of the theorem are given therein, but our simplified
version is sufficient for most applications. The FI theory is often applicable to the problems  of determining the forms of functions $\alpha$ between rings that share some properties with homomorphisms, and then  Theorem \ref{tsal} is of crucial importance.

\subsection{Quasi-polynomials}
One can handle more general FIs on $d$-free sets than those occurring in the definition. This is what we will show in this subsection. 
We will actually present only a small extract from the theory established in \cite{BC2} and \cite[Chapter 4]{FIbook}, focusing only on a couple of results that have turned out to be extremely useful in applications.

Our general setting is the same as in the preceding subsection. In particular, we are given a function
$\alpha:S\to Q$ which will appear in the FIs under consideration. 

Let us introduce the additional notation. For each
$i\in\{1,\dots,m\} $, we define 
$X_i:S^m\to Q$ by
$$X_i(\ov{x}_m)= x_i^\alpha, $$
and
for distinct $i_1,\dots,i_p\in \{1,\dots,m\}$,
we define a {\em monomial function} $M:S^m\to Q$ as the pointwise product
$$M=X_{i_1}\cdots X_{i_p}, $$
that is,
$$M(\ov{x}_m)= x_{i_1}^\alpha\cdots x_{i_p}^\alpha$$
for all $\ov{x}_m\in S^m$.
We write dom$(M)=\{i_1,\dots,i_p\}$ and  $\deg(M)=p$. The
 constant function 
$M(\ov{x}_m)=1$ is considered 
a monomial 
function with
dom$(M)=\emptyset$ and $\deg(M)=0$.

Fix $n < m $. 
Let $M$ and $N$ be monomial functions such that
$${\rm dom}(M)\cap {\rm dom}(N) =\emptyset\mbox{ \,\, and \,\, }
|M| + |N|=m-n,$$ and let 
$j_1 < \dots < j_n$ be such that the sets
$\{j_1,\dots,j_n\}$, 
 ${\rm dom}(M) $, and
${\rm dom}(N)  $ form a partition of $\{1,\dots,m\}$. Further, given a function $F_{M,N}:S^n\to Q$,  we define 
$MF_{M,N}N:S^m\to Q$ by
$$(MF_{M,N}N)(\ov{x}_m)
= M(\ov{x}_m) F_{M,N}(x_{j_1},\dots, x_{j_n}) N(\ov{x}_m).$$
A sum of such functions,
\begin{equation}
\label{ecore}
\sum_{M,N}  MF_{M,N}N,\end{equation}
is called a {\em core function}. 

\begin{example}\label{exco} Let $m=5$,
$n=2$, and let 
$F,G,H,K:S^2\to Q$.
The function that sends $\ov{x}_5$ to
$$x_1^\alpha x_4^\alpha F(x_2,x_5) x_3^\alpha +
x_4^\alpha x_1^\alpha G(x_2,x_5) x_3^\alpha+
 x_5^\alpha H(x_1,x_4) x_3^\alpha x_2^\alpha  +
 K(x_3,x_5) x_4^\alpha x_1^\alpha x_2^\alpha
$$
is a core function (with $F_{X_1X_4, X_3}=F$, $F_{X_4X_1, X_3}=G$, $F_{X_5,X_3 X_2}=H$, and $F_{1,X_4 X_1X_2}=K$).
\end{example}

\begin{example}\label{exco2}
The  FIs \eqref{3S1al} and \eqref{3S2al} can  be presented in terms of core functions. 
Indeed, we can write \eqref{3S1al} as
$$\sum_{i\in I} E_{1,X_i}X_i + 
\sum_{i\in I} X_jF_{X_j,1} =0.$$
The left-hand sides of  \eqref{3S1al} and \eqref{3S2al} are thus core functions corresponding to the case where $n=m-1$. \end{example}

The functions 
$F_{M, N}$ in \eqref{ecore} are called the {\em middle functions}. We say that a middle function $F_{M_0,N_0}$ is a {\em rightmost middle function} if $F_{M,N} =0$ whenever
$\deg(N) < \deg(N_0)$. Analogously we define a  {\em leftmost middle function}.

\begin{example}
The rightmost middle functions in the core function from Example  \ref{exco} are $F$ and $G$, and $K$ is the leftmost middle function.
\end{example}

\begin{example}
In Example  \ref{exco2}, 
the rightmost middle functions  are $F_{X_j,1}$
and the leftmost functions are
$E_{1,X_i}$.
\end{example}

We continue by introducing {\em quasi-polynomials},  which are of crucial importance in the FI theory. Informally, they are defined similarly as core functions, just that $n$ is not fixed and the middle functions are assumed to have values in the center $C$ of $Q$ (so they can  be written as leftmost middle functions and there is no need to involve two monomial functions $M$ and $N$ in each term but only one). 
Let us start with small $m$, for which the definition can be given in a straightforward manner. 

\begin{example}\label{exqp}
(1) If $m=1$, a quasi-polynomial
is a function $P:S\to Q$ of the form
$$P(x_1)=\lambda x_1^\alpha +\mu(x_1)$$
where $\lambda\in C$ and
$\mu:S \to C$.

(2) If $m=2$, a quasi-polynomial
is a function $P:S^2\to Q$ of the form
$$P(x_1,x_2)=\lambda_1 x_1^\alpha x_2^\alpha +
\lambda_2 x_2^\alpha x_1^\alpha+\mu_1(x_2)x_1^\alpha + \mu_2(x_1)x_2^\alpha +\nu(x_1,x_2) $$
where $\lambda_1,\lambda_2\in C$,
$\mu_1,\,\mu_2:S \to C$, and $\nu:S^2\to C$.

(3) If $m=3$,  a quasi-polynomial
is a function $P:S^3\to Q$ of the form
\begin{align*}P(x_1,x_2,x_3)=& \sum_{\sigma\in S_3}\lambda_{\sigma} x_{\sigma(1)}^\alpha x_{\sigma(2)}^\alpha x_{\sigma(3)}^\alpha + 
\sum_{\sigma\in S_3}\mu_{\sigma}(x_{\sigma(1)}) x_{\sigma(2)}^\alpha x_{\sigma(3)}^\alpha\\&+
\nu_1(x_2,x_3)x_1^\alpha + \nu_2(x_1,x_3)x_2^\alpha +
\nu_3(x_1,x_2)x_3^\alpha +\omega(x_1,x_2,x_3)\end{align*}
where $\lambda_{\sigma}\in C$, $\mu_\sigma:S\to C$,
$\nu_i: S^2\to C$, and
$\omega:S^3\to C$.
\end{example}

Observe that, using the  notation introduced above, $P$ in (1) can be written as
$$P=\lambda X_1 + \mu,$$ and $P$ in (2) can be written as
$$P=\lambda_1 X_1X_2 + \lambda_2 X_2X_1 + \mu_1 X_1+\mu_2 X_2 + \nu. $$
 Here, it should be understood that, for example, $\mu_1 X_1$ 
 is the function given
 by $(x_1,x_2)\mapsto \mu_1(x_2)x_1^\alpha$.
 
 Now, for an arbitrary $m$,
 we define a quasi-polynomial as a function of the form
 $$P=\sum_M \lambda_M M$$
 where the summation runs over all monomial functions
 $M=X_{i_1}\dots X_{i_p}$ with $p\le m$ and $\lambda_M$ is a function from $S^{m-p}$  to $C$; more precisely, $ \lambda_M M$ is defined by
 $$\ov{x}_m\mapsto \lambda_M(x_{j_1},\dots, x_{j_{m-p}}) x_{i_1}^\alpha\dots x_{i_p}^\alpha$$
 where $$\{j_1,\dots,j_{m-p}\} =
 \{1,\dots,m\}\setminus{\{i_1,\dots,i_p\}}$$ and $$j_1 < \dots < j_{m-p}$$ (if $p=m$, $\lambda_M$ is an element in $C$). The central-valued functions $\lambda_M$  are called the {\em coefficients} of the quasi-polynomial $P$.
 The coefficient $\lambda_1$ is called the {\em central coefficient} (for example,
 the central coefficient in Example \ref{exqp} is $\mu$ if $m=1$, $\nu$ if $m=2$, and $\omega$ if $m=3$). Note that the FI  \eqref{3S2al} can be written as
$$\sum_{i\in I} E_{1,X_i}X_i + 
\sum_{i\in I} X_jF_{X_j,1} =\lambda_1$$
 where $\lambda_1$ is a quasi-polynomial consisting only of its central coefficient.
 
 The following simple lemma is one of the most frequently used results 
in applications of the FI theory. 

\begin{lemma}\label{lfreq}
Let $P=\sum_M \lambda_M M:S^m\to Q$ be a quasi-polynomial. Set $R=S^\alpha$ and assume that one of the following two conditions is fulfilled:
\begin{enumerate}
    \item[{\rm (a)}] $R$ is an $(m+1)$-free subset of $Q$, or
    \item[{\rm (b)}] $\lambda_1=0$ and $R$ is an $m$-free subset of $Q$.
\end{enumerate}
Then $P=0$ only if each $\lambda_M =0$.
\end{lemma}

The proof is very easy. For example, if $m=2$ and 
$P$ is as in Example  \ref{exqp}\,(2), then
$P=0$ can be written as
$$\big(\lambda_2x_2^\alpha +\mu_1(x_2)\big) x_1^\alpha 
+ \big(\lambda_1x_1^\alpha +\mu_2(x_1)\big) x_2^\alpha =-\nu(x_1,x_2)\in C. $$
We may now use Theorem \ref{tsal} to conclude that if either (a) $R$ is 
a $3$-free subset of $Q$ or (b) $\nu=0$ and $R$ is 
a $2$-free subset of $Q$, then $\lambda_2x_2^\alpha +\mu_1(x_2)=0$ and
$\lambda_1x_1^\alpha +\mu_2(x_1)=0$. The same assumptions further imply  that $\lambda_2 = \mu_1=0$ and $\lambda_1 = \mu_2=0$.  The general case can be handled by induction on $m$ in much the same way.

 What can be said when a core function $\sum MF_{M,N}N$ is equal to a quasi-polynomial $P$? This is 
the basic problem in the study of  quasi-polynomials.
We will state only one  theorem about it, which has turned out to be the most useful and applicable one. 

In the statement of the theorem we will use the notation introduced above. 

\begin{theorem}\label{thcorquasi}Let $\sum_{M,N}  MF_{M,N}N:S^m\to Q$ be a core function (where
$F_{M,N}:S^n\to Q$, $n< m$) such that for every middle function $F_{M,N}$ there exist an element
$c\in C$, a leftmost middle function $F_{M_0,N_0}$, and a permutation $\sigma\in S_n$ such that
$$F_{M,N}(x_1,\dots,x_n)=
c F_{M_0,N_0}(x_{\sigma(1)},\dots,x_{\sigma(n)})$$
for all $x_1,\dots,x_n\in S$.
Suppose there exists a quasi-polynomial $P$ with central coefficient $\lambda_1$ such that
$$\sum_{M,N}  MF_{M,N}N =P.$$
Set $R=S^\alpha$ and assume that one of the following two conditions is fulfilled:
\begin{enumerate}
    \item[{\rm (a)}] $R$ is an $(m+1)$-free subset of $Q$,
    \item[{\rm (b)}] $\lambda_1=0$ and $R$ is an $m$-free subset of $Q$.
    \end{enumerate}
    Then all the middle functions
    $F_{M,N}$ are quasi-polynomials.
\end{theorem}

In most applications, all middle functions  differ only up to sign and permutation of variables. 
We also remark that of course one can replace ``leftmost" by ``rightmost" in the statement.

Some instances where Theorem \ref{thcorquasi} is applicable will be given in the next section. Let us, at this point, only mention the following.

\begin{remark}
 The theorem can be used to obtain the description of a commuting trace of a biadditive function $F(x)=B(x,x)$, as presented in Example \ref{ex8}. Indeed, we can regard \eqref{b22b} as a core function (corresponding to the case where $\alpha$ is the identity function), and so the theorem implies that the middle function $(x,y)\mapsto B(x,y)+B(y,x)$ is a quasi-polynomial, provided that $R$
 is a  $3$-free subset. Assuming that $R$ is a prime ring with $\deg(R)\ne 2$ and char$(R)\ne 2$, it thus follows from Theorem \ref{tbei} that $F(x)$ is of the form \eqref{b22a}. The technical assumption that $\deg(R)\ne 2$ is actually redundant, but  
 proving  this requires different techniques (see
\cite{BS4, BSp, LLWW}).

In  essentially the same way,  Theorem \ref{thcorquasi} is also applicable to the problem of describing commuting traces of $n$-additive functions for any $n$, as well as to a more general problem where the trace of an $n$-additive function $F(x)=B(x,\dots,x)$ satisfies $$[F(x),x^\alpha]=0$$
for all $x\in S$, provided of course that 
the set $S$ is  an additive group and 
the given function $\alpha$ is   additive.\end{remark}

\subsection{Generalized functional identities} A {\em generalized polynomial identity} (GPI for short) is an important generalization
of  a polynomial identity (PI). Roughly speaking, unlike PIs, GPIs involve fixed elements from rings. 

\begin{example}
Let $V$ be a vector space over a field $F$ and let $R= {\rm End}_F(V)$, the ring of all endomorphisms of $V$. Take any $a\in R$ that has rank one. Observe that for  each $x\in R$ there exists a $\lambda_x\in F$ such that $axa=\lambda_x a$. Observe that
$$axaya=ayaxa$$
for all $x,y\in R$ (indeed, both sides are equal to $\lambda_x\lambda_ya$). This is a simple example of a GPI. We   remark that $R$ satisfies no nonzero PI if $V$ is infinite-dimensional.
\end{example}

In the study of GPIs we often confine ourselves to prime rings. A GPI of a prime ring is defined as an element 
$f(X_1,\dots,X_n)$ of $ Q_s(R)\ast C\langle X_1,X_2,\dots\rangle,   $
the free product of the symmetric Martindale ring of quotients $Q_s(R)$ and the free algebra
$C\langle X_1,X_2,\dots\rangle $ over the extended centroid $C$, such that
$f(a_1,\dots,a_n)=0$
for all $a_i\in R$. If $R$ satisfies a nonzero GPI, then it is called a GPI-{\em ring}. 

We remark that
applying the standard multilinearization process one easily shows that a prime GPI-ring satisfies a multilinear GPI, meaning that
\begin{equation}\label{eqgpi2}
\sum_{\sigma\in S_n}  \sum_{i=1}^{n_\sigma} a_{0i}^\sigma x_{\sigma(1)}a_{1i}^\sigma x_{\sigma(2)}a_{2i}^\sigma\dots a_{n-1,i}^\sigma x_{\sigma(n)}a_{ni}^\sigma =0
\end{equation}
for all $x_1,\dots,x_n\in R$ and some fixed  $a_{ji}^\sigma\in Q_s(R)$, and there exists a permutation $\sigma$ such that
$$\sum_{i=1}^{n_\sigma} a_{0i}^\sigma x_{\sigma(1)}a_{1i}^\sigma x_{\sigma(2)}a_{2i}^\sigma\dots a_{n-1,i}^\sigma x_{\sigma(n)}a_{ni}^\sigma\ne 0$$
for some $x_1,\dots,x_n\in R$. 

In the seminal paper
\cite{Amitsur}, Amitsur characterized primitive GPI-rings \cite{Amitsur}.  A few years later, Martindale
\cite{Martindale} generalized  Amitsur's 
result to prime rings  as follows: A nonzero prime ring $R$ is a 
GPI-ring if and only if $A=RC$, the $C$-subalgebra of $Q_s(R)$ generated by $R$ (called the {\em central closure} of $R$), is a primitive ring containing an idempotent $e$ such that $Ae$ is a minimal left ideal of $A$ and $eAe$ is a finite-dimensional division algebra over $C$. For more information on such rings  see \cite[Section 4.3]{BMMb} or \cite[Sections 5.4 and 7.7]{INCA}.

We refer to \cite{BMMb} for a full account of the theory of generalized polynomial identities.  Let us now turn to {\em generalized  
functional identities} (GFIs) of prime rings.
They are defined in such a way that multilinear GPIs and  the fundamental FIs \eqref{3S1} and \eqref{3S2} are their special cases. Let us give this definition. 

Let $R$ be a prime ring and let $Q=Q_{ml}(R)$. As usual, 
let $m$ be a positive integer and let $I$, $J$ be subsets of $\{1,2,\ldots,m\}$.
Further, for each $i\in I$, let
$s_i$  be a positive integer and let
$\{a_{i1},a_{i2},\ldots,a_{is_i}\}$
be a subset of $Q$ that is linearly independent over the extended centroid $C$. Similarly, for each 
 $j\in J$, let $t_j$
  be a positive integer and let
  $\{b_{j1},b_{j2}\ldots,b_{j t_j}\}$ be a linearly independent subset of $Q$. 
Finally, 
let $V$ be  a finite-dimensional vector subspace of $Q$ and let
$E_{iu},F_{jv}:R^{m-1}\to Q$ be functions. The following is the basic GFI:
\begin{equation}\label{5SJ2}
\sum_{i\in I}\sum_{u=1}^{s_{i}}E_{iu} (\ov{x}_m^i)x_ia_{iu} +
\sum_{j\in J}\sum_{v=1}^{t_j} b_{jv}x_j F_{jv}(\ov{x}_m^j) \in V
\end{equation}
for all $\ov{x}_m \in R^m$. Observe that if 
$s_i=t_j=1$,
$a_{i1}=b_{j1}=1$ for all $i\in I$, $j\in J$, and $V=\{0\}$ (resp. $V=C$), \eqref{5SJ2} becomes \eqref{3S1} (resp. \eqref{3S2}).

A {\em standard solution} of \eqref{5SJ2}
 is defined as 
follows: there exist
functions
\begin{align*}
&\mbox{$p_{iujv}:R^{m-2}\to Q$, $i\in I$, $j\in J$, $1\le
u\le s_i$,
$1\le v \le t_j$,}\\
&\mbox{$\lambda_{kuv}:R^{m-1}\to C$, $k\in I\cup J$, $1\le
u\le s_i$, $1\le v \le t_j$,}
\end{align*} such that
\begin{eqnarray}\label{5SJ3}
E_{iu}(\ov{x}_m^i) &=& \sum_{j\in J\atop j\not=i}\sum_{v=1}^{t_j}
b_{jv}x_jp_{iujv}(\ov{x}_m^{ij}) +
\sum_{v=1}^{t_{i}}\lambda_{iuv}(\ov{x}_m^i)b_{iv},
\nonumber \\
F_{jv}(\ov{x}_m^j) &=& -\sum_{i\in I\atop i\not=j}\sum_{u=1}^{s_i}
p_{iujv}(\ov{x}_m^{ij})x_{i}a_{iu} -
\sum_{u=1}^{s_{j}}\lambda_{juv}(\ov{x}_m^j)a_{ju}, \\
& & \lambda_{kuv} = 0\quad\mbox{if}\quad k\not\in I\cap
J.\nonumber
\end{eqnarray}
One can easily check that \eqref{5SJ3} implies
\eqref{5SJ2} (with $V=\{0\}$).

The following theorem
is due to Chebotar \cite{C2} (see also \cite[Section 5.5]{FIbook}) who generalized an earlier result of the author \cite{Bfideg2} in which the case where $m=2$ was treated. 

\begin{theorem}
Let $R$ be a prime ring. If the GFI \eqref{5SJ2} has nonstandard solutions, then $R$ is a GPI-ring.
\end{theorem}

See also \cite{BBC} for a more general result that also involves automorphisms, antiautomorphisms, and derivations.

Another generalization of ordinary FIs, adjusted to superalgebras, was proposed by Wang in \cite{Wsuper}. Therein, the {\em $d$-superfree} subsets of superalgebras were introduced and some results analogous to those on  $d$-freeness of prime rings were established.

\subsection{Nonstandard solutions of 
FIs in one variable}
 Recall from Corollary \ref{ctbei2} that a central simple algebra of dimension less than $d^2$, in particular the matrix algebra $M_n(F)$ with $n < d$, is not $d$-free. In Example \ref{ex4} we saw how to derive a nonstandard solution from the Cayley-Hamilton identity. {\em Are all nonstandard solutions consequences of the Cayley-Hamilton identity?}
 
 This question was considered in the more recent years in
 our papers with Špenko  \cite{BSp, BSp2} and  Procesi and Špenko \cite{BPS}. In the last three subsections, we will briefly survey the main results, without explaining the methods of proofs.
 Let us just say that
 they are essentially different from the methods used elsewhere in the FI theory.

A rough summary of the paper \cite{BSp} is that our question has a positive answer for FIs in one variable. To state a simplified version of the main theorem, we first have to explain what we mean by the Cayley-Hamilton identity in a (finite-dimensional) central simple $F$-algebra $A$. 

  We first define tr$(x)$, the trace of the element $x$ in $A$, as the trace of the matrix
$1\otimes x \in K\otimes A \cong M_n(K)$ where $K$ is a splitting field for $A$. It is well known that this definition is independent of the choice of $K$
and that tr$(x)$ lies in  $F$.  
Since the coefficients
of the characteristic polynomial of a matrix can be expressed by the traces of its
powers,
we can therefore define adj$(x)\in A$,
 the {\em adjugate} of $x$, in a self-explanatory way. Also, we can naturally define $\det(x)$, the {\em determinant} of $x$, and just as for matrices  we obtain the {\em Cayley-Hamilton identity}
$$ x\, {\rm adj}(x) = \det(x) \in F$$
for every $x\in A$. Repeating the argument from Example \ref{ex4} we see that this  yields a nonstandard solution of the FI of the type
$\sum_{j\in J}x_jF_j(\ov{x}_m^j) \in C$.

The following is a special case of the main theorem of \cite{BSp}.

\begin{theorem}\label{MT} Let $F$ be a field with {\rm char}$(F)=0$, let $A$ be
a central simple $F$-algebra, let
$q_0,q_1,\ldots,q_r:A\to A$ be   traces of $d$-linear functions, and let
$q:A\to A$ be given by
$$q(x)=q_0(x)x^{r} + xq_1(x)x^{r-1} + \cdots + x^{r}q_{r}(x).$$
Assume that $q(x)\in F$ for all $x\in A$.
Then there exist  traces of $(d-1)$-linear functions $p_0,p_1,\ldots,p_{r-1}:A\to A$ and  traces of $d$-linear functions $\mu_0,\mu_1,\ldots,\mu_{r-1}:A\to F$ such that
\begin{align*}
q_0(x) &= x p_0(x) + \mu_0(x),\\
q_i(x) &= -p_{i-1}(x)x + xp_i(x) + \mu_i(x),\,\,\,i=1,\ldots,r-1,
\end{align*}
for all $x\in A$. Moreover, if
$q(x)=0 $ for all  $x\in A$,
 then  
$$
q_r(x) = -p_{r-1}(x)x  -  \sum_{i=0}^{r-1} \mu_i(x)
$$
for all $x\in A$, and 
if  $q(x)\ne 0$ for some $x\in A$, then  $\dim A = n^2$ with  $r(n-1)\le d$ and there exists the trace of a  $(d-r(n-1))$-linear function $\lambda :A\to F$ such that $\lambda\ne 0$ and
$$q_r(x)=\lambda(x){\rm adj}(x^r)-p_{r-1}(x)x  - \sum_{i=0}^{r-1} \mu_i(x)$$ 
for all $x\in A$. 
\end{theorem}

Note that in the special case where
$r=1$, $q_0=-q_r$, and $q=0$, the condition of the theorem can be stated as that $q_0$ is a commuting trace of a multilinear function (see Examples \ref{ex6} and \ref{ex8}). It can be shown that in this case 
$q_0$ is of standard form, that is,  $q_0(x) =\sum_{i=0}^{d} \lambda_i(x)x^i$ where each $\lambda_i$ is the trace of a  $(d-i)$-linear map from $A$ to $F$. For infinite-dimensional algebras and algebras of sufficiently large dimensions this was proved earlier in 
\cite{LLWW}. We also mention the  papers \cite{Leecomm, LiuCK, LiuPu} in which   certain important special cases are treated in more general settings.

\subsection{Nonstandard solutions of quasi-identities}
In this subsection, we will give an overview of the  paper \cite{BPS}. It
considers only the case where $A=M_n(F)$, but is, nevertheless,   technically quite complex. We assume throughout this subsection that char$(F)=0$.

We begin by  modifying the definition of a quasi-polynomial and adjust it to the $n\times n$ matrix case. 
By a {\em quasi-polynomial} we will now mean a formal expression 
 $$P=\sum_M \lambda_M M$$
 where $M$ is a monomial
 in noncommuting indeterminates $X_1, X_2,\dots$, and 
 $\lambda_M$ is an ordinary
polynomial in  commutative indeterminates
$x_{ij}^{(1)}, x_{ij}^{(2)},\dots$, where $1\le i,j\le n$.
More precisely, writing
$$\mathcal C = F[x_{ij}^{(k)}\,|\,
1 \le i, j \le  n, k = 1, 2,\dots],$$ our quasi-polynomials are
elements of the free $\mathcal C$-algebra 
 $\mathcal C\langle X_1,X_2,\dots\rangle $.
 
 Further,  we define 
{\em generic matrices} $\xi_k$  
as matrices in  $M_n(\mathcal C)$  
whose entries are  $x_{ij}^{(k)}$. Define
$\Phi:\mathcal C\langle X\rangle \to M_n(\mathcal C)$ by
$$\Phi\Big(\sum \lambda_{X_{i_1}\dots X_{i_m}} X_{i_1}\dots X_{i_m}\Big)=   \lambda_{X_{i_1}\dots X_{i_m}} \xi_{i_1}\dots \xi_{i_m}. $$
For each $P\in\mathcal  C\langle X_1,X_2,\dots\rangle$,  let $\Phi(P)_{ij}$ be the $(i,j)$ entry of 
$\Phi(P)\in M_n(\mathcal C)$. By a {\em substitution} in 
$C\langle X_1,X_2,\dots\rangle$ we mean  that one replaces
$X_k$
by some $P_k\in C\langle X_1,X_2,\dots\rangle$ and simultaneously $x_{ij}^{(k)}$ by $\Phi(P_k)_{ij}$.  By a {\em T-ideal} of $C\langle X_1,X_2,\dots\rangle$ we mean an ideal that is closed under all such substitutions.

We
define the {\em evaluation} of a quasi-polynomial $P =P(X_1,\dots,X_m)$ at an $m$-tuple $a_1,\dots, a_m\in M_n(F)$, denoted $P(a_1,\dots,a_m)$, by substituting $a_k$ for $X_k$
and $a_{ij}^{(k)}$ for
 $x_{ij}^{(k)}$  where
 $a_k=  (a_{ij}^{(k)})$.
If $P(a_1,\dots,a_m)=0$
for all $a_1,\dots,a_m\in M_n(F)$, then we say that
$P$ is a {\em quasi-identity} of $M_n(F)$. The set of all quasi-identities is clearly a $T$-ideal of 
$C\langle X_1,X_2,\dots\rangle$.

As a side remark, we mention that quasi-identities are related to locally linearly dependent noncommutative polynomials considered in \cite{BWar, BKl}.

The fundamental example of a quasi-identity
 of $M_n(F)$ is the {\em Cayley-Hamilton polynomial} 
 \begin{equation}\label{CHp} Q_n = Q_n(X_1)=  X_1^n +\tau_1(X_1)X_1^{n-1} +\dots + \tau_{n-1}(X_1)X_1 + \tau_n(X_1),\end{equation}
 where the commutative polynomials 
$\tau_i(X_1)$ can be expressed as  $\mathbb Q$-linear  combinations of the products of tr$(\xi_1^j)$ (e.g.,
$Q_2= X_1^2 - {\rm tr}(\xi_1)X_1 +\frac{1}{2} ({\rm tr}(\xi_1)^2 - {\rm tr}(\xi_1^2)$).
We say that a quasi-identity $P$ of $M_n(F)$ is a {\em consequence of the Cayley–Hamilton identity} if $P$ lies in the T-ideal of $C\langle X_1,X_2,\dots\rangle$
generated by the Cayley–Hamilton polynomial $Q_n$.

One can now ask whether every  quasi-identity of $M_n(F)$ is a consequence of the Cayley–Hamilton identity. The basic motivation for this question is the famous result, obtained  independently by Procesi \cite{Procesi} and Razmyslov \cite{Raz}, which states  that the answer to such a question
is positive for trace identities, and so in particular for polynomial identities (that is,
the T-ideal of trace identities of $M_n(F)$ is generated
by the Cayley-Hamilton polynomial,  see also \cite[p.\ 444]{Probook}).

The answer to our question, however, is negative in general. Thus, the following is true.

\begin{theorem}\label{tbcp}
Not every 
quasi-identity   of $M_n(F)$, $n\ge 2$, is a   consequence of the Cayley–Hamilton identity.
\end{theorem}

This theorem  is  just the main message of \cite{BPS}. The bulk of the paper is actually devoted to the description of a certain family of quasi-identities that does not lie in the T-ideal generated by $Q_n$.

Let us state one more result from
\cite{BPS} which shows that, on the other hand, the
quasi-identities of $M_n(F)$ are not too far from the Cayley–Hamilton identity. Recall that a {\em central polynomial} of $M_n(F)$ is a noncommutative polynomial
$c=c(X_1,\dots, X_m)$ with zero constant term  such that
$c(a_1,\dots,a_m)$ is a scalar matrix for all $a_1,\dots,a_m\in M_n(F)$ and $c$ is not a polynomial identity. 

\begin{theorem}
Let $P$ be a quasi-identity of $M_n(F)$. For every central polynomial $c$ of $M_n(F)$
there exists a positive integer
$k$ such that
$c^kP$ is a consequence of the
Cayley–Hamilton identity.
\end{theorem}

Theorem \ref{tbcp} shows that quasi-identities present a genuinely new type   of identities of matrix algebras. They are still far from being fully understood.

\subsection{Nonstandard solutions of FIs when regarded as GPIs}
In this last subsection, we will present the main results of \cite{BSp2}. They consider nonstandard solutions of the basic FI
\begin{equation}\label{eqxz}\sum_{i\in I}E_i(\ov{x}_m^i)x_i+ \sum_{j\in J}x_j F_j(\ov{x}_m^j)
 = 
0\end{equation} on $M_n(F)$,
with $m,I,J$ having the usual meaning,
and
  the functions $E_i,F_j$  assumed to be multilinear. It is easy to see that every multilinear function $F:M_n(F)^{m-1}\to M_n(F)$ is a sum of functions of the form
$$(x_1,\dots,x_{m-1})\mapsto a_{i_0}x_{j_1}a_{i_1}\dots a_{i_{m-2}}x_{j_{m-1}}a_{i_{m-1}}$$ where $a_{i_u}\in M_n(F)$ and $\{j_1,\dots,j_{m-1}\} = \{1,\dots,m-1\}$. Therefore, there is no loss of generality in assuming that $E_i$ and $ F_j$ are (multilinear) generalized polynomials, i.e., elements of the free product $M_n(F)\ast F\langle X_1,X_2,\dots\rangle$. Accordingly,  \eqref{eqxz} can   be viewed as a multilinear GPI.

We will state two theorems. The first one concerns the situation where $J$ in \eqref{eqxz} is $\emptyset$, and the second one the general situation.

We will say that  $C\in M_n(F)\ast F\langle X_1,X_2,\dots\rangle$ is a {\em  central generalized polynomial} (on $M_n(F)$) if all its evaluations on $M_n(F)$
are scalar matrices. An example that immediately presents itself is obtained
by taking
the Cayley-Hamilton polynomial $Q_n(X_1)$ as given in \eqref{CHp} and subtracting the central part $\tau_n(X_1)$.  
Denote by $$\widetilde{Q}_n=\widetilde{Q}_n(X_1,\dots,X_n)$$ the quasi-polynomial obtained by the complete linearization of $Q_n(X_1)-\tau_n(X_1)$. So, for example, 
$$\widetilde{Q}_2(X_1,X_2) = X_1X_2+X_2X_2 + \tau_1(X_1)X_2 + \tau_1(X_2)X_1$$
(where $\tau_1(X_i)= -{\rm tr}(\xi_i)$). 

Since  $Q_n(X_1)-\tau_n(X_1)$ is a central generalized polynomial, so is
$\widetilde{Q}_n(X_1,\dots,X_n)$. Therefore,
for any matrices
$a_1,\dots,a_{n+1}\in M_n(F)$, we have
\begin{equation}[\widetilde{Q}_n(a_1x_1,\dots,a_nx_n),a_{n+1}x_{n+1}]=0 \label{follow}  \end{equation}
for all $x_1,\dots,x_{n+1}\in M_n(F)$. Note that this is an FI of the form \eqref{eqxz} with $J=\emptyset$. Our first theorem states that every FI of the form \eqref{eqxz} follows from those of type \eqref{follow}.

\begin{theorem}\label{FIGPI}
Let $E_i \in  M_n(F)\ast F\langle X_1,X_2,\dots\rangle$, $i\in I$,  be multilinear generalized polynomials such that $P=\sum_{i\in I} E_i(\overline{X}_m^i) X_i$ is a GPI of $M_n(F)$. Then $P$ can be written as a sum of GPIs of the form  
$$C \cdot  \big[ \widetilde{Q}_n(a_{1}X_{k_1},\dots,a_{n}X_{k_n}),a_{n+1}X_{k_{n+1}}\big],
$$
where $a_{i}\in M_n(F)$, $k_i\ne k_j$ if $i\ne j$,  and $C$ is a multilinear central generalized polynomial (in indeterminates from $\{X_1,\dots,X_m\}\setminus{\{X_{k_1},\dots,X_{k_{n+1}}\}}$).
\end{theorem}

A similar result of course holds 
for generalized polynomials
$\sum_{j\in J}X_j F_j(\overline{X}_m^j)  $, i.e., for FIs \eqref{eqxz} with $I=\emptyset$ and $F_j$ multilinear.  We can therefore handle ``one-sided identities".

Let us turn to  ``two-sided identities".
We say that multilinear functions $E_i,F_j$ form  a {\em standard solution modulo one-sided identities} of \eqref{eqxz} if there exist multilinear functions 
\begin{eqnarray*}
&  & p_{ij}:M_n(F)^{m-2}\to M_n(F),\;\;
i\in I,\; j\in J,\; i\not=j,\nonumber\\
&  & \lambda_k:M_n(F)^{m-1}\to F,\;\; k\in I\cup J,\\
& & \widehat{E}_i:M_n(F)^{m-1}\to M_n(F),\,\, 
i\in I,\; \\
& & \widehat{F}_j:M_n(F)^{m-1}\to M_n(F),\,\, 
j\in J,
\end{eqnarray*}
such that 
\begin{eqnarray*}
E_i(\ov{x}_m^i) & = & \sum_{j\in J,\atop
j\not=i}x_jp_{ij}(\ov{x}_m^{ij})
+\lambda_i(\ov{x}_m^i) + \widehat{E}_i(\ov{x}_m^i),\quad i\in I,\nonumber\\
F_j(\ov{x}_m^j) & = & -\sum_{i\in I,\atop
i\not=j}p_{ij}(\ov{x}_m^{ij})x_i
-\lambda_j(\ov{x}_m^j) + \widehat{F}_j(\ov{x}_m^j) ,\quad j\in J,\\
&   & \lambda_k=0\quad\mbox{if}\quad k\not\in I\cap J,\nonumber
\end{eqnarray*}
and
$$\sum_{i\in I} \widehat{E}_i(\ov{x}_m^i)x_i = \sum_{j\in J} x_j\widehat{F}_j(\ov{x}_m^j) =0$$
for all $\ov{x}_m\in R^m$.

We can now state our last theorem of this section.

\begin{theorem}\label{2fi}
 If $E_i,F_j$ are multilinear functions, then every solution of the FI \eqref{eqxz} on $
 M_n(F)$
 is standard modulo one-sided identities.
\end{theorem}

Loosely speaking, Theorems \ref{FIGPI} and \ref{2fi} show that if we regard  FIs  of the form
\eqref{eqxz} as GPIs, then they all follow from the Cayley-Hamilton identity.

\section{Applications} 

The main reason for the existence
of the general FI theory are its applications. This section is devoted to presenting the
 most important ones.

\subsection{Lie maps}\label{ss41}
The early development of the FI theory is closely connected with the long-standing Herstein's conjectures on Lie maps. We therefore start with a brief survey of  their solutions.

First, 
a few words on the history.  From the 1950s to 1970s,
Herstein and his students systematically studied the Lie (and the Jordan) structure of associative rings. They have answered many natural questions,  but the problems on the structure of Lie homomorphisms and Lie derivations, as posed in Herstein's 1961``AMS Hour Talk"  \cite{Her},
remained unsolved.
More precisely, some of them were solved  by Martindale \cite{Mart1, Mart4, Mart5, Mart6}, however, under the assumption that  rings contain (enough) nontrivial idempotents. The question whether the presence of idempotents can be  removed was open until the discovery of the FIs. Using them, all problems were finally solved in a series of papers that started in 1993 and ended in 2002.

The first breakthrough was made in the author's paper \cite{B5} in which the structure of Lie isomorphisms between prime rings was described. To state it, we first give the definition and describe the problem. 

Let $M$ and $R$ be rings and let $L$ be a Lie subring
of $M$. An   additive map $\alpha:L\to R$, denoted $x\mapsto x^\alpha$, is called a {\em Lie homomorphism} if
\begin{equation}
    [x,y]^\alpha=[x^\alpha,y^\alpha] \label{eLieh}
\end{equation}
for all $x,y\in L$. That is, $\alpha $ is homomorphism from the Lie ring $(L,\,+\,, [\,\cdot\,,\,\cdot\,])$ to the Lie ring $(R,\,+\,, [\,\cdot\,,\,\cdot\,])$. 

Let us for now restrict ourselves to the simplest case where $L=M$. The obvious examples of Lie homomorphisms are then homomorphisms and the negatives of antihomomorphisms (i.e., maps $\alpha$ satisfying
$\alpha(xy)=-\alpha(y)\alpha(x)$
for all $x,y\in M$, so that $-\alpha$ is an antihomomorphism). These, however, are not the only possible examples, at least not when $M\ne [M,M]$, since if $\alpha$ is a Lie homomorphism and $\tau$ is an additive map from $M$ to the center of $R$ that vanishes on all commutators ($\tau([x,y]) =0$ for all $x,y\in M$), then $\alpha+\tau$ is again a Lie homomorphism. The question that can be asked, and was asked by Herstein,  is whether a Lie homomorphism is the sum of a homomorphism or the negative of an antihomomorphism and such a map $ \tau$.

One usually assumes that
$\alpha$ is at least surjective, if not bijective. To describe the approach from \cite{B5}, assume that  $R$ is a prime ring with char$(R)\ne 2$ and that $\alpha$ is bijective. Writing $x^2$ for $y$ in \eqref{eLieh} we obtain
$$[x^\alpha, (x^2)^\alpha] =0$$ for all $x\in M$, and hence 
$$[y,(y^{\alpha^{-1}})^2 ] =0 $$
for all $y\in R$. This can be read as that the function $F:R\to R$ defined by 
$$F(y)= (y^{\alpha^{-1}})^2$$
is a commuting trace of a biadditive function. As we said in Example \ref{ex8}, such a function is of the form
  \begin{equation}
      F(y)= \lambda y^2 + \mu(y) y + \nu(y,y)\label{eclo}
  \end{equation}
  for all $y\in R$, where $\lambda$ is an element
  from the extended centroid $C$, $\mu$ is an additive function from $R$ to $C$, and $\nu$
  is a biadditive function from $R\times R$ to $C$.
 Now, controlling the action of $\alpha$ on squares of elements, and hence on the Jordan products of elements, and simultaneously, by the very definition, controlling the action
 of $\alpha$ on the Lie product, it is not surprising that we have  control of the action of $\alpha$ on the ordinary product $xy = \frac{1}{2}(x\circ y + [x,y])$. In this way one can prove that $\alpha$ is of the expected, aforementioned form.
  
Showing that commuting traces of biadditive functions are of the form
\eqref{eclo} was the fundamental result of \cite{B5} from which all others were derived. 
This was actually established under the additional assumption that $\deg(R)\ne 2$ which
has later turned out to be
unnecessary (this assumption is indeed necessary  to establish that $R$ is a $3$-free subset of $Q_{ms}(R)$, but not for showing that commuting traces of biadditive functions are of standard form).

We made this short overview of the approach taken in \cite{B5} since its main ideas are illustrative and can be easily understood even without being exposed to FIs.  Let us now present in a somewhat greater detail a  more sophisticated approach based on the general FI theory. We follow \cite{FIbook}.

One can combine basic examples of Lie homomorphisms to obtain new ones. We will say that
$\alpha$ is the {\em direct sum} of a homomorphism and the negative of an antihomomorphism if there exists a central idempotent $e$ such that 
$x\mapsto ex^\alpha$ is a homomorphism and  
$x\mapsto (1-e)x^\alpha$ is the negative of an antihomomorphism. This is not relevant in prime rings since they do not contain central idempotents different from $0$ and $1$. However, we will now consider general rings.
Our key assumption is that the image of the Lie homomorphism $\alpha$ is a $3$-free subset.

\begin{theorem}\label{wecan}
 Let $\alpha$ be a Lie homomorphism from a ring $M$ to a unital ring $Q$ with center $C$. If $M^\alpha$ is a $3$-free subset of $Q$, then $\alpha=\varphi+\tau$ where $\alpha:M\to Q$  is the direct sum of a homomorphism and the negative of an antihomomorphism and $\tau:M\to C$ is an additive map which vanishes on commutators. 
\end{theorem}

\noindent
{\em Sketch of proof.}  Observe that $$[xy,z]+ [zx,y] + [yz,x]=0 $$
holds for any $x,y,z\in M$ (this is an ``associative version" of the Jacobi identity). Hence it follows that 
$$[(xy)^\alpha,z^\alpha]+ [(zx)^\alpha,y^\alpha] + [(yz)^\alpha,x^\alpha]=0. $$
Since $R=M^\alpha$ is a $3$-free subset of $Q$, we are in a position to apply Theorem \ref{thcorquasi} for the case where $m=3$, $n=2$, $P=0$, and each $c=\pm 1$. Therefore, $(xy)^\alpha$ is a quasi-polynomial, i.e.,
\begin{equation}
  \label{sube}  
(xy)^\alpha=\lambda_1 x^\alpha y^\alpha +
\lambda_2 y^\alpha x^\alpha+\mu_1(y)x^\alpha + \mu_2(x)y^\alpha +\nu(x,y) \end{equation}
for some $\lambda_1,\lambda_2\in C$, $\mu_1,\mu_2:M\to C$, and $\nu:M^2\to C$. Using this form back in \eqref{sube} it easily follows from Lemma \ref{lfreq} that $\mu_1=\mu_2$. This lemma also implies that
$\mu=\mu_1$ is additive
and $\nu$ is biadditive (just replace $x$ and $y$ by the sum of two elements  and use that $\alpha$ is additive).

Using \eqref{sube}, we can compute $(xyz)^\alpha$ in two different ways, firstly as $((xy)z)^\alpha$ and secondly as $(x(yz))^\alpha$. Comparing both expressions we obtain
\begin{equation}
   \label{ifitwas} 
\lambda_1\lambda_2 [y^\alpha, [x^\alpha,z^\alpha]] + \xi(x,y)z^\alpha -\xi(y,z)x^\alpha \in C, \end{equation}
where $\xi:M^2\to C$ is a function that can be expressed by $\lambda_i,\mu,\nu$. If $R=M^\alpha$ was $4$-free, then Lemma \ref{lfreq} would imply that $\lambda_1\lambda_2=0$. However, we are only assuming that it is $3$-free, so we need another step to reach this conclusion. Observe that by fixing $y$, we can interpret \eqref{ifitwas} as 
$$E_1(z)x^\alpha + E_2(x)z^\alpha + x^\alpha F_1(z) + z^\alpha F_2(x)\in C$$
for suitable functions $E_i,F_j$. Applying Theorem \ref{tsal} it is now easy to see (by using the exact form of $E_i,F_j$) that the $3$-freeness is sufficient for concluding that 
$\lambda_1\lambda_2=0$.

Next, applying \eqref{sube} to
$$(xy)^\alpha - (yx)^\alpha = [x^\alpha,y^\alpha] $$ we arrive at 
$$(1-\lambda_1 + \lambda_2) [x^\alpha,y^\alpha]\in C $$
for all $x,y\in M$, which, again by Lemma \ref{lfreq}, yields 
$1-\lambda_1 + \lambda_2=0$. Along with 
$\lambda_1\lambda_2=0$, this shows that $e=\lambda_1$ is a central idempotent (and $\lambda_2 = -(1-e)$).

We now define
$\varphi:M\to Q$ by
$$x^\varphi= x^\alpha - (1-2e) \mu(x). $$
Using similar methods as above one easily shows that $x\mapsto e x^\varphi$ is a homomorphism, 
$x\mapsto (1-e) x^\varphi$  is the negative of an antihomomorphism, and
$\tau(x)=(1-2e) \mu(x)$ is an additive  map that vanishes on commutators and maps to $C$.
$\hfill \qed$

\bigskip
Let us return to the situation where
 $R=M^\alpha$ is a prime ring. Setting $Q=Q_{ms}(R)$ we see from Theorem \ref{tbei}
 that we
 can apply Theorem   \ref{wecan}, provided that
 $\deg(R)\ge 3$ (and in this way obtain the result from \cite{B5}). The case where $\deg(R)=1$, i.e., $R$ is commutative, is trivial, so we are left with the $\deg(R)=2$ case. The structure of  $R$  is then well known, that is, $R$ can be nicely embedded into the ring of $2\times 2$ matrices over a field.
 On the one hand, this makes the problem  easy, but on the other hand we are no longer in a position to apply the FI machinery. However, as already mentioned above, the form of commuting traces of biadditive functions is the same  even when $\deg(R)=2$ (and, additionally, char$(R) \ne 2$). Using the approach from the beginning of this subsection one can  obtain the following corollary to
 Theorem \ref{wecan}.
 
 \begin{corollary}\label{b51}
 Let $R$ be a noncommutative prime with
 {\rm char}$(R)\ne 2$. If $ \alpha$ is a Lie isomorphism from a ring $M$ onto $R$, then 
 $\alpha =\varphi +\tau$ 
 where $\varphi$ is a homomorphism or the negative of an antihomomorphism from
 $M$ to $R+C$, where $C$ is the extended centroid of $R$, and $\tau:M\to C$ is an additive map which vanishes on commutators.
 \end{corollary}

The conclusion that $\varphi$ maps to $R+C$ (rather than to $R$) may seem strange at first glance, but can be justified---see \cite[Example 6.10]{FIbook}.

The  proof of Corollary \ref{b51} that we outlined nicely represents the way FIs are applied. The general theory solves the problem at a high level of generality, but   does not cover  rings that are close to  algebras of low dimensions.  One is therefore forced  to combine FIs with more classical methods.

Corollary \ref{b51} solves the easiest among Herstein's problems on Lie homomorphisms. The others concern Lie rings of skew elements in rings with involution, and, the most difficult ones, Lie ideals of rings and Lie ideals of skew elements. 

Let now $M$ be a ring with involution and let $K$ be the set of all skew elements in $M$. Since $K$
is a Lie subring of $M$, we can speak about a Lie homomorphism $\alpha$ from $K$ to another ring $R$. The natural question here is whether $\alpha$ can be extended to a homomorphism
from the subring generated by $K$ to $R$ (now there is no need to involve the negatives of antihomomorphisms  since we can compose $\alpha$  with  the negative of the involution $\ast$ which itself is the negative of an antiisomorphism on $M$ and acts as the identity on $K$).

The method of  proof of Theorem \ref{wecan} obviously does not work since $K$ is not an (associative) subring. However, the cube of a skew element is again a skew element, so $\alpha$ satisfies $$[x^\alpha,(x^3)^\alpha]=0$$
for all $x\in K$, and hence, if $\alpha$ is injective, 
\begin{equation}
    \label{ynatri}
[y,(y^{\alpha^{-1}})^3 ] =0 \end{equation}
for all $y\in K^\alpha$. We have thus arrived 
at a commuting trace of a triadditive function. Based on this observation, Beidar, Martindale and Mikhalev \cite{BMM1}
described Lie isomorphisms between skew elements of prime rings with involution. The main idea of their proof was thus essentially the same as 
that of the proof of 
Corollary  \ref{b51},  
but the technical challenges were greater. See also \cite[Theorem 6.15]{FIbook} for an abstract version involving $d$-free sets. We remark that the proof is based on the fact that $xyz+zyx\in K$ whenever $x,y,z\in K$, and so applying $\alpha$ to the
identity 
$$[u, xyz+zyx]+     [z,uxy+yxu]+ [y,zux+xuz] + [x,yzu+uzy] =0    $$
 yields an FI which is somewhat more convenient than \eqref{ynatri}. 
 
After solving Herstein's problems for rings (in 1993) and for skew elements  (in 1994), it was still not clear for several years how to handle Lie isomorphisms of their Lie ideals. As they are not closed under some $n$th powers, the  methods described above do not work. However, Lie ideals 
of a ring $M$  (e.g., $[M,M]$) or of skew elements $K$ (e.g., $[K,K]$) are of special interest since they are often simple as Lie rings. More precisely, the classical Herstein's theorems state that
if $M$ is a simple ring with center $Z$, then
$[M,M]/Z\cap [M,M]$ is a simple Lie ring, unless 
char$(M)=2$ and $\deg(M)=2$, and similarly, if $K$ is the set of skew elements of a simple ring $M$ with involution and center $Z$, then 
$[K,K]/Z\cap [K,K]$ is a simple Lie ring, provided that char$(M)\ne 2$ and
$\deg(M) > 4$  \cite{Her, Her2} (see also \cite{BKSh} for a description of Lie ideals of more general rings). It is therefore more natural to consider Lie ideals modulo the central elements.

The problems on Lie ideals were completely solved in
the early 2000s \cite{Onher1, Onher2, Onher3, BCIV, BCIVh}. The proofs are involved and it is difficult to present their main ideas in a few lines.
Many results from the general FI theory are used in the proof. In fact, 
the problems on Lie ideals had served as a principal motivation for 
some advanced parts of the general theory.

Let us restrict ourselves to the more difficult problem on Lie ideals of skew elements. 
We will state two results from \cite{FIbook}. 
Some notation is needed first. For a unital ring $Q$ with center $C$, we write $\overline{Q}$ for the Lie ring  $Q/C$
and for each  $x\in Q$ we write $\overline{x}$ for $x+ C\in \overline{Q}$. Accordingly, for a subset $R$ of $Q$ we write
$\overline{R} = \{\overline{x}\,|\, x\in R\}$.
By $\langle L\rangle $ we denote the subring generated by the subset $L$. We say that an additive group $G$ {\em admits the operator $\frac{1}{2}$} if $x\mapsto 2x$ is an automorphism of $G$.

\begin{theorem}\label{thk}
Let $K$ be the set of skew elements of a ring $M$ with involution,
 let $L$ be a Lie ideal of $K$, let $Q$ be a unital ring with center $C$, and let $\alpha:L\to \overline{Q}$ be a Lie homomorphism. Suppose that $L$ and $Q$ admit the operator $\frac{1}{2}$ and that $C$ is a direct summand of the additive group $Q$. If there exists a 
 $9$-free subset $R$ of $Q$ such that $\overline{R}=L^\alpha$, then there exists a homomorphism $\varphi:\langle L\rangle \to Q$ such that
 $x^\alpha= \overline{x^\varphi}$ for every $x\in L$.
\end{theorem}

Combining Theorem \ref{thk} with Corollary \ref{leserab} we obtain the following corollary.

\begin{corollary}\label{6Lie1}Let
 $K$ be the set of skew elements of a ring $M$ with involution,
 let $L$ be a Lie ideal of $K$, 
 let $R$ be a prime ring with involution,
let $C$ be the extended centroid of $R$, let $T$ be the set of skew elements of $R$, let $U$ be a noncentral Lie ideal of $T$, and let
$\alpha$ be a Lie homomorphism from $L$ onto $\overline{U} = U/U\cap C$.
Suppose that $L$ admits the operator $\frac{1}{2}$  and that ${\rm char}(R)\not=2$. If  $\deg(R)\ge 21$, then there exists
a homomorphism $\varphi:\langle L\rangle\to C\langle U\rangle  + C$
such that $x^\alpha =\overline{x^\varphi}$ for all $x \in L$. 
\end{corollary}

This corollary solves the most difficult among Herstein's problems on Lie homomorphisms. To be precise, it  does not exactly solve it but reduces it to the case where
$\deg(R)\le 20$. It turns out that some of these low degree cases are really exceptional, that is, counterexamples show that the corollary does not hold for them. A detailed analysis is given in \cite{Onher3}. The proofs therein do not use FIs but  classical methods.

We close this subsection with a  short discussion on Lie derivations. As above, let
 $L$ be a Lie subring of a ring $R$. An additive map $\delta:L\to R$ is called a {\em Lie derivation} if
\begin{equation}
    \label{efld}
[x,y]^\delta =[x^\delta,y] + [x,y^\delta]\end{equation}
for all $x,y\in L$. 
The obvious problem is 
to show that a Lie derivation is close to a derivation.

The study of Lie derivations has always been parallel to the study of Lie homomorphisms. The latter, however, is usually  more  demanding. One of the reasons is that we have no analog of an antihomomorphism among derivation-like maps. It is therefore not surprising that FIs can be used for describing Lie derivations in terms of derivations.

If $L$ is a ring, then by setting $x^2$ for $y$ in \eqref{efld} we obtain
$$[x, (x^2)^\delta - x^\delta x - xx^\delta]=0 $$
for all $x\in L$. Thus, we have again arrived at commuting traces of biadditive functions. Using this observation, one can prove the following.

\begin{theorem} \label{6derC3}
Let $R$ be a prime ring with extended centroid $C$. Then every Lie derivation $\delta:R\to R$ 
is of the form $\delta = d +\tau$ where
 $d$ is a derivation from $R$ to $CR+C$ and  $\tau: R\to C$ is an additive map vanishing on commutators, unless
 $\deg(R) = 2$ and {\rm char}$(R) =2$.
\end{theorem}

This was also proved already in \cite{B5}, however, under the additional assumption that $\deg(R)=2$. Before that, results of this kind were known for rings containing nontrivial idempotents (incidentally, in \cite{Her} Herstein mentioned 
that in an unpublished work Kaplansky proved this for rings 
containing $n\times n$ matrix units with  $n\ge 3$).

Theorem \ref{6derC3} is an analog of Corollary \ref{b51}. One can prove analogs of many other results of Lie homomorphisms. For example, the analog of Theorem \ref{thk} reads as follows.

\begin{theorem}\label{63Tder}
Let $K$ be the set of skew elements of a ring $R$ with involution, let $L$ be  Lie ideal of $K$, let $Q \supseteq R$ be a unital ring with center $C$, and let $\delta:L\to\overline{Q}$ be a Lie derivation. Suppose that $L$ and $Q$ admit the operator $\frac{1}{2}$ and that $C$ is a direct summand of the additive group $Q$.
If $L$ is a $9$-free subset of $Q$, then there exists a derivation $d:\langle L \rangle \to Q$ such that $x^\delta = \overline{x^d}$ for all $x\in L$.
\end{theorem}

Theorem \ref{63Tder}
is not exactly surprising in view of Theorem \ref{thk}. It is perhaps more surprising that it actually follows from Theorem \ref{thk} combined with Theorem  \ref{thomder}.

For more details on applications of FIs to Lie derivations see 
\cite{FIbook} or the original sources \cite{Onher2,  BCIVh, B5, Swain}. 

Let us finally also mention that Lie maps in algebras occurring in functional analysis have also been handled by  combining FIs with analytic methods \cite{Alam, AMb, BVi1, BVi2, BVi3, BCFV, Vil}.

 \subsection{Group gradings} Let $G$ be a group and let $A$ be a nonassociative (= not necessarily associative) algebra over a field $F$.
 We say that $A$ is {\em graded by $G$} if there exist linear subspaces $A_g$, $g\in G$, such that $$A = \oplus_{g\in G} A_g\quad\mbox{and}\quad A_gA_h\subseteq A_{gh}$$ for all $g,h\in G$. 
 
 Gradings of  Lie algebras are particularly interesting; see \cite{EK} for a general reference. 
If  a Lie algebra $L$ is graded by a group $G$,  $L = \oplus_{g\in G} L_g$, then 
the elements from the support of the grading Supp$\,L =\{g\in G\,|\, L_g\ne \{0\}\}$) always commute \cite{PZ}. It is therefore natural to restrict ourselves to gradings by Abelian groups.

 Let $L$ be a Lie subalgebra of an associative algebra $A$. Suppose that $L$ is graded (as a Lie algebra) by a group $G$. Is this grading induced from a grading of $A$, that is, is
 $A$ graded by $G$  and $L_g = A_g\cap L$? Assuming that
 gradings on $A$ can be described, a positive answer to this question implies that gradings on $L$ can be described too.

The approach to this question, proposed in our paper with Bahturin \cite{BB}, is based on the following observation. Writing $H$ for the group algebra $FG$, the linear map 
$\alpha: L\otimes H\to L\otimes H$ given by
$$\alpha(a\otimes h) = \sum_{g\in G} a_g\otimes gh $$ is a Lie automorphism of the Lie subalgebra $L\otimes H$ of $A\otimes H$. If 
$A$ is generated as an algebra by $L$ and
$\alpha$ can be extended to a homomorphism  $\overline{\alpha}$ from 
$A\otimes H$ to itself, then our question has an affirmative answer. Indeed, by defining  $$A_g=\{a\in A\,|\, \overline{\alpha}(a\otimes 1) = a\otimes g\}$$ we obtain a grading of $A$ such that $L_g = A_g\cap L$ for each $g\in G$.

As we saw above, the problem of extending Lie homomorphisms to homomorphism can be solved by means of FIs. The problem that occurs is that the algebra $A\otimes H$ may not be prime even when $A$ has the most favorable properties. One is therefore forced to use abstract results on Lie homomorphisms that involve $d$-free sets. More precisely, the notion of the fractional degree has turned out to be the key to solutions.

Let us state only one theorem which illustrates this line of investigation. We will avoid stating its  general version \cite[Theorem 7.1]{BB} which considers the case where $A$ is a centrally closed prime algebra, but confine ourselves to the corollary treating central simple algebras. Here, by a central simple algebra we mean a not necessarily finite-dimensional simple $F$-algebra such that its centroid is $F$. 

\begin{theorem}\label{thrad}
Let $F$ be a field with {\rm char}$(F)\ne 2$ and let
$A$ be a central simple $F$-algebra with involution such
that $\dim_F(A)\ge 441$.
 Let $K$ be the Lie subalgebra of skew elements in $A$. Suppose that the Lie subalgebra $L=[K,K]$
is graded by an Abelian group $G$, which is finite if $A$ is not unital. Then there exists 
a grading of $A$ by $G$ such that $L_g = A_g\cap L$ for each $g\in G$.
\end{theorem}

Under the assumption of the theorem, the condition  $\dim_F(A)\ge 441$ is equivalent to 
$\deg(A)\ge 21$. The latter condition also appears  in Corollary \ref{6Lie1}. However, as explained above, Corollary \ref{6Lie1} is not directly applicable in the proof of  Theorem \ref{thrad}.
 
A similar result does not hold for  gradings of  the Lie subalgebra $[A,A]$.
Besides gradings
that are induced from  gradings of 
$A$, there are also equally natural gradings that arise from an involution 
of $A$. See \cite{BB} for details.
We remark that as a byproduct of the main results from this paper, 
new proofs and generalizations of known results concerning classical Lie algebras were given.

Similar results on gradings of Jordan subalgebras of associative algebras were obtained in \cite{BBSh}. 

Finally, we mention the paper \cite{BBK} in which similar, but more complex methods are used for describing   gradings 
by  Abelian
groups on simple finitary Lie algebras of linear transformations  on infinite-dimensional vector spaces.

\subsection{Poisson algebras}

A vector space $A$ over a field $F$ is called a {\em
Poisson algebra} if it is endowed with two multiplications,
$\cdot$ and $\{\,\cdot\,,\,\cdot\,\}$, such that
\begin{enumerate}
    \item[(a)] $A$ is an associative algebra under $\cdot$\,,
    \item[(b)] $A$ is a Lie algebra under $\{\,\cdot\,,\,\cdot\,\}$,
    \item[(c)] $\{x\cdot y,z\} = x\cdot \{y,z\}+\{x,z\}\cdot y$ for all $x,y,z\in A$.
\end{enumerate}

A Lie homomorphism from a commutative Poisson algebra  
$A$ to the algebra of linear operators on a Hilbert space is called a 
 {\em Dirac map}.
 The {\em Dirac problem} asks for describing Dirac maps. It
 has been studied since a long time ago \cite{Jo, Sou, St}. 
 
 Using FIs, we can consider a more general version of  the Dirac problem where the Poisson algebra is not necessarily commutative and  the Lie homomorphism does not necessarily map to the algebra of linear operators on a Hilbert space. We already know from Subsection \ref{ss41} that a
 Lie homomorphism from
 a Poisson algebra, just as from any
  associative algebra,  to another associative algebra gives rise to certain FIs. It turns out that using other properties of Poisson algebras one can derive further FIs. Based on these FI, one can prove the following theorem which is due to Beidar and Chebotar and appears as Theorem 8.3 in the book \cite{FIbook}  (and was not published elsewhere).

\begin{theorem}\label{PAT1}
Let $F$ be a field with {\rm char}$(F) \ne  2$,  let $A$ be a
Poisson $F$-algebra over $F$, and let $Q$ be a unital $F$-algebra whose center is $F$. If $\alpha:A\to Q$ is a
linear map such that
$$
\{x,y\}^\alpha=[x^\alpha,y^\alpha]
$$
for all $x,y\in A$ and
$A^\alpha$ is a $4$-free subset of $Q$, then there exist a $\lambda
\in F$, linear maps $\mu_1, \mu_2:A\to F$, and a
bilinear map $\nu:A^2\to F$ such that either
$$
(xy)^\alpha=\lambda x^\alpha y^\alpha+
 \mu_1(x)y^\alpha+\mu_2(y)x^\alpha+\nu(x,y)
$$
or
$$
(xy)^\alpha=\lambda y^\alpha x^\alpha+
 \mu_1(x)y^\alpha+\mu_2(y)x^\alpha+\nu(x,y),
$$
and
$$
\lambda \nu(x,y)= \mu_1(x)\mu_1(y)-\mu_1(xy) = \mu_2(x)\mu_2(y)-\mu_2(xy)
$$
for all $x,y\in A$.
Moreover, if $A$ is a commutative algebra, then $\lambda=0$, $\mu_1=\mu_2$, and $\nu$ is symmetric.
\end{theorem}

\subsection{Lie superhomomorphisms}

Over the years, several authors extended
parts of Herstein's theory of Lie structures in associative rings 
 to superalgebras (see \cite{GS, Lal, LS, Mont, Z} and references therein). It is therefore natural to seek  extensions of the above results on Lie homomorphisms to  superalgebras. It is not immediately clear that FIs are applicable here too.  However, it has turned out that they are, although in a less straightforward manner \cite{BB2, BBS, WangJA, Wsuper2}. We will present only a result from our joint  paper with Bahturin and Špenko \cite{BBS}. The others are similar in nature, but have more complicated statements.

We start with the necessary definitions, notation, and terminology. By an algebra we will mean
a not necessarily associative algebra over a field $F$, for which we assume that char$(F)\ne 2$.
A {\em  superalgebra} is a $\mathbb Z_2$-graded  algebra $A$. This means that $A$ contains linear subspaces $A_0$ and $A_1$ such that $A= A_0\oplus A_1$ and $A_iA_j \subseteq A_{i+j}$
for all $i,j\in \mathbb Z_2$. Observe that the map $\sigma:A\to A$ given by 
$$(a_0+a_1)^\sigma = a_0-a_1$$
for all $a_0\in A_0$ and $a_1\in A_1$ is an automorphism of $A$ such that $\sigma^2 = {\rm id}_A$. 
 Conversely, if $\sigma$ is an automorphism of $A$
  such that $\sigma^2 = {\rm id}_A$, then 
  $A$ is a superalgebra with respect to
$A_0=\{a\in A\,|\, a^\sigma = a\}$ and $A_1=\{a\in A\,|\, a^\sigma = -a\}$. Indeed,  every $a\in A$ can be written as $a = a_0 + a_1$ with  $a_0 = \frac{1}{2}(a + a^\sigma)\in A_0$ and  $a_1 = \frac{1}{2}(a- a^\sigma)\in A_1$.

Elements from $A_0\cup A_1$ are called {\em homogeneous}. More precisely, elements
from $A_0$ are
 {\em homogeneous of degree $0$}
and elements
from $A_1$ are
 {\em homogeneous of degree $1$}.  We also say that elements in $A_0$ are {\em even} elements, and elements in $A_1$ are {\em odd} elements. For a homogeneous element
 $a$ we write $|a|=0$ if $a$ is even, and
  $|a|=1$ if $a$ is odd.
Further, we say that 
 a linear  subspace $V$ of $A$ is  {\em graded} if $V = V_0\oplus V_1$ where $V_i= A_i\cap V$. 
 A linear map $\alpha$ from a graded space $V$ to a graded space $W$ is a {\em graded map} if $V_0^\alpha \subseteq W_0$ and $V_1^\alpha \subseteq W_1$.

An {\em associative superalgebra} is just a superalgebra which is associative as an algebra.
The {\em supercommutator} of   
 homogeneous elements $a$ and $b$ in an associative superalgebra $ A$ is the element
  $$[a,b]_s = ab - (-1)^{|a||b|}ba.$$
 We extend
 $[\,\cdot\,,\,\cdot\,\,]_s$  to $A\times A$ 
 by bilinearity. One can check that
 $A$ endowed with the product  $[\,\cdot\,,\,\cdot\,\,]_s$ is a {\em Lie superalgebra},  which means that
 $$   [a,b]_s=-(-1)^{|a| |b|}[b,a]_s $$
 and
  $$
(-1)^{|a||c|}[[a,b]_s,c]_s + (-1)^{|c||b|}[[c,a]_s,b]_s + (-1)^{|b||a|}[[b,c]_s,a]_s = 0
$$
for all homogeneous elements $a,b,c\in A$.

A graded linear map $\alpha$ from 
an associative superalgebra $B$ to an associative superalgebra $A$ is called a  {\em Lie superhomomorphism} if    $$([a,b]_s)^\alpha = [a^\alpha,b^\alpha]_s$$
for all $a,b\in B$.
The obvious examples are {\em superhomomorphism}, i.e.,
graded linear maps that are algebra homomorphisms in the usual sense, and
{\em superantihomomorphisms}, i.e., graded linear maps $\beta$ satisfying $$ab^\beta = (-1)^{|a||b|}b^\beta a^\beta$$ for all homogeneous elements $a$ and $b$.

Let $A$ be an associative superalgebra and let $Z$ be its center (in the usual sense, that is, the set of all elements in $A$ that commute with any other element).
Observe that $Z$ is a graded subspace.
If $A$ is unital and $Z_0=F$, i.e.,
$Z_0$ consists
of scalar multiples of unity, then
we say that $A$ is a {\em central superalgebra}.

 We say that an
  associative superalgebra $A$ is   {\em simple} if $A^2\ne \{0\}$ and $A$ has no nonzero proper graded ideals.
  Similarly, $A$ is {\em prime} if the product of  two nonzero graded ideals of $A$ is  always nonzero.
  Obviously, 
  if $A$ is simple (resp. prime) as an algebra, then it is  simple (resp. prime) also as a superalgebra.
 The converse, however, is not true.

 The paper \cite{BBS} studies Lie superisomorphisms between prime associative superalgebras. We will present only the result for the simple ones (and for Lie superautomorphisms), since it is definite and easy to state.
 
\begin{theorem} \label{Tsimple}
Let $A$ be a central simple associative superalgebra over the field $F$. If  $\dim_F A \ne 2,4$, then every Lie superautomorphism $\alpha$ of $A$ is of the form $\alpha= \varphi  + \tau$ where $\varphi$ is either a superautomorphism   or the negative of a superantiautomorphism of  $A$ and $\tau:A\to F$ is a linear map which vanishes on supercommutators.
\end{theorem}

The necessity of the assumption $\dim_F A \ne 2,4$ is explained in \cite[Examples 3.5 and 5.2]{BBS}.

The proof is divided into three parts. The one where FIs are applicable considers the situation where the automorphism $\sigma$
 inducing the 
 $\ZZ_2$-grading is outer. The results on FIs with automorphisms from \cite{BBC} then turn out to be applicable.

\subsection{Lie-admissible algebras}

\def\slb{{[\hskip-2pt[}}
\def\srb{{]\hskip-2pt]}}
\def\sb#1{\slb#1\srb}

Let $A$ be a nonassociative  algebra over a field $F$. Unlike so far, we will denote multiplication in $A$ by 
 $*$. We say
that $A$ is a {\em Lie-admissible algebra} if the vector space of $A$ is a Lie algebra under the product
$$\sb{x,y}=x*y-y*x.$$
The usual problem is to describe $*$ under some additional assumptions. This goes back to Albert
\cite{Al} 
who was particularly interested in Lie-admissible algebras that are also {\em flexible}, that is, they satisfy
 $$(x*y)*x=x*(y*x)$$ for all $x,y\in
A$.

By the Poincar\' e-Birkhoff-Witt Theorem, every Lie algebra is isomorphic to a Lie subalgebra of an associative algebra.
Let us therefore assume that the Lie-admissible algebra $A$ is a Lie subalgebra of an associative algebra $Q$ such that
\begin{equation}
    \label{err}
\sb{x,y}=x*y-y*x=xy-yx\end{equation}
for all $x,y\in A$, where $xy$ is the product of $x$ and $y$
   in $Q$. Rather than  assuming that
   $A$ is flexible, we only assume that
$A$ is
{\em third power-associative}.
This means that $$(x*x)*x=x*(x*x)$$ for all $x\in
A$. Substituting 
$x*x$ for $y$ in \eqref{err} it thus follows that
$$ x(x\ast x) = (x\ast x)x$$
for all $x\in A$. This means that
$x\mapsto x * x$ is a commuting trace of a biadditive function, so the methods of FIs are applicable.

The idea that we just described 
was used in the papers by Beidar, Chebotar et al. \cite{BClie, BCFK}. Their results are also surveyed in \cite{FIbook}.
Let us state a simplified version of \cite[Theorem 8.2]{FIbook}.

\begin{theorem}\label{bogen} Let $F$ be 
 a field with {\rm char}$(F)\ne 2$ and let
 $Q$ be
a unital   $F$-algebra whose center is $F$. Let $A$ be a linear
subspace of $Q$ which is a $3$-free subset of $Q$. 
Suppose that $A$ is endowed with another (nonassociative) multiplication $*$
such that $$x*y-y*x=xy-yx$$ for all $x,y\in A$.
Then the algebra
 $(A,\,+\,,\,*\,)$ is third power-associative
if and only if there exist a $\lambda\in F$, a
symmetric linear map $\mu:A\to F$, and a symmetric
bilinear map $\nu:A^2\to F$ such that
\begin{equation*}
    \label{LAA1}
x*y=\frac12 [x,y]+\lambda x\circ y+\mu(x)y+\mu(y)x+\nu(x,y)
\end{equation*}
for all $x,y\in A$.
Moreover,  if $(A,\,+\,,\,*\,)$ is flexible, then
$\mu([x,y])=0=\nu(x,[x,y])$ 
for all $x,y\in A$.
\end{theorem}
 
 The special case of Theorem \ref{bogen}  where $A=Q=M_n(F)$
 was obtained earlier by Benkart and Osborn \cite{BO}. 
 \subsection{Jordan maps}
 The definition of a  Jordan homomorphism is analogous to that of a Lie homomorphism: 
if $M$ and $R$ are rings and $J$ is a Jordan subring
of $M$, then a {\em Jordan homomorphism} is an  additive map $\alpha:J\to R$  such that
\begin{equation}
    (x\circ y)^\alpha=x^\alpha\circ y^\alpha \label{ejorh}
\end{equation}
for all $x,y\in J$. Observe that if
$R$ is $2$-torsion free (i.e., $2x=0$ implies $x=0$), then \eqref{ejorh} is equivalent to $$(x^2)^\alpha =(x^\alpha)^2 $$ for all $x\in J$.

As for Lie homomorphisms, we can ask whether Jordan homomorphisms can be described by homomorphisms and antihomomorphisms. This question can also be approached by FIs. The results that were obtained are similar to those on 
 Lie homomorphisms.  However, 
 they do not bring essentially new information concerning the classical situation of simple and prime rings. This is simply because the answers to the most natural questions were obtained already before the introduction of  FIs. In 1956, Herstein  proved that, under mild characteristic assumptions, a Jordan homomorphisms from a ring onto a prime ring is either a homomorphism or an antihomomorphism \cite{Her0}. More than twenty years later,  Jordan homomorphisms on some proper Jordan subrings, such as the set of symmetric elements in a ring with involution, were described \cite{Mart9, McC} by using powerful methods discovered by Zelmanov \cite{Zel}.
 
 Still, the FI approach to Jordan homomorphisms is interesting since it provides  alternative proofs of known results and yields abstract theorems involving  $d$-free subsets  that may be applicable  in  different situations. Let us present two such theorems, both taken from \cite{FIbook}. 
 The first one is analogous to Theorem \ref{wecan}.

 \begin{theorem} \label{6TJ1}
 Let $\alpha$ be a Jordan homomorphism from a ring $M$ to a unital ring $Q$ which admits the operator $\frac{1}{2}$. If  $M^\alpha$ is a $4$-free  subset of $Q$, then $\alpha$ is the direct sum of a homomorphism and an  antihomomorphism.
\end{theorem}

Let us present the idea of proof. Since, by the very definition, we control the action of $\alpha$
on the Jordan product, it is enough to 
determine the action of $\alpha$ on the Lie product.
It is therefore natural to use the identity
$$[[u,v],t] = u\circ (v\circ t) - v\circ (u\circ t)$$
which connects the Jordan and Lie products. This identity implies that $\alpha$ satisfies
$$[[u,v],t]^\alpha = 
[[u^\alpha,v^\alpha],t^\alpha]$$
for all $u,v,t\in M$.
Hence, for all $x,y,z,w\in M$ we have, on the one hand,
$$
[[x,y],[z,w]]^\alpha = [[x^\alpha,y^\alpha],[z,w]^\alpha],
$$
 and on the other hand,
 $$
 [[x,y],[z,w]]^\alpha = [[x,y]^\alpha,[z^\alpha,w^\alpha]].
 $$
 Comparing we see that
 $F(x,y) = [x,y]^\alpha$ satisfies
\begin{equation} \label{6fijor}
 [[x^\alpha,y^\alpha], F(z,w)] = [F(x,y),[z^\alpha,w^\alpha]].
\end{equation}
This FI  can be solved  since $M^\alpha $ is assumed to be a  $4$-free subset of $Q$. In this way, one determines the action of $ \alpha$ on the Lie product.

The second result is analogous to Theorem \ref{thk}.

\begin{theorem}\label{64MT}
Let $S$ be the set of symmetric elements of a ring $M$ with involution,  let $Q$ be a unital ring with center $C$,  and let $\alpha:S\to Q$ be a Jordan homomorphism. Suppose that $S$ and $Q$ admit the operator $\frac{1}{2}$ and   that 
 $C$ is a direct summand of the additive group $Q$. If $ S^\alpha$ is a $7$-free  subset of $Q$, then $\alpha$ can be extended to  a homomorphism from  $\langle S\rangle $, the subring generated by $S$, to $Q$.
\end{theorem}

The proof is based on defining the map
$\beta: S + [S,S]\to \overline{Q}=Q/C$ by
$$
\Bigl(s + \sum_i [s_i,t_i]\Bigr)^\beta = \ov{ s^\alpha +  \sum_i [s_i^\alpha,t_i^\alpha]}.
$$
It can be shown that $\beta $ is a well-defined Lie homomorphism and that 
$S+[S,S]$ is a Lie ideal of the ring $\langle S\rangle $. One can therefore apply  the description of Lie homomorphisms on Lie ideals of rings (the result that is needed was not explicitly stated in Subsection \ref{ss41}, but enough information was given to get the idea).

A {\em Jordan derivation} $\delta:J\to R$  is of course
an additive map satisfying
\begin{equation}
    \label{defjd}
(x\circ y)^\delta = x^\delta\circ y + x\circ y^\delta \end{equation}
for all $x,y\in J$. If $R$ is $2$-torsion free, then \eqref{defjd} is equivalent to
$$(x^2)^\delta= x^\delta x+ xx^\delta $$
for all $x\in J$. 
The FI methods can be applied to Jordan derivations and some related maps  in a similar fashion as to  
Jordan homomorphisms, see \cite{FIbook, Leejd, LeeLin} and references therein. We will therefore skip this topic and move to another type of seemingly similar maps, which, however,
must be treated in 
a different way.

Let $R$ be a ring with involution $\ast$. An additive map $\delta:R\to R$ is called a 
{\em Jordan $*$-derivation} if
$$(x^2)^\delta= x^\delta x^*+ xx^\delta $$
for all $x\in R$. For every $a\in R$,
\begin{equation}
    \label{ejdst}
    x\mapsto ax^*- xa
\end{equation}
is an example of a Jordan
$*$-derivation.
 The question whether every Jordan $*$-derivation of $R$ is of such a form is closely connected with the problem of describing quadratic functionals \cite{Sem1, Sem2}. This question was considered in several papers, but the deepest results were obtained by
 Lee, Wong, and Zhou \cite{LWZ, LZ}. They can be summarized as follows.
 
\begin{theorem}\label{jorsa}
Let R be a noncommutative prime ring with involution $*$. Then
every Jordan $*$-derivation $\delta$ of R is of the form
$x^\delta = ax^*- xa$ for some 
$a\in Q_{ms}(R)$, unless {\rm char}$(R)=2$ and $\deg(R)=2$.
\end{theorem}

The first step of the proof is  showing that Jordan $*$-derivations  satisfy certain (non-obvious!) FIs with involution. After that,  the general theory of such FIs is applied. One such identity is
\begin{align*}
&(xw)^\delta z^*y^* + (yz)^\delta w^*x^* +
xw (yz)^\delta + yz(xw)^\delta\\  &-
x^\delta z^* y^* w^*-(wyz)^\delta x^* - x(wyz)^\delta - wyzx^\delta\\
= & (zxw)^\delta y^*+y^\delta w^*x^*z^*
+zxwy^\delta + y(zxw)^\delta\\& -(zx)^\delta y^*w^* - (wy)^\delta x^* z^* - zx (wy)^\delta - wy (zx)^\delta.
\end{align*}
This is the key identity for handling the case where char$(R)\ne 2$. The char$(R)=2$ case is based on some other FI.

An example  showing that the theorem does not hold if  {\rm char}$(R)=2$ and $\deg(R)=2$ was also constructed.

 \subsection{$f$-homomorphisms}
 Let  $f=f(X_1,\dots,X_m)\in\ZZ\langle X_1,X_2,\dots\rangle$ be a multilinear polynomial of degree $m$, that is, a polynomial of the form 
  $$f = \sum_{\sigma \in S_m} \lambda_\sigma
  X_{\sigma(1)}  X_{\sigma(2)}\cdots X_{\sigma(m)} $$
 for some integers $ \lambda_\sigma$, not all zero.
 An additive map $\alpha$ from a ring $M$ to a ring $Q$ is called an {\em $f$-homomorphism} if
 $$
f(x_1,x_2,\ldots,x_m)^\alpha = f(x_1^{\alpha},x_2^{\alpha},\ldots,x_m^\alpha)$$
for all $x_1,\ldots,x_m\in M$.
If $f=X_1X_2-X_2X_1$ then an $f$-homomorphism is of course a Lie homomorphism, and if 
$f=X_1X_2+X_2X_1$ then an $f$-homomorphism is a Jordan homomorphism. 

Describing the form of an arbitrary $f$-homomorphism may seem a very ambitious goal. However, it has turned out to be  doable. This was first observed by Beidar and Fong \cite{BFon}, and then  also in 
\cite{Onher1, BBCMp}. We will follow the exposition from \cite{FIbook}.

The main point of the proof is to
find an FI involving an $f$-homomorphism. This is done as follows.  Since
$f$ is multilinear, we have
\begin{equation*} \label{65e0}
 [f(\ov{x}_m),y]
 = \sum_{i=1}^m f(x_1,\ldots,x_{i-1},[x_i,y],x_{i+1},\ldots,x_m)
 \end{equation*}
 for all $x_i,y\in M$.
In particular,
\begin{eqnarray*}
 [f(\overline{x}_m),f(\overline{y}_m)]
 = \sum_{i=1}^m f(x_1,\ldots,x_{i-1},[x_i,f(\overline{y}_m)], x_{i+1},\ldots,x_m)
\end{eqnarray*}
for all $x_i,y_j\in M$.
Since
$$
[x_i,f(\overline{y}_m)] =  \sum_{j=1}^m f(y_1,\ldots,y_{j-1},[x_i,y_j],y_{j+1},\ldots,y_m),
$$
this yields 
\begin{align*}
\label{65e1} &[f(\overline{x}_m),f(\overline{y}_m)]\\\nonumber
=& \sum_{i=1}^m\sum_{j=1}^m f(x_1,\ldots,x_{i-1},f(y_1,\ldots,y_{j-1},[x_i,y_j],y_{j+1},\ldots,y_m), x_{i+1},\ldots,x_m). 
\end{align*}
Now,
 since $$[f(\overline{x}_m),f(\overline{y}_m)] = - [f(\overline{y}_m),f(\overline{x}_m)],$$
 we can, after changing the sign, substitute $y_i$ for $x_i$  on the right-hand side of the above identity.
 Hence,
\begin{align*}
&\sum_{i=1}^m\sum_{j=1}^m f(x_1,\ldots,x_{i-1},f(y_1,\ldots,y_{j-1},[x_i,y_j],y_{j+1},\ldots,y_m), x_{i+1},\ldots,x_m)\\ 
 +& 
\sum_{i=1}^m\sum_{j=1}^m f(y_1,\ldots,y_{i-1},f(x_1,\ldots,x_{j-1},[y_i,x_j],x_{j+1},\ldots,x_m), y_{i+1},\ldots,y_m)\\ =&0.
\end{align*} 
Applying an $f$-homomorphism $\alpha$ to this identity we see that 
 $$F(x,y)= [x,y]^\alpha$$ satisfies \begin{align*}
&\sum_{i=1}^m\sum_{j=1}^m f(x_1^\alpha,\ldots,x_{i-1}^\alpha,f(y_1^\alpha,\ldots,y_{j-1}^\alpha,F(x_i,y_j),y_{j+1}^\alpha,\ldots,y_m^\alpha), x_{i+1}^\alpha,\ldots,x_m^\alpha)\\ 
 +& 
\sum_{i=1}^m\sum_{j=1}^m f(y_1^\alpha,\ldots,y_{i-1}^\alpha,f(x_1^\alpha,\ldots,x_{j-1}^\alpha,F(y_i,x_j),x_{j+1}^\alpha,\ldots,x_m^\alpha), y_{i+1}^\alpha,\ldots,y_m^\alpha)\\ = &0.
\end{align*} 
This is an FI   to 
which Theorem \ref{thcorquasi} is  applicable. Assuming that
$M^\alpha$ is an $(2m)$-free subset of $Q$, it follows that the middle function $F(x,y)$ is a quasi-polynomial. This makes the problem of describing $\alpha$ just slightly more difficult than the problem of describing a Lie homomorphism. Using a similar  approach as in the proof of Theorem \ref{wecan} one then proves the following theorem.

\begin{theorem}\label{65maint} Let
$f\in\ZZ\langle X_1,X_2,\dots\rangle$ be a multilinear polynomial of degree $m$ such that one of its coefficients is $1$. Let $\alpha$ be an $f$-homomorphism from a ring $M$ to a unital ring $Q$ such that  its center  $C$ is a field.  If $M^\alpha$ is a $(2m)$-free subset of $Q$, then
 $x^\alpha = \lambda x^\varphi + x^\mu$ for all $x\in M$, where $\lambda \in C$, $\varphi: M\to Q$ is a homomorphism or an antihomomorphism, and $\mu:M\to C$ is an additive map.  
\end{theorem}

A special case of particular interest is when $R=M^\alpha$ is a prime ring with $\deg(R)\ge 2m$. Note that some degree assumption is necessary to exclude polynomial identities.

One  similarly defines $f$-derivations. Theorem \ref{thomder} can be used to reduce the problem of their description to that of describing $f$-homomorphisms.
 
 \subsection{Near-derivations}
 
 Let $L$ be a Lie algebra.  By Der$(L)$ we denote the Lie algebra of all derivations of $L$.
For any $x\in L$, let $\Ad x\in {\rm Der}(L)$ denote the inner derivation $(\Ad x)(y) = [x,y]$.
A linear map $\delta:L\to L$ is called  a {\em near-derivation}   if  there exists 
a linear map $\gamma:L\to L$ such that 
\begin{equation}\label{eqnear}(\Ad x) \delta - \gamma(\Ad x)\in {\rm Der}(L)\end{equation} for every $x\in L$. There are two basic 
examples, both occurring for 
 $\gamma = \delta$. These are 
 derivations (which satisfy
 $(\Ad x) \delta - \delta(\Ad x) = - \Ad x^\delta $) and elements from the {\em centroid} of $L$
  (which satisfy
 $(\Ad x) \delta - \delta(\Ad x) = 0 $).
 Besides, every map from $L$ to the center of $L$ is also a near-derivation (for $\gamma=0$).

Near-derivations were introduced and studied
by the present author in \cite{Bnear}. The main motivation was the paper by Leger and Luks \cite{LegL} in which similar  maps, called
  {\em generalized derivations} by the authors,
  were treated. They are defined through the condition
  $$(\Ad x) \delta - \gamma(\Ad x) =  \Ad x^\sigma$$
  for all $x\in L$,
  where $\sigma:L\to L$ is another linear map. Of course,
  $\Ad x^\sigma\in {\rm Der}(L)$, so 
  the notion of a near-derivation is
  slightly more general than
  the notion of a 
  generalized derivation. This, however, was not the main reason for its introduction. The purpose of \cite{Bnear} was to show that results of the same type as those from
  \cite{LegL} can be obtained as a consequence  of the FI theory.
  
  How to obtain an FI involving a near-derivation $\delta$? Observe that \eqref{eqnear}
  means that
  $$
[x,[y,z]^\delta] - [x,[y,z]]^\gamma
 = [[x,y^\delta] - [x,y]^\gamma, z] + [y, [x,z^\delta] - [x,z]^\gamma]$$
 for all $x,y,z\in L$.
  Using this along with the Jacobi identity which tells us that
  $$[x,[y,z]]^\gamma +[z,[x,y]]^\gamma + [y,[z,x]]^\gamma =0,$$
  one quickly derives that the function
  $F:L^2\to L$,
  $$F(x,y)= [x,y]^{2\gamma - \delta} - [x^\delta,y] - [x,y^\delta], $$
  satisfies 
  $$ [F(x,y),z] + [F(z,x),y] + [F(y,z),x] =0
  $$
  for all $x,y,z\in L$. This brings us to a position where Theorem \ref{thcorquasi} can be used. Under suitable assumptions, $F(x,y)$ is thus a quasi-polynomial. This  is the starting point in the proof of the following theorem.

  \begin{theorem} \label{tale}
Let $L$ be a Lie subalgebra of a unital $F$-algebra $Q$ with center $C$, and let 
$\delta:L\to L$ be a near-derivation. If
$L$ is a $4$-free subset of $Q$, then there exists a $\lambda\in C$ such that
$$[x,y]^\delta  - 
[x^\delta,y] - [x,y^\delta] - \lambda [x,y]\in C$$
for all $x,y\in L$. Moreover,  if the second cohomology group $H^2 (L, F)$ is trivial, then
there exist a derivation $d:L\to Q$, an element $\gamma\in C$,  and a
linear map $\tau: L \to C$ such that $\delta= d+ \ell_\gamma  +\tau$ (where $\ell_\gamma$ denotes the function 
$x\mapsto \gamma x$). 
\end{theorem}
 
In \cite{Bnear},
  Theorem \ref{tale} plays the role of the  fundamental lemma, while the main results consider different situations in which
  it can be used to show that 
  $\delta$ is the sum of a derivation, an element from the centroid, and a central map. We will not go into  this here since our main purpose  is only to demonstrate the ways to apply FIs.

  \subsection{Linear preserver problems} The title refers to a variety of problems that concern linear maps between algebras that preserve some algebraic properties. The goal is to describe their form. 
 This is a vast research area, popular especially in operator theory and linear algebra, but also in pure algebra. 
 
 Of course, FIs are not applicable to all linear preserver problems. To those to which they are, however,
 they yield considerably more general results from those previously known. This is because of their formal
 nature which allows a unified treatment of very different objects.
 
 The most prominent example of a linear preserver problem  that can be solved by FIs  concerns {\em commutativity preserving} linear maps. We say that
 a linear map $\alpha:B\to A$, where  $A$ and $B$ are algebras over a field $F$, preserves commutativity if, for all $x,y\in B$,
 $$ [x,y] = 0\implies [x^\alpha,y^\alpha]=0.$$
 Such maps were first studied in algebras of matrices over a field \cite{W}, and after that in many other more general algebras. See \cite[pp.\ 218-219]{FIbook} for an overview of the early history.

 Lie homomorphisms obviously preserve commutativity, so the problem we are facing now is more difficult. However, we can take the same approach as presented at the beginning of Subsection 
 \ref{ss41}. We simply use the obvious fact that $x$ commutes with $x^2$ to obtain the FI \begin{equation*}
     \label{ecm}
[x^\alpha, (x^2)^\alpha] =0 \end{equation*} for all $x\in B$. As we know from Subsection 
 \ref{ss41}, under suitable assumptions this identity can be used to describe the action of $\alpha$ on squares of elements. The problem that remains is indeed more difficult than for Lie homomorphisms, but the breakthrough has been made.
 
 This idea was used already in the seminal paper \cite{B5} on applications of FIs. The following theorem is just a slight technical generalization
 of the result therein, which was noticed a bit later.  Before stating the theorem, we recall
 that a   prime unital $F$-algebra is said to be {\em centrally closed} if its extended centroid is equal to $F$. For example, a  simple unital ring is a centrally closed algebra over its center.
 
 \begin{theorem}\label{tcommp}
Let $A$ and $B$ be   centrally closed prime unital algebras over a field $F$ with {\rm char}$(F)\ne 2$, and let $\alpha:B\to A$ be a bijective linear map. If $\deg(B)\ge 3$ and
$\alpha $ satisfies $[x^\alpha, (x^2)^\alpha] =0$ for all $x\in B$  (in particular, if $\alpha$ 
 preserves commutativity), then $\alpha$ is of the form $x^\alpha = \lambda x^\varphi + \mu(x)$ for all $x\in B$, where $\lambda\in F$, $\varphi$ is an isomorphism or an antiisomorphism from $B$ onto $A$, and $\mu:B\to F$ is a linear map.
 \end{theorem}
 
 We remark that the assumption that $\deg(B)\ge 3$ is necessary. Indeed, if $A=B=M_2(F)$ (and so $\deg(B)=2$), then $x,y\in B$ commute if and only if $x$, $y$ and $1$ are linearly dependent. Therefore, every linear map that sends $1$ to a scalar multiple of $1$ preserves commutativity.
 
 Before publication of \cite{B5}, commutativity preservers were studied  in matrix algebras 
and algebras occurring in functional analysis. Theorem \ref{tcommp} unified and generalized many  of the existing results, and opened the doors to the consideration of commutativity and some related linear preservers in pure algebra. Using more advanced techniques of the FI theory, its various generalizations were extensively studied (see \cite{BanM, BBCF, BeiLin,BE, Bcomp, BQuar, BMiers, BrSe2, YFL}  and also \cite{FIbook}). However, we will not discuss them here.

Assume now that $A$ and $B$ are $F$-algebras with involution.  We say that a map $\alpha$ {\em preserves normality} if 
$x^\alpha$ is normal whenever $x$ is normal, that is, \begin{equation}
    \label{exnor}
[x,x^*]=0 \implies [x^\alpha,(x^\alpha)^*]=0.
\end{equation}
The problem of describing normality preserving linear maps is obviously related to but more complicated than the problem of describing commutativity preserving linear maps. It has also gained some 
 interest through  the years, and is another example of a linear preserver problem that can be handled by FIs. This was shown in \cite{BBCF} and is also surveyed in \cite[Section 7.2]{FIbook}. It is interesting to note that these algebraic results were later used for solving some seemingly unrelated preserver problems from functional analysis  \cite{BMN, Ham, Molnar}.
 
 We will state and comment on a sample theorem from \cite{BBCF}.  Let us start with some general remarks.

 Like in Theorem \ref{tcommp}, we will assume that
 $A$ and $B$ are centrally closed prime unital algebras over $F$. An involution $\ast$ on such an algebra is said to be of the {\em first kind} if it is $F$-linear. 
 Otherwise,  it is said to be of the {\em second kind}. It turns that if $A$ and $B$ have involutions of the second kind, the problem of normality preserving maps can be reduced to that of commutativity preserving maps. We will therefore only consider involutions of the first kind. 
 We also add a natural assumption  that our normality preserving linear map $\alpha:B\to A$ is $\ast$-linear, meaning that 
 \begin{equation}\label{ealstar}
     (x^*)^\alpha = (x^\alpha)^*
      \end{equation}
      for every $x\in B$
 (in the case of involution of the second kind this is unnecessary since  a slightly weaker version of  \eqref{ealstar}  is automatically fulfilled). If 
  char$(F)\ne 2$,
 then this assumption implies that \eqref{exnor} is equivalent to the condition that for every symmetric element $s$ in $B$ and every skew element $k$ in $B$,
\begin{equation}
    \label{esk}
[s,k]=0 \implies [s^\alpha, k^\alpha]=0. \end{equation}
 This is because every 
 $x\in B$ can be written as the sum of the symmetric element
 $s=\frac{1}{2}(x+x^*)$ and the skew element $k= \frac{1}{2}(x-x^*)$,   so $[x,x^*]=0$ is equivalent to 
 $[s,k]=0$, and, similarly,  $[x^\alpha,(x^*)^\alpha]=[x^\alpha,(x^\alpha)^*]=0$ is equivalent to
 $[s^\alpha, k^\alpha]=0$.
 Observe that $k^2$
 is symmetric and $k^3$ is skew if $k$ is skew. Therefore,
 \eqref{esk} implies
 that
 \begin{equation}
     \label{ekkk}
[(k^2)^\alpha, k^\alpha] =0\quad\mbox{and}\quad [(k^2)^\alpha, (k^3)^\alpha]=0  \end{equation}
 for every skew element $k$ in $B$. The first identity is a standard FI which we know how to handle. Using the result derived from this first FI, the second identity in \eqref{ekkk}  becomes an FI that fits into the general theory. This approach yields the following theorem.
 
 \begin{theorem}
 Let $A$ and $B$
be centrally closed prime unital algebras over a field $F$
with involution of the first kind. Denote by $A_0$ (resp.\ $B_0$) the subalgebra of $A$ (resp.\ $B$) generated by all skew elements in $A$ (resp.\ $B$). Let $\alpha:B\to A$ be  a bijective $*$-linear map that preserves normality. If {\rm char}$(F)\ne 2,3$, $\deg(B)\ge 7$, and $\deg(A)\ge 14$, then $\alpha$ is of the form 
$x^\alpha= (\lambda_1x + \lambda_2 x^*)^\varphi + \mu(x)$ for all $x\in B_0$, where
 $\lambda_1,\lambda_2\in F$, $\lambda_1\ne \pm\lambda_2$, 
$\varphi:B_0\to A_0$ is  a $*$-linear isomorphism, and $\mu:B_0\to F$ is a linear map that vanishes on skew elements from $B_0$.
 \end{theorem}
 
If $A$ and $B$ are simple algebras, then $A_0=A$ and $B_0=B$ \cite{Her2}. An example in which $A=B$ is the free algebra on two indeterminates shows that in general $A_0$ and $B_0$ may be proper subalgebras.
 
 After solving the problem of commutativity preserving  maps, one may wonder what can be said about  maps that {\em preserve anticomutativity}, that 
 is, about  maps $\alpha$ between algebras (or rings) that 
 satisfy
 $$ x\circ y = 0\implies x^\alpha\circ y^\alpha=0.$$
 In other words, these maps preserve zero Jordan products.

 It is not unusual that  a ring has no pairs of nonzero anticommuting elements. For example, this is true for the Weyl algebra \cite[Example 3.18]{zpdbook}, which is a basic example of a simple ring. One therefore cannot expect that results similar to 
 the preceding two theorems can proved for anticommutativity preservers. We have to consider other classes of rings to obtain some interesting result. 
 
 An example of a ring having many pairs of anticommuting elements is the  ring of square matrices over any ring.
 The following result was proved in \cite{CKLZ} (see also \cite[Section 7.3]{FIbook}).
 
 \begin{theorem}
 \label{tcklz} Let $S$ be a unital ring admitting  the operator $\frac{1}{2}$, let $R=M_n(S)$ with $n\ge 4$, and let $\alpha:R\to R$ be a surjective additive map that preserves anticommutativity. Then there exists an element $\lambda$ from the center of $R$ and a Jordan homomorphism $\varphi:R\to R$ such that $x^\alpha = \lambda x^\varphi$ for all $x\in R$.
 \end{theorem}
 
 The first step in the proof is showing that
 $\alpha$ satisfies 
 \begin{equation}\label{eowto}(xyx)^\alpha \circ y^\alpha = (yxy)^\alpha\circ x^\alpha\end{equation}
 for all $x,y\in R$. Linearizing in $x$ and $y$, we obtain an FI in four variables to which Theorem \ref{matfree} is applicable (since $n\ge 4$). From that point on one proceeds in a standard way. However, how to derive \eqref{eowto}? The answer is connected with the next, and  final, topic.

  \subsection{Zero product determined algebras}

An $F$-algebra $A$ is said to be  { \em zero product determined} (zpd for short)  if every bilinear  functional $\varphi:A^2\to F$  with the property that  $ \varphi(x,y)=0$ whenever $xy=0$
is of the form $\varphi(x,y)=\tau(xy)$ for some linear functional $\tau$. The theory of zpd algebras has been developing, first sporadically and later systematically, over the last 15 years. The main reason for this development  
was a variety of applications to different mathematical areas. The theory, together with its applications, is surveyed in the recent book \cite{zpdbook}.

The theory  has two branches, algebraic and analytic. The analytic branch deals with Banach algebras and in the definition we require that the functionals $\varphi$ and $\tau$ are continuous. This branch  has turned out to be
 richer, especially since the class of zpd Banach algebras is really large, while the class of ordinary algebras that are zpd is  narrower.

In the above definition of a zpd algebra, $A$ is not necessarily associative. Still, the case where $A$ is associative is the most studied one. The other two important cases are when $A$ is an associative algebra considered  either as a Lie algebra under the Lie product $[\,\cdot\,,\,\cdot\,]$ or as a Jordan algebra  under the Jordan product $\circ$. In the first case, we talk about the {\em zero Lie product determined} algebra (zLpd for short), and in second case, we talk about the {\em zero Jordan product determined} algebra (zJpd for short). 

What are basic examples of zpd, zLpd and zJpd (associative) algebras? If $A$ is unital and is, as an algebra, generated by idempotents,
then it is zpd \cite[Theorem 2.15]{zpdbook} and, if char$(F)\ne 2$, also zJpd \cite[Theorem 3.15]{zpdbook}. A simple example of a unital algebra that is generated by idempotents is the matrix algebra $A=M_n(S)$, where $n\ge 2$ and $S$ is any unital algebra \cite[Corollary 2.4]{zpdbook}. Such an algebra is therefore zpd and, if char$(F)\ne 2$, zJpd, but not necessarily zLpd \cite[Example 3.13]{zpdbook}. However, if $S$ is commutative, then $A$ is also zLpd \cite[Corollary 3.11]{zpdbook}.

There are many interactions between the FI theory  and the zpd theory, both philosophical and technical.
 Concerning the former, we mention that they both were originally influenced by the problem of commutativity preserving linear maps.
 As we mentioned in the preceding subsection, the description of bijective commutativity preservers appeared in \cite{B5} which is the first paper providing applications of FIs. The description  of not 
 necessarily bijective 
commutativity preservers
on finite-dimensional central simple algebras 
was obtained in \cite{BS06}, which is the first paper in which a version of the zpd condition was considered. The first and crucial step of proof was  showing that the matrix algebra $M_n(F)$ is, using the present terminology, zLpd. 
 
 In some papers, notably in \cite{BQuar} and \cite{Breproc}, more direct connections between FIs and zpd algebras are examined. However, they deal with Banach algebras, so we will not discuss them in this algebraic paper.
 Let us present only one, very simple (algebraic) result which can serve as 
 a model of how FIs and zpd algebras can be combined. It concerns {\em zero product preserving} linear maps, i.e.,
 linear maps 
 $\alpha$ between algebras  that 
 satisfy
 $$ x y = 0\implies x^\alpha y^\alpha=0.$$
 The study of such maps has a long history, especially in functional analysis (see  \cite[p.\ 114]{zpdbook} for more details and references).

 \begin{theorem}\label{thth}
 Let $A$ be a
 $3$-free algebra and let $B$ be a
 zpd algebra. If $\alpha:B\to A$ is a surjective zero product preserving linear map, then 
 there exists 
 an element $\lambda$
 from the center of $A$ such that 
 $(xy)^\alpha =\lambda x^\alpha y^\alpha$ for all $x,y\in B$ (and hence $x\mapsto \lambda x^\alpha$ is a homomorphism).
 \end{theorem}
 
 \begin{proof}
 Define $\Phi:B\times B\to A$ by
 $\Phi(x,y)=x^\alpha y^\alpha$.
 Since $B$ is zpd and $xy=0$ implies
 $\Phi(x,y)=0$, 
 there exists a linear map $T:B\to A$ such that
 \begin{equation}\label{seses}\Phi(x,y)=T(xy).\end{equation}
 Actually, the definition only tells us that this holds if $A=F$.
However, it is easy to see that the same is true if $A$ is any vector space \cite[Proposition 1.3]{zpdbook}. 

Using the associative law, we see 
 that \eqref{seses} implies that 
$\Phi(xy,z) = \Phi(x,yz)$
for all $x,y,z\in B$, that is,
\begin{equation}
  \label{sube6} (xy)^\alpha z^\alpha = x^\alpha (yz)^\alpha.\end{equation}
Since $\alpha$ is surjective and $A$ is $3$-free,  Theorem \ref{thcorquasi} tells us that  $(xy)^\alpha$ is a quasi-polynomial. Thus, denoting by $C$ the center of $A$,
there exist  $\lambda,\lambda'\in C$, $\mu,\mu':B\to C$, and $\nu:B^2\to C$ such that
\begin{equation}
  \label{sube22}  
(xy)^\alpha=\lambda x^\alpha y^\alpha +
\lambda' y^\alpha x^\alpha+\mu(y)x^\alpha + \mu'(x)y^\alpha +\nu(x,y) \end{equation}
 for all $x,y\in B$. Using \eqref{sube22} in \eqref{sube6}, we obtain
 \begin{align*}
     \lambda' y^\alpha x^\alpha  z^\alpha
     - \lambda' x^\alpha & z^\alpha  y^\alpha
     +(\mu-\mu')(y)x^\alpha z^\alpha\\+& \mu'(x)y^\alpha z^\alpha 
     - \mu(z)x^\alpha y^\alpha 
     +\nu(x,y)z^\alpha-\nu(y,z)x^\alpha =0
 \end{align*}
 for all $x,y,z\in B$. Lemma \ref{lfreq} shows that $\lambda'=\mu = \mu'=\nu =0$.  Therefore,  \eqref{sube22} reduces to $  
(xy)^\alpha=\lambda x^\alpha y^\alpha$.
 \end{proof}
 
 Theorem \ref{thth} probably has 
 not yet appeared in the literature. 
 However, it does 
 not bring anything essentially new, its idea can be in fact  traced back to the paper \cite{CKL} which was published before the introduction of zpd algebras. Our main purpose was to present the method of proof rather than the theorem as such which is much simpler than many other results on  zero product preservers.
 
 One can similarly consider zero Jordan product (i.e., anticommutativity) preserving linear maps $\alpha:B\to A$. If 
 $B$ is zJpd, then
 it follows that
 $$x^\alpha \circ y^\alpha = T(x\circ y)$$ for some linear map $T:B\to A$. 
 Since   $(xyx) \circ y = (yxy)\circ x$, \eqref{eowto} follows. This explains the concept behind the proof of Theorem \ref{tcklz}.
 
 \smallskip 

 {\bf Concluding remark.} As we saw through the numerous examples presented in this last section, there are many problems in different areas that give rise to certain FIs, and can consequently be solved by using the tools of the general theory. Once we derive some FIs, it is often just a routine matter to arrive at a solution. The main and everlasting problem of the FI theory is to discover FIs hidden behind some mathematical problems. In some sense, this problem is to find problems to which the theory is applicable. The author   is certain that  there are still many such problems occurring  throughout mathematics, and that the potential of the FI theory  is  far from  being exhausted.
 
 \smallskip
 {\bf Acknowledgment}.
 The author would like to thank the referee for useful comments.

\end{document}